\documentclass[10pt,reqno,a4paper]{amsart}

\usepackage[hmargin=2.5cm,bmargin=2.5cm,tmargin=3cm]{geometry}

\usepackage{amsfonts,amssymb,mathrsfs}
\usepackage{rotating,pgf,tikz}

\usepackage{empheq}
\usepackage{hyperref}

\title[On the eigenvalues of the Robin Laplacian]{On the eigenvalues of the Robin Laplacian with a complex parameter}

\author{Sabine B\"ogli}
\author{James B.~Kennedy}
\author{Robin Lang}

\address{Sabine B\"ogli, Department of Mathematical Sciences, Durham University, Lower Mountjoy, Stockton Road, Durham DH1 3LE, United Kingdom}
\email{sabine.boegli@durham.ac.uk}

\address{James B.~Kennedy, Grupo de F{\'i}sica Matem\'atica, Faculdade de Ci\^encias, Universidade de Lisboa, Campo Grande, Edif{\'i}cio C6, P-1749-016 Lisboa, Portugal}
\email{jbkennedy@fc.ul.pt}

\address{Robin Lang, Institut f\"ur Analysis, Dynamik und Modellierung, Universit\"at Stuttgart, Pfaffenwaldring 57, D-70569 Stuttgart, Germany}
\email{robin.lang@mathematik.uni-stuttgart.de}

\newtheorem{theorem}{Theorem}[section]
\newtheorem{lemma}[theorem]{Lemma}
\newtheorem{proposition}[theorem]{Proposition}
\newtheorem{corollary}[theorem]{Corollary}

\newtheorem{conjecture}[theorem]{Conjecture}

\theoremstyle{definition}
\newtheorem{definition}[theorem]{Definition}

\newtheorem{assumption}[theorem]{Assumption}

\newtheorem{oproblem}[theorem]{Open Problem}

\theoremstyle{remark}
\newtheorem{remark}[theorem]{Remark}

\numberwithin{equation}{section}
\numberwithin{figure}{section}

\newcommand{\R}{\mathbb{R}}
\newcommand{\N}{\mathbb{N}}
\newcommand{\C}{\mathbb{C}}
\renewcommand{\S}{\mathbb{S}}

\newcommand{\A}{\mathcal{A}}

\DeclareMathOperator{\divergence}{div}
\DeclareMathOperator{\dist}{dist}

\DeclareMathOperator{\tr}{tr}
\DeclareMathOperator{\conv}{conv}

\DeclareMathOperator{\myacc}{acc}

\renewcommand{\Im}{{\rm Im}\,}
\renewcommand{\Re}{{\rm Re}\,}
\newcommand{\I}{{\rm i}}
\newcommand{\e}{{\rm e}}
\newcommand{\Res}{{\rm Res}}
\newcommand{\lm}{\lambda}
\newcommand{\rd}{{\rm d}}

\DeclareMathOperator*{\essinf}{{\rm ess}\,{\rm inf}}
\DeclareMathOperator*{\esssup}{{\rm ess}\,{\rm sup}}

\newcommand{\dx}{{\rm d}x}

\newcommand{\dsigma}{{\rm d}\sigma}

\let\phi\varphi
\let\rho\varrho
\let\epsilon\varepsilon

\begin{document}

\begin{abstract}
We study the spectrum of the Robin Laplacian with a complex Robin parameter $\alpha$ on a bounded Lipschitz domain $\Omega$. We start by establishing a number of properties of the corresponding operator, such as generation properties, 
analytic dependence of the eigenvalues and eigenspaces on $\alpha \in \C$, and basis properties of the eigenfunctions.
Our focus, however, is on bounds and asymptotics for the eigenvalues as functions of $\alpha$: we start by providing estimates on the numerical range of the associated operator, which lead to new eigenvalue bounds even in the case $\alpha \in \R$. For the asymptotics of the eigenvalues as $\alpha \to \infty$ in $\C$, in place of the min-max characterisation of the eigenvalues and Dirichlet-Neumann bracketing techniques commonly used in the real case, we exploit the duality between the eigenvalues of the Robin Laplacian and the eigenvalues of the Dirichlet-to-Neumann map. We use this to show that along every analytic curve of eigenvalues, the Robin eigenvalues either diverge absolutely in $\C$ or converge to the Dirichlet spectrum, as well as to classify all possible points of accumulation of Robin eigenvalues for large $\alpha$. We also give a comprehensive treatment of the special cases where $\Omega$ is an interval, a hyperrectangle or a ball. This leads to the conjecture that on a general smooth domain in dimension $d\geq 2$ all eigenvalues converge to the Dirichlet spectrum if $\Re \alpha$ remains bounded from below as $\alpha \to \infty$, while if $\Re \alpha \to -\infty$, then there is a family of divergent eigenvalue curves, each of which behaves asymptotically like $-\alpha^2$. 
\end{abstract}


\thanks{\emph{Mathematics Subject Classification} (2010) 35J05 (35J25 35P10 35P15 35S05 47A10 81Q12)}

\thanks{\emph{Key words and phrases}. Laplacian, Robin boundary conditions, spectral theory of non-self-adjoint operators, estimates on eigenvalues}

\thanks{The authors would like to thank Timo Weidl for suggesting the problem, Andr\'e Froehly, Matthias Hofmann and Jens Wirth for helpful advice and suggestions, in particular regarding the use of the Dirichlet-to-Neumann operator, and a number of colleagues, in particular Marco Marletta, Hugo Tavares and Alexander Thumm, for providing valuable assistance with references to the literature on various topics. The work of the authors was supported by the Funda{\c{c}}{\~a}o para a Ci{\^e}ncia e a Tecnologia, Portugal, via the program ``Investigador FCT'', reference IF/01461/2015 (J.B.K.), and project PTDC/MAT-CAL/4334/2014 (all authors). The work of Robin Lang was supported by the Deutsche Forschungsgemeinschaft (DFG) through the Research Training
Group 1838: Spectral Theory and Dynamics of Quantum Systems.}

\date{\today}

\maketitle

\tableofcontents


\section{Introduction}
\label{sec:introduction}

In recent years a large body of literature has developed around the asymptotic behaviour of the eigenvalues of the Robin Laplacian
\begin{equation}
\label{eq:robin-laplacian}
\begin{aligned}
	-\Delta u &= \lambda u \qquad &&\text{in } \Omega,\\
	\frac{\partial u}{\partial\nu} + \alpha u &=0 \qquad &&\text{on } \partial\Omega,
\end{aligned}
\end{equation}
defined on a fixed domain $\Omega$, that is, a sufficiently smooth, bounded open set in $\R^d$, $d\geq 1$, as the parameter $\alpha \in \R$ appearing in the boundary condition tends to $\pm \infty$ (here and throughout $\nu$ denotes the \emph{outer} unit normal to $\partial\Omega$; if $d=1$, then we understand $\Omega$ to be a bounded interval). Denote these eigenvalues, which depend smoothly on $\alpha$, by $\lambda_1 (\alpha)  \leq \lambda_2 (\alpha) \leq \ldots \to \infty$, and the eigenvalues of the Dirichlet Laplacian, i.e., the solutions of
\begin{equation}
\label{eq:dirichlet-laplacian}
\begin{aligned}
	-\Delta u &= \lambda u \qquad &&\text{in } \Omega,\\
	u &=0 \qquad &&\text{on } \partial\Omega,
\end{aligned}
\end{equation}
by $\lambda_1 \leq \lambda_2 \leq \ldots \to \infty$. Then it is known that $\lambda_k (\alpha) \to \lambda_k$ from below as $\alpha \to +\infty$ for each $k\in\N$ \cite{Filinovskiy0,Filinovskiy}. If $\alpha \to -\infty$, then the situation is more complicated: if $\Omega$ is $C^1$, then $\lambda_k (\alpha) \sim -\alpha^2$ as $\alpha\to -\infty$ for each fixed $k \in \N$, but there are further curves of eigenvalues which converge to eigenvalues of the Dirichlet Laplacian from above \cite{AfrouziBrown,CCH,DanersKennedy,GiorgiSmits,GiorgiSmits2,LOS,LouZhu}; moreover, in the last few years very precise asymptotics have been developed for the divergent eigenvalues in the case $\alpha \to -\infty$ \cite{ExnerMinakovParnovski,FreitasKrejcirik,HelfferKachmar,KovarikPankrashkin,PankrashkinPopoff}. The case of less regularity, namely when $\Omega$ has a finite number of ``model corners'' and the asymptotic behaviour is different, has also been extensively considered \cite{BruneauPopoff,Khalile,KOBP,KhalilePankrashkin,LevitinParnovski}. We refer to \cite{BFK} for a recent summary of the problem, its history and more references.

Our principal goal is to investigate what happens when $\alpha \in \C$ is a large \emph{complex} parameter; the corresponding boundary condition is often called an impedance boundary condition, where it appears frequently in the context of electromagnetic and acoustic scattering (see, e.g., \cite{CCM,ChandlerWilde,LCS}). In this case, it is easy to see that the problem \eqref{eq:robin-laplacian} still admits a discrete spectrum, and studying this problem should give a more complete picture of the eigenvalue behaviour even in the real case. However, for $\alpha \in \C \setminus \R$, the Robin Laplacian obviously ceases to be self-adjoint, and thus neither the known results themselves, nor their methods of proof, which to a large extent rely on variational methods in some form, are applicable. Thus new methods and insights are required.

What is more, although there seems to be a burgeoning interest in non-self-adjoint Robin Laplacians in various contexts such as half-spaces \cite{ChandlerWilde,CossettiKrejcirik,PZ} and scattering problems (for example \cite{AlberRamm,CCM,LakshtanovVainberg,LCS} among many others); waveguides (e.g., \cite{Borisov08,Novak,Olendski11,Olendski12}); thin layers \cite{Borisov12,KRRS}; triangles \cite{McCartin,Shanin}; and metric graphs \cite{HKS}, to say nothing of the extensive physics literature on impedance boundary conditions (see for example the references in \cite{ChandlerWilde,LakshtanovVainberg,LCS,Olendski11}, etc.), to date many basic spectral properties of this operator on general (bounded) domains seem not yet to have been established. Thus we also wish to give a thorough and systematic treatment of these spectral properties; our first result is as follows.

\begin{theorem}
\label{thm:analytic-dependence}
Suppose $\Omega \subset \R^d$, $d\geq 1$, is a bounded Lipschitz domain. Then:
\begin{enumerate}
\item each eigenvalue has finite algebraic multiplicity and depends locally analytically on $\alpha \in \C$: more precisely, if $(\lambda_k (\alpha_0))_{k \in \N}$ is an enumeration of the eigenvalues (each repeated according to its finite algebraic multiplicity) for some $\alpha_0 \in \R$, then each $\lambda_k (\alpha_0)$ may be extended to a meromorphic function $\lambda_{k} (\alpha)$ such that for any $\alpha \in \C$, these eigenvalues form the spectrum of the corresponding Robin Laplacian;
\item away from crossing points of eigenvalues, each eigenvalue $\lambda_{k} (\alpha)$ and the corresponding eigenprojection are holomorphic functions of $\alpha$, whereas at the crossing points the weighted eigenvalue mean and the total projection are holomorphic;
\item if $\lambda_k (\alpha)$ is simple with eigenfunction $\psi=\psi(\alpha)$, then $\lambda_k'(\alpha)$ is given by
\begin{displaymath}
	\lambda_k'(\alpha) = \frac{\int_{\partial\Omega}\psi^2\,\dsigma(x)}{\int_\Omega \psi^2\,\dx}
\end{displaymath}
(where the right-hand side is to be interpreted as a holomorphic continuation in the event that the denominator is zero, as any singularities are removable);
\item for any $\alpha \in \C$, the set of eigenfunctions and generalised eigenfunctions corresponding to the eigenvalues $\{ \lambda_{ k} (\alpha): {k} \in \N \}$ can be chosen to form an Abel basis of $L^2(\Omega)$, of order $(d-1)/2+\delta$ for any $\delta>0$, and even a Riesz basis if $d=1$;
\item however, for any $\alpha \in \C \setminus \R$, the eigenfunctions 
can \emph{not} be chosen to form an orthonormal basis of $L^2 (\Omega)$.
\end{enumerate}
\end{theorem}

This theorem combines statements from several theorems which we will give below, namely Theorems~\ref{thm:robin-operator},~\ref{thm:holo},~\ref{thm:derivative},~\ref{thm:NoONB} and~\ref{thm:abel-basis}, and for its proof we refer to the respective proofs of these results. As a consequence of Theorem~\ref{thm:analytic-dependence}, the question of the asymptotic behaviour of the eigenvalues is meaningful since we can speak of analytic curves of eigenvalues in the complex plane (up to crossing points).

However, our main focus is on the location of these eigenvalues in the complex plane, in particular as regards their behaviour for large $\alpha$. Let us start by examining what we \emph{expect} to happen in the general case. Based on explicit calculations on the interval (which we will perform in Section~\ref{sec:interval}) and other concrete examples, we can expect that the behaviour should mirror the real case.

\begin{conjecture}
\label{conj:general}
Let $\Omega \subset \R^d$, $d\geq 2$, be a bounded Lipschitz domain, and suppose $\alpha \in \C$, $|\alpha| \to \infty$.
\begin{enumerate}
\item If $\Re \alpha \to -\infty$, then there exists a sequence of 
absolutely divergent eigenvalues.
Any limit point of non-divergent analytic eigenvalue curves of eigenvalues is an eigenvalue of the Dirichlet Laplacian (that is, a solution of \eqref{eq:dirichlet-laplacian}).
\begin{enumerate}
\item[(i)] If $\Omega$ has $C^1$ boundary, then each divergent eigenvalue behaves asymptotically like $-\alpha^2 + o (\alpha^2)$.
\item[(ii)] If $\Omega$ has Lipschitz boundary, then for any divergent analytic curve of eigenvalues $\lambda=\lambda_k(\alpha)$, there is a constant $C_{\Omega,k} \in[1,\infty)$ depending only on $k$ from Theorem \ref{thm:analytic-dependence},
such that $\lambda_k(\alpha) = -C_{\Omega,k} \alpha^2 + o (\alpha^2)$.
\end{enumerate}
\item If $\Re \alpha$ remains bounded from below, then each eigenvalue converges to an eigenvalue of the Dirichlet Laplacian.
\end{enumerate}
\end{conjecture}

We repeat that most statements in Conjecture~\ref{conj:general} are known for real $\alpha$, although some questions are still open; in particular the asymptotics on general Lipschitz domains has not yet been settled, see \cite[Open Problems~4.17 and~4.20]{BFK}. Regarding the divergent eigenvalues, we emphasise that it is now possible for them to have large \emph{positive} real part: $\Re (-\alpha^2) \to +\infty$ when $\alpha \to \infty$ in $\C$, if $|\Im \alpha|$ grows faster than $|\Re \alpha|$.

As mentioned above, existing techniques used in the real case are completely inapplicable to Conjecture~\ref{conj:general}, as they rely in an essential way either on the variational characterisation of the eigenvalues and test function arguments, as in \cite{DanersKennedy,GiorgiSmits,GiorgiSmits2,LOS}, or, what for our purposes amounts to the same thing, on Dirichlet-Neumann bracketing techniques (or equivalent) to decompose the operator, as in \cite{ExnerMinakovParnovski,LevitinParnovski,PankrashkinPopoff} etc.

Here, while we are not able to give a complete answer to Conjecture~\ref{conj:general}, and also leave open the question of the higher terms in the corresponding asymptotic expansions, we will make progress on two fronts. Firstly, we give sharp trace-type estimates on the boundary integral of the Robin eigenfunctions -- the only term in the expression for the eigenvalues with possibly non-zero imaginary part -- to control the location of the spectrum of the Robin Laplacian for fixed $\alpha \in \C$ inside an explicitly specified parabolic-type region of the complex plane.

\begin{theorem}
\label{thm:general-eigenvalue-estimate}
Suppose $\Omega \subset \R^d$, $d\geq 2$, is a bounded Lipschitz domain. Then there exist constants $C_1 \geq 2$ and $C_2>0$ depending only on $\Omega$, such that for any $\alpha \in \C$, any corresponding eigenvalue $\lambda \in \C$ of \eqref{eq:robin-laplacian} is contained in the set
\begin{displaymath}
	\Lambda_{\Omega,\alpha} := \left\{ t + \alpha \cdot s \in \C: t \geq 0,\, s \in [0,C_1\sqrt{t} + C_2] \right\};
\end{displaymath}
in particular, we have the estimate
\begin{displaymath}
	\Re \lambda \geq -\frac{C_1^2}{4}|\Re\alpha|^2 - C_2|\Re\alpha|.
\end{displaymath}
If $\Omega$ has $C^2$ boundary, then we may choose $C_1=2$.
\end{theorem}

Actually, we will prove a slightly stronger version of this theorem, namely for the numerical range of the associated form: see Section~\ref{sec:numerical-range} for details, including a description of the parabolic-type region $\Lambda_{\Omega,\alpha}$, and in particular Theorem~\ref{thm:numerical-range} for the stronger version and its proof.

In the case of real negative $\alpha$, Theorem~\ref{thm:general-eigenvalue-estimate} is already new; for general reference, we will formulate it here explicitly:

\begin{corollary}
\label{cor:neg-alpha-bound}
Suppose $\Omega \subset \R^d$, $d\geq 2$, is a bounded Lipschitz domain. Then there exist constants $c_1\geq 1$ and $c_2>0$ depending only on $\Omega$ such that for any $\alpha<0$ and any corresponding eigenvalue $\lambda \in \R$ we have
\begin{displaymath}
	\lambda \geq -c_1\alpha^2 - c_2\alpha.
\end{displaymath}
If $\Omega$ has $C^2$ boundary, then we may choose $c_1=1$.
\end{corollary}

Among other things, this essentially answers \cite[Open Problem~4.17]{BFK} in the affirmative (see Remark~\ref{rem:neg-alpha-bound} for more details): as $\alpha \to -\infty$, for any bounded Lipschitz domain $\Omega$, there exists a constant $c_1 = c_1(\Omega) >0$ such that $\lambda_1 (\alpha) \gtrsim -c_1 \alpha^2$. To the best of our knowledge, this is also the first time a bound of the form $\lambda \geq -\alpha^2 + c_2 \alpha$ has been found which is valid for all $\alpha < 0$ and general smooth domains; in this case, the constant $c_2 = c_2 (\Omega)$ can be estimated explicitly in terms of the geometry of $\Omega$ and is related to the maximal mean curvature of $\partial\Omega$ (see Remark~\ref{rem:comega}).

The other part of our approach is based on the duality between the Robin Laplacian on $L^2 (\Omega)$ and Dirichlet-to-Neumann-type operators defined on $L^2 (\partial\Omega)$. Suppose that some $\lambda \in \C$ is an eigenvalue of the Robin Laplacian for some given $\alpha \in \C$ and not in the spectrum of the Dirichlet Laplacian. Then $\alpha$ is an eigenvalue of the operator which maps given Dirichlet data $g \in L^2 (\partial \Omega)$ to the (negative of the) outer normal derivative $-\frac{\partial u}{\partial\nu}$, if one exists, of the solution $u$ of the Dirichlet problem
\begin{equation}
\label{eq:helmholtz}
\begin{aligned}
	-\Delta u &= \lambda u \qquad &&\text{in } \Omega,\\
	u &=g \qquad &&\text{on } \partial\Omega
\end{aligned}
\end{equation}
for this value of $\lambda$, and vice versa. The Dirichlet-to-Neumann operator is defined in such a way that $\lambda \in \C$ is an eigenvalue of the Robin Laplacian \eqref{eq:robin-laplacian} for a given $\alpha \in \C$ if and only if $\alpha$ is an eigenvalue of the Dirichlet-to-Neumann eigenvalue for the corresponding spectral parameter $\lambda$. As such, the study of the Robin eigenvalues is equivalent to the problem of studying the dependence of the Dirichlet-to-Neumann eigenvalues $\alpha$ as functions of $\lambda$; and indeed the duality between the two has been explored and exploited frequently in various other contexts such as  \cite{ArendtMazzeo,AtE11,Daners,GesztesyMitrea,Marletta}, among numerous others. It turns out that it is often easier to study the behaviour of $\alpha$ as a function of $\lambda$ than the other way round, and this is the approach we will take. It firstly allows us to give a short proof of a dichotomy result which forms part of Conjecture~\ref{conj:general}.

\begin{theorem}
\label{thm:Eigencurve_EitherConvDir_Div}
Let $\Omega \subset \R^d$, $d\geq 1$, be a bounded Lipschitz domain and $\alpha \in \C$. As $\alpha \to \infty$ in $\C$, each analytic eigenvalue curve $\lambda = \lambda (\alpha) $ of \eqref{eq:robin-laplacian} either converges to a point in the Dirichlet spectrum or diverges to $\infty$ in $\C$.
\end{theorem}

This will be proved in Section~\ref{sec:dno} (see also Theorem~\ref{Interval:thm:ConvergenceDirichletSpectrum} for the case $d=1$). In the real case, although this was expected, it does not previously seem to have been formally proved, see \cite[Open Problem~4.11]{BFK}; thus, Theorem~\ref{thm:Eigencurve_EitherConvDir_Div} also fills this small gap in the literature in the case of real $\alpha$.

Instead of looking at individual eigencurves, we can consider the asymptotic distribution of eigenvalues across all eigencurves in their entirety; more precisely, we can consider all possible points of accumulation of the Robin eigenvalues as $\alpha \to \infty$ in $\C$. Away from the negative real semi-axis, we have a stronger statement than the one of Theorem~\ref{thm:Eigencurve_EitherConvDir_Div}, namely that regardless of how we choose the eigenvalues, as $\alpha \to \infty$ the only points of accumulation are Dirichlet eigenvalues. If $\alpha$ is allowed near the negative real semi-axis, however, the situation is more complicated; it is for these values of $\alpha$ that the eigenvalue ``crossings'' accumulate.

\begin{theorem}
\label{thm:EitherConvDir_Div}
Let $\Omega \subset \R^d$, $d\geq 2$, be a bounded Lipschitz domain and $\alpha \in \C$.
\begin{enumerate}
\item If $\alpha \to \infty$ in $\C$ in such a way that either $\Re \alpha$ remains bounded from below \emph{or} $\left|\frac{\Re \alpha}{\Im \alpha}\right|$ remains bounded, then the only points of accumulation of the Robin Laplacian eigenvalues as $\alpha \to \infty$ are eigenvalues of the Dirichlet Laplacian.
\item However, any $\lambda \in \C$ is a point of accumulation of the eigenvalues of the Robin Laplacian if $\alpha \in \C$ is allowed to be arbitrary. More precisely, given any $\lambda \in \C$ there exist $\alpha_k \in \C$, $k\in \N$, $|\alpha_k| \to \infty$, such that $\lambda$ is an eigenvalue of the Robin Laplacian with parameter $\alpha_k$, for all $k \in \N$.
\end{enumerate}
\end{theorem}

For the proof and a more detailed discussion of the statement (in particular the contrast between parts (1) and (2)), see Section~\ref{sec:limit-points}; we refer in particular to the more precise version of (1) that is the statement of Theorem~\ref{thm:robin-accum-sector}, as well as Remark~\ref{rem:limit-points-various}.

We will also use Dirichlet-to-Neumann operators to give a detailed analysis of the asymptotic behaviour of the Robin eigenvalues in a number of concrete examples, namely the interval, rectangles and hyperrectangles, and balls in $d\geq 2$ dimensions, which support Conjecture~\ref{conj:general}. We expect that many of the ideas here could be carried over to more general settings. For example, the case of quantum graphs with some $\delta$ vertex conditions, that is, the Laplacian defined on a metric graph with complex Robin-type potentials at some of the vertices, can be analysed using the same ideas and will be treated in a later work.

This paper is organised as follows. To motivate our results, we start out in Section~\ref{sec:interval} by sketching the case of the eigenvalues of \eqref{eq:robin-laplacian} in the special case when $\Omega$ is a bounded interval and everything can be calculated explicitly. Divergence of the eigenvalues $\lambda$ outside an arbitrarily small sector around the positive real semi-axis is shown to be possible only if $\Re \alpha \to -\infty$; in this case, one obtains exactly two divergent eigenvalues, which behave like $-\alpha^2$, while the rest converge to the Dirichlet spectrum. If $\Re \alpha$ remains bounded from below, then, at least outside such a sector, all eigenvalues are convergent (see Theorems~\ref{Interval:thm:ConvergenceDirichletSpectrum} and~\ref{thm:DivergingEigenvaluesInterval}, as well as Proposition~\ref{prop:strip}).

In Section~\ref{sec:robin-operator}, we then introduce the Robin Laplacian as an operator on $L^2(\Omega)$ and establish basic spectral and generation properties such as m-sectoriality. Section~\ref{sec:holomorphy} is devoted to the holomorphic dependence of the eigenvalues and eigenfunctions on $\alpha$ (Theorem~\ref{thm:holo} and Remark~\ref{rem:cont}) based on Kato's theory, the question of the possible existence of eigennilpotents at eigenvalue crossing points (Remarks~\ref{rem:eigennilpotents} and~\ref{rem:eigennilpotents-not-zero}) as well as the proof of the formula for the derivative of a simple eigenvalue (Theorem~\ref{thm:derivative}). In Section~\ref{sec:eigenfunctions} we treat the failure of the eigenfunctions to form an orthonormal basis in $L^2$ (Theorem~\ref{thm:NoONB}), as well as the positive result that they at least form an Abel basis (Theorem~\ref{thm:abel-basis}). Theorem~\ref{thm:analytic-dependence} follows immediately from the results in Sections~\ref{sec:holomorphy} and~\ref{sec:eigenfunctions} (plus elementary properties of the operator given in Section~\ref{sec:robin-operator}). The bounds on the region in $\C$ in which eigenvalues can be found are in Section~\ref{sec:numerical-range}; in particular, we give the statement and proof of our main Theorem~\ref{thm:numerical-range}, which in particular implies Theorem~\ref{thm:general-eigenvalue-estimate} and hence also Corollary~\ref{cor:neg-alpha-bound}.

In Section~\ref{sec:dno}, we introduce and prove a few basic properties of the Dirichlet-to-Neumann operator, including the ``duality'' between the Robin and Dirichlet-to-Neumann eigenvalue problems, which is well known in the real case; we also give the proof of Theorem~\ref{thm:Eigencurve_EitherConvDir_Div}. We then use the Dirichlet-to-Neumann operator among other tools to prove Theorem~\ref{thm:EitherConvDir_Div} on the points of accumulation of the Robin eigenvalues in Section~\ref{sec:limit-points}. Finally, in Section~\ref{sec:examples}, we give three concrete examples: we start with the interval, where we furnish a number of technical details omitted from the exposition in Section~\ref{sec:interval}, including a consideration of the relation between the eigenvalues diverging near the positive real semi-axis and the parameter $\alpha$. We then use our results on the interval to deal with $d$-dimensional rectangles (hyperrectangles), see Theorem~\ref{hyperrectangles:spectrum}, and finally, we treat $d$-dimensional balls in Section~\ref{subsec:balls}, see in particular Theorems~\ref{lem:Ball:DirichletConvergenceRate} and~\ref{thm:ExistenceDiv}. In these examples we also pay attention to the error estimates appearing in the asymptotic expansions.

\begin{remark}
Some of our results are valid in essentially the same form if $\alpha$ is allowed to be variable, that is, a complex-valued function $\alpha \in L^\infty (\partial \Omega, \C)$, in place of a constant $\alpha \in \C$. This is especially true of the basic operator-theoretic properties collected in Section~\ref{sec:robin-operator} (see Remark~\ref{rem:operator-variable-alpha}), as well as our estimates on the numerical range (see Remark~\ref{rem:range-variable-alpha}), since these are all based on trace-type estimates which continue to hold for $\alpha \in L^\infty(\partial\Omega,\C)$. However, in most other cases it introduces significant complications and many results are unlikely to hold in the same form (as for example with the Dirichlet-to-Neumann operator, or see also Remark~\ref{rem:eigennilpotents-not-zero}). Since for our main problem, namely the asymptotic behaviour of the eigenvalues for large $\alpha$, it is customary and of most interested to treat $\alpha$ as a (real or in this case complex) parameter, we will not consider the case of variable $\alpha$ beyond the aforementioned remarks.
\end{remark}

\section{A motivating example: the interval}
\label{sec:interval}

To gain insight into what to expect in general, we start by looking at the case of intervals, where everything can be computed explicitly. We start by fixing $a>0$ and consider the interval $\Omega = (-a,a) \subset \R$ of length $2a$. Here we will present a slightly abridged version; the (somewhat tedious) details of the calculations are given in Section~\ref{subsec:ExampleInterval}. In one dimension, our problem becomes
\begin{equation}
\label{eq:interval-robin}
\begin{aligned}
	-\Delta u = -u'' &= \lambda u \qquad \text{on } (-a,a), \\
	-u'(-a) + \alpha u(-a) &= 0, \\
	u'(a) + \alpha u(a) &= 0
\end{aligned}
\end{equation}
for given $\alpha \in \C$ (where the sign in front of $u'(\pm a)$ corresponds to the outer normal derivative at $\pm a$). We will study this problem with the help of the inhomogeneous Dirichlet problem
\begin{equation}
\label{Interval:DirichletProblem}
\begin{aligned}
	  -u''&=\lambda u\qquad\text{on }(-a,a), \\
	  u(-a)&= g_1, \\
	  u(+a)&= g_2,
\end{aligned}
\end{equation}
for given Dirichlet data $g:=(g_1, g_2)^T\in\mathbb{C}^2$ and $\lambda\in\C$. A number $\lambda$ solving \eqref{Interval:DirichletProblem} for given $g$ is an eigenvalue of the Laplacian with complex Robin boundary conditions \eqref{eq:interval-robin} if and only if there is a solution $u$ of \eqref{Interval:DirichletProblem} such that $u'(-a) = \alpha g_1$ and $-u'(a) = \alpha g_2$. Let us write $M(\lambda)$ for the mapping which takes $(g_1,g_2)^T$ to $(u'(-a),-u'(a))^T$, that is, $M(\lambda) \in \mathbb{C}^{2\times 2}$ is the \emph{Dirichlet-to-Neumann operator} (matrix) mapping given Dirichlet data to the associated Neumann data of the corresponding $\lambda$-harmonic function $u$, which we study in more detail in Section~\ref{sec:dno}. Thus a Robin eigenvalue $\lambda$ for given $\alpha$ corresponds to an eigenvalue $\alpha$ of the equation
\begin{equation}
\label{Interval:RobinCondition}
	M(\lambda)g = \alpha g = \alpha ( g_1, g_2)^T
\end{equation}
for given $\lambda$. In anticipation of our later strategy, to study the behaviour of the Robin eigenvalues, we will in fact study the eigenvalues $\alpha$ of the matrix $M(\lambda)$. To this end, starting with the general solution of \eqref{Interval:DirichletProblem} given by
\begin{equation}
\label{Interval:GeneralSolution}
	u(x)=C_+\cos(\sqrt{\lambda} x)+ C_-\sin(\sqrt{\lambda} x),
\end{equation}
whose coefficients $C_+$ and $C_-$ depend on $a$, 
$\sqrt{\lm}$,
and $g$, it is not difficult to derive the representation
\begin{equation}
\label{Interval:DNOperatorMatrix}
	M(\lambda)= \sqrt{\lambda}\begin{pmatrix}
		-\cot 2\sqrt\lambda a & \csc 2 \sqrt\lambda a \\
		\csc 2 \sqrt\lambda a & -\cot 2\sqrt\lambda a
	\end{pmatrix}.
\end{equation}
We see that this matrix is well defined, and has two eigenvalues, except at the singularities of $\cot$ and $\csc$. These correspond exactly to the values $\pi^2j^2/(4a^2)$, 
$j \in \mathbb{Z}$, 
of $\lambda$, that is, the eigenvalues of the Dirichlet Laplacian on $(-a,a)$ together with $0$, which is a double eigenvalue. From this representation we can also deduce that the eigenvalues 
\begin{equation}
\label{eq:alpha-curves-1D}
	\alpha_\pm = \sqrt{\lambda}\left(\pm\csc(2a\sqrt{\lambda})-\cot(2a\sqrt{\lambda})\right)
\end{equation}
of \eqref{Interval:RobinCondition} depend analytically on $\lambda \neq \pi^2j^2/(4a^2)$. Moreover, apart from the crossing at $\lambda = 0$, the two curves $\alpha_\pm$ described by \eqref{eq:alpha-curves-1D} have no points of intersection; and their respective derivatives $\mathrm{d}\alpha_\pm/\mathrm{d}\lambda$ never vanish. Hence, away from this one crossing, the eigenvalues $\lambda = \lambda (\alpha)$ of \eqref{eq:interval-robin} are simple and depend analytically on $\alpha \in \C$. A more general version of this will be discussed in Section~\ref{sec:holomorphy}. At any rate, for this reason, whenever we speak of divergent or convergent eigenvalues $\lambda (\alpha)$ as $\alpha \to \infty$ in $\C$, we have a family of (in general meromorphic, here even analytic) functions and are considering the asymptotic behaviour of each of these.

Moreover, to establish what types of behaviour of $\lambda (\alpha)$ are possible as $\alpha \to \infty$, we may equally ask what conditions on $\lambda$ guarantee that the eigenvalues $\alpha$ of the matrix $M(\lambda)$ diverge. To this end, we classify the different situations in which this can happen as follows:
\begin{enumerate}
\item $\sqrt{\lambda}$ approaches a pole of $\cot$ or $\csc$, which represent the Dirichlet eigenvalues. In this case, as $\alpha \to \infty$ the Robin eigenvalue $\lambda$ converges to a Dirichlet eigenvalue;
\item $\lambda$ diverges to $\infty$ in $\C$ away from the positive real axis, where the poles of $\cot$ and $\csc$ are located. In this case, as we shall see, both eigenvalues of $M(\lambda)$ diverge as $\pm \I\sqrt\lambda$, corresponding to two divergent Robin eigenvalues $\lambda \sim -\alpha^2$;
\item $\lambda$ diverges to $\infty$ but remains within a finite distance of the real axis. While it is clear that the eigenvalues of $M(\lambda)$ must also diverge in this case, the relationship between $\alpha$ and $\lambda$ appears to be more complicated owing to the proximity of $\sqrt{\lambda}$ to the poles of $M(\lambda)$.
\end{enumerate}
Let us examine each situation a little more closely.

\subsection{Convergence to the Dirichlet spectrum}


Consider the behaviour of the eigenvalues $\alpha(\lambda)$ of $M(\lambda)$ as $\sqrt{\lambda}$ approaches a singularity of $\cot$ or $\csc$, that is, $\lambda$ approaches an eigenvalue of the Dirichlet Laplacian: this is the only case in which $\alpha$ may diverge while $\lambda$ remains bounded. Inverting this statement by writing $\lambda$ as a function of $\alpha$ leads to the following theorem, whose proof will be given in Section \ref{subsec:ExampleInterval}; see also Theorem~\ref{thm:EitherConvDir_Div}.

\begin{theorem}
\label{Interval:thm:ConvergenceDirichletSpectrum}
Suppose the analytic eigencurve $\lambda=\lambda(\alpha)$ remains bounded as $\alpha \to \infty$ in $\C$. Then it converges to some eigenvalue of the Dirichlet Laplacian, that is, there exists some $j\in\mathbb{Z}$ such that
\begin{displaymath}
	\lambda(\alpha)\to \frac{\pi^2j^2}{4a^2}
\end{displaymath}
as $\alpha \to \infty$.
\end{theorem}

\subsection{Divergent eigenvalues away from the positive real axis}
\label{subsec:interval_divergent}
Suppose now that $\lambda \to \infty$ in $\C$ in such a way that its distance to the positive real axis diverges. For simplicity, we will actually suppose that $\lambda$ diverges in a sector away from the positive real axis; more precisely, we start by dividing the complex plane in the following fashion:

\begin{definition}
\label{def:sectors}
\begin{enumerate}
\item Let $0<\theta<\pi/2$ be an (arbitrarily small) angle and define the open sectors
\begin{equation}\label{eq:sectors}
	S_\theta^+:=\{z\in\C\;:\; \theta<\arg z < \pi-\theta\}\qquad\text{and}\qquad T_\theta^+ :=\{z\in\C\;:\; |\arg z|<\theta\}
\end{equation}
in the upper and right-hand half-planes, respectively (here we assume that the principal argument is always between $-\pi$ and $\pi$). We then define $S_\theta^- := -S_\theta^+$ and $T_\theta^-:=-T_\theta^+$ to be the corresponding sectors reflected in the real and imaginary axes, respectively, so that the complex plane is, up to two straight lines mutually crossing in $z=0$, symmetrically partitioned into four sectors; see Figure~\ref{fig:sectors}.
\item If $\theta = \pi/2$, the sectors $S_\theta^\pm$ vanish and $T_\theta^\pm$ are defined as in \eqref{eq:sectors}.
\end{enumerate}
 Furthermore, if $\pi/2<\theta'<\pi$, we set define $T_{\theta'}^+$ by \eqref{eq:sectors}, that is, a partition of the complex plane in two sectors $T_{\theta'}^+$ and $T_{\pi-\theta'}^-$
\end{definition}

\begin{figure}[h]
  \includegraphics[scale=0.42]{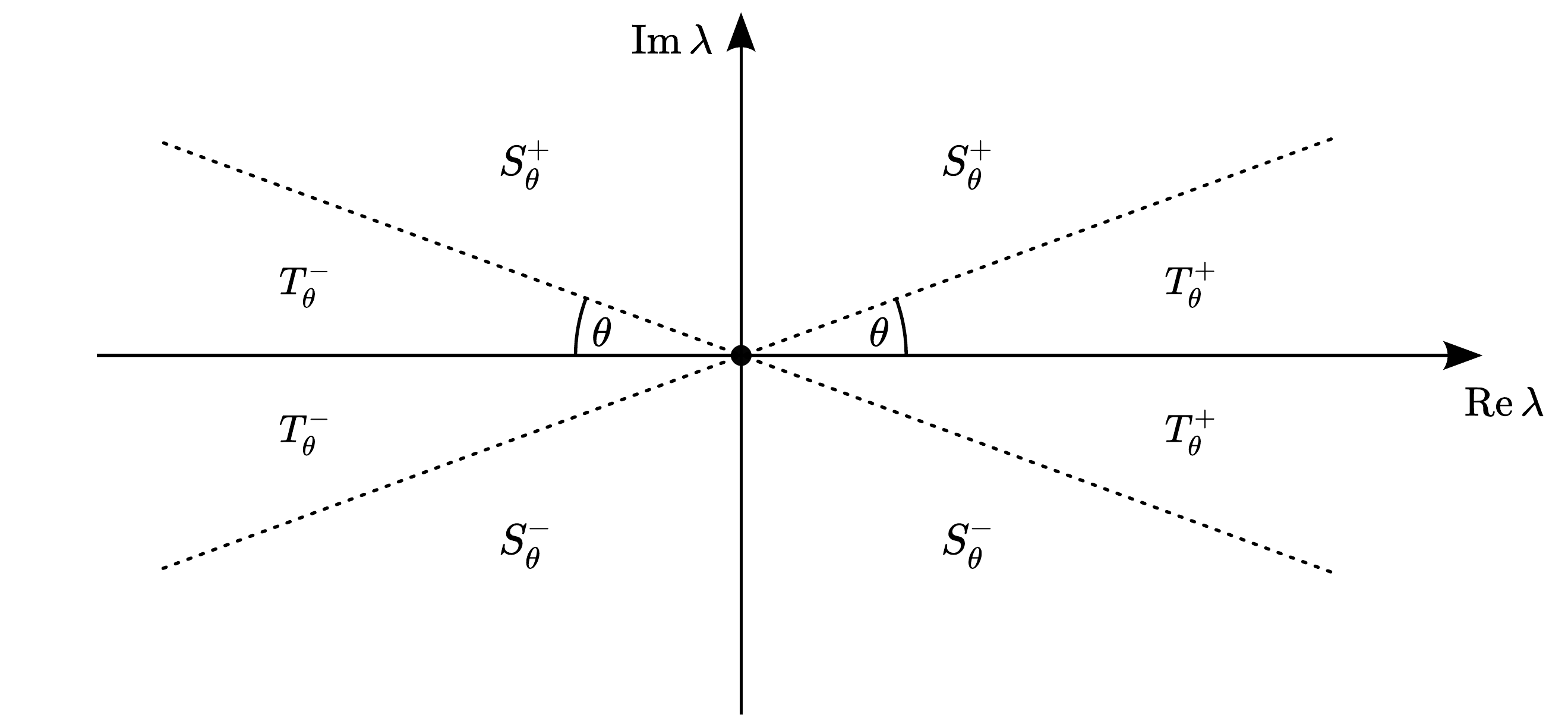}
  \caption{The four sectors $S_\theta^\pm$ and $T_\theta^\pm$.}
  \label{fig:sectors}
\end{figure}

We then make the following assumption.

\begin{assumption}
\label{assumptionSector}
We suppose that $\lambda$ diverges in the sector $\C \setminus T_{2\theta}^+$ for some small $\theta \in (0,\pi/2)$.
\end{assumption}

This ensures that $\lambda$ does not approach any eigenvalue $\lambda_j\in\sigma(-\Delta_\Omega^D)\subset\R$ of the Dirichlet Laplacian; moreover, the assumption is equivalent to $\sqrt{\lambda}$ diverging to $\infty$ in one of the sectors $S_\theta^\pm$.

But this implies in particular that $\Im \sqrt{\lambda} \to \pm \infty$, and for such $\sqrt{\lambda}$ we can determine the asymptotic behaviour of the Dirichlet-to-Neumann matrix \eqref{Interval:DNOperatorMatrix}, based on
\begin{equation}
\label{cotAsy}
	\cot z = \I \left(1+ \frac{2}{\e^{2\I z}-1} \right) = \mp \I + \mathcal{O} \left( \e^{\mp 4\Im z}\right)
\end{equation}
and
\begin{equation}
\label{cscAsy}
	\csc z = \frac{2\I}{\e^{\I z} - \e^{-\I z}} = \mathcal{O} \left(\e^{\mp 2 \Im z}\right)
\end{equation}
as $\Im z \to \pm \infty$, independently of $\Re z$. Indeed, Assumption~\ref{assumptionSector} allows us to choose $z=a\sqrt\lambda$, which leads to
\begin{equation}\label{eq:IntervalDNAsymptotics}
	M(\lambda)=\I\sqrt{\lambda} 
	\begin{pmatrix}
	\pm 1 & 0 \\
	0 & \pm 1
	\end{pmatrix}+\mathcal{O} \left(\sqrt{\lambda}\e^{\mp 2a\Im\sqrt{\lambda}}\right). 
\end{equation}
Recalling \eqref{Interval:RobinCondition}, in each of the cases $\Im\sqrt{\lambda}\to +\infty$ and $\Im\sqrt{\lambda} \to -\infty$ we obtain the respective existence of exactly one diverging eigenvalue behaving like $\alpha = \alpha (\lambda)$, whose square satisfies the behaviour
\begin{equation}
\label{Interval:Asymptotic_alpha(lambda)}
	\alpha^2 = -\lambda + \mathcal{O} \left(\lambda\e^{\mp 2a\Im\sqrt{\lambda}}\right)
\end{equation}
as $\Im \sqrt{\lambda} \to \pm \infty$. Inverting the equation from $\alpha(\lambda)$ to $\lambda(\alpha)$ and noting that these eigenvalues always correspond to $\Re \alpha \to -\infty$ (more precisely, we want $\alpha \to \infty$ in the left half-plane away from the imaginary axis, in order to guarantee that $-\alpha^2$ remains away from the positive real axis), we arrive at the following result.

\begin{theorem}
\label{thm:DivergingEigenvaluesInterval}
For the interval $\Omega=(-a,a)$, if $\alpha \to \infty$ in a sector of the form $T_\varphi^-$ for any $\varphi \in (0,\pi/2)$ (see Definition~\ref{def:sectors}), then for any $\theta \in (0,\pi-2\varphi)$ there are exactly two divergent eigenvalues of the Robin Laplacian in the sector $\C \setminus T_\theta^+$; these satisfy the asymptotics
\begin{equation}
\label{eq:DivergingEigenvaluesInterval}
	\lambda(\alpha)=-\alpha^2+\mathcal{O}\left( \alpha^2\e^{2a\Re\alpha}\right)
\end{equation} 
as $\alpha \to \infty$ in $T_\varphi^-$. If $\alpha \to \infty$ in such a way that $\Re \alpha$ remains bounded from below, then the Robin Laplacian has no divergent eigenvalues in $\C \setminus T_\theta^+$, for any $\theta > 0$.
\end{theorem}

A special case and immediate implication of the latter theorem is $\alpha$ diverging on any ray (half-line) in the left half-plane and thus in a sector $T_\varphi^-$ for some given $\varphi \in (0,\pi/2)$: we suppose $\alpha$ may be written as a function $\alpha: (0,\infty)\ni t\mapsto t\e^{\I\vartheta}\in\C$ for some fixed $\pi/2 <\vartheta<3\pi/2$, which in particular means that $\alpha(t)\in T_\varphi^-$ for all $t>0$.

\begin{corollary}
For the interval $\Omega=(-a,a)$, if $\alpha(t)=t\e^{\I\vartheta}\to\infty$ for any fixed $\pi/2 <\vartheta<3\pi/2$, then for any $\theta \in (0,\pi-2\vartheta)$, for sufficiently large $t>0$ there are exactly two eigenvalues $\lambda$ of the Robin Laplacian in the sector $\C\setminus T_\theta^+$, and these both satisfy the asymptotics
\begin{equation}\label{Interval:LambdaForRays}
\lambda(\alpha(t)) = -t^2\e^{2\I\vartheta} + \mathcal{O}\left(t^2\e^{2\cos(\vartheta)at} \right)
\end{equation}
as $t\rightarrow\infty$.
\end{corollary}

For full details and a proof of the theorem we refer to Section~\ref{subsec:ExampleInterval}. The eigenvalue behaviour described in Theorems~\ref{Interval:thm:ConvergenceDirichletSpectrum} and~\ref{thm:DivergingEigenvaluesInterval}, and our approach taken here, should be compared with the corresponding case of real $\alpha$ discussed in \cite[Section~4.3.1]{BFK}.

\subsection{Divergent eigenvalues near the positive real axis}
\label{subsec:interval_divergent_positive_real}

The other possibility is that $\lambda \to \infty$ inside a sector of the form $T_{2\theta}^+$, equivalently, that $\sqrt{\lambda} \to \infty$ in one of the sectors $T_\theta^\pm$. While it is clear that the corresponding eigenvalues $\alpha$ of $M(\lambda)$ must diverge, equivalently, that the divergent eigenvalues $\lambda(\alpha)$ in this sector correspond to divergent $\alpha$, the situation is complicated by the proximity of $\sqrt{\lambda}$ to the poles of $\cot$ and $\csc$. In such cases, we can expect $\Im\alpha \to \pm \infty$, consistent with the asymptotics $\lambda \sim -\alpha^2$. 
However, in general any particular $\lambda$-curve
such that $\lambda$ diverges along a path within a strip of fixed width around the positive real axis, the corresponding eigenvalues $\alpha (\lambda)$ of $M(\lambda)$ satisfy $|\Im \alpha (\lambda)| \to \infty$ and $\Re \alpha (\lambda)$ oscillates and diverges indefinitely:

\begin{proposition}\label{prop:strip}
Suppose $\lambda$ diverges along a path within a strip of fixed width around the positive real axis. Then the corresponding eigenvalues $\alpha (\lambda)$ of $M(\lambda)$ satisfy $|\Im \alpha (\lambda)| \to \infty$ and $\Re \alpha (\lambda)$ oscillates and diverges indefinitely.
\end{proposition}

For the proof, see again Section~\ref{subsec:ExampleInterval}. Among other things, this intimates, when combined with the proof of Theorem \ref{thm:DivergingEigenvaluesInterval}, that the Robin Laplacian can only have divergent eigenvalues in the regime $\Re\alpha \to -\infty$ (indeed, if $\Re\alpha$ remains bounded from below, then the conjecture rules out divergent eigenvalues $\lambda$ such that $\Im\sqrt{\lambda}$ remains bounded); while by Theorem~\ref{thm:DivergingEigenvaluesInterval} and its proof there can be no divergent eigenvalues $\lambda$ such that $\Im\sqrt{\lambda} \to \pm \infty$. This observation, in particular, supports Conjecture~\ref{conj:general}.

A more complete description of the relationship between $\lambda$ and $\alpha$ in this case will however be deferred to a later work.


\section{The Robin Laplacian with complex parameter}
\label{sec:robin-operator}

In this section we will collect a number of basic properties of the Robin Laplacian. We will be using the framework of Kato \cite[Chapter V and VII]{Kato}, and we start by recalling some definitions from there.

We assume throughout that $H$ is a Hilbert space with inner product $(\,\cdot\,,\,\cdot\,)$ and norm $\|\cdot\|_H$, $\mathcal{A}:D(\mathcal{A})\subset H\to H$ is a closed, densely defined linear operator with spectrum $\sigma (\mathcal{A}) \subset \C$, point spectrum (set of eigenvalues) $\sigma_p(\mathcal A)$ and resolvent set $\rho (\mathcal{A}) = \C \setminus \sigma(\mathcal{A})$, and $a: D(a) \times D(a) \subset H \times H \to \C$ is a densely defined sesquilinear form. We call the set
\begin{equation}
\label{eq:operator-numerical-range}
	W(\mathcal{A}):=\{(\mathcal{A}u,u): u \in D(\mathcal{A}) \text{ and } \|u\|_H = 1\} \subset \C
\end{equation}
the \emph{numerical range of $\mathcal{A}$} and, likewise, the set
\begin{equation}
\label{eq:form-numerical-range}
	W(a):=\{a[u,u]: u \in D(a) \text{ and } \|u\|_H = 1\} \subset \C
\end{equation}
the \emph{numerical range of $a$}. If $\mathcal{A}$ is the operator associated with $a$, that is, if $\mathcal{A}$ is defined by
\begin{displaymath}
\begin{aligned}
	D(\mathcal{A}) &= \{ u \in D(a): \exists h \in H \text{ such that } a[u,v]=(h,v)\;\; \forall v \in D(a) \},\\
	\mathcal{A}u &= h,
\end{aligned}
\end{displaymath}
then it follows immediately from the definitions that $\sigma_p (\mathcal{A}) \subset W (\mathcal{A}) \subset W (a)$.

Finally, we call $\mathcal{A}$ \emph{m-sectorial} (of semi-angle $\theta$) if there exist a vertex $\gamma \in \R$ and an angle $0\leq \theta < \pi/2$ such that
\begin{equation}
\label{eq:sectorial}
	W (\mathcal{A}) \subset \{z \in \C: |\arg (z-\gamma)| \leq \theta \}
\end{equation}
(where, again, the principal argument of a complex number is taken to be between $-\pi$ and $\pi$) and for all $\lambda \in \C$ with $\Re \lambda < \gamma$ we have that $\lambda \in \rho(\mathcal{A})$ satisfies the resolvent norm estimate
\begin{displaymath}
	\| (\mathcal{A}-\lambda I)^{-1} \|_{H\to H} \leq \frac{1}{|\gamma - \Re \lambda|}.
\end{displaymath}
The form $a$ is likewise called \emph{sectorial} (of semi-angle $\theta$) if \eqref{eq:sectorial} holds for $W (a)$.

Now let $\Omega \subset \R^d$, $d \geq 1$, be a bounded domain, that is, a bounded open set with a finite number of connected components, and (if $d\geq 2$) assume that its boundary $\partial\Omega$ is locally the graph of a Lipschitz function, for short Lipschitz. For $\alpha \in \C$ we define the sesquilinear form $a_\alpha : H^1 (\Omega) \times H^1 (\Omega) \to \C$ by
\begin{equation}
\label{eq:robin-form}
	a_\alpha[u,v]=\int_\Omega \nabla u\cdot\overline{\nabla v}\,\dx +\int_{\partial\Omega}\alpha u\overline{v}\,\mathrm{d}\sigma(x),
\end{equation}
where the boundary integral is to be understood in the sense of traces, as is customary; more precisely, we have written $u$ and $\overline{v}$ as shorthand for the traces $\tr u,\overline{\tr v} \in L^2(\partial\Omega)$ of the functions $u,v \in H^1(\Omega)$, respectively (see, e.g., \cite[Section~1]{AtE11} and also Lemma~\ref{lem:robin-operator} below). We will refer to the form $a_\alpha$ as the \emph{Robin form} (for the parameter $\alpha$) and call the operator on $L^2 (\Omega)$ associated with $a_\alpha$ the \emph{Robin Laplacian}, denoted by $-\Delta_\Omega^\alpha$. The arguments of, e.g., \cite[Section~2]{ArendtMazzeo} or \cite[Section~1]{RohlederBook} for real $\alpha$ may be repeated verbatim here to show that this operator is given by
\begin{displaymath}
\begin{aligned}
	D(-\Delta_\Omega^\alpha) &= \left\{ u \in H^1 (\Omega): \Delta u \in L^2 (\Omega) \text{ and } \frac{\partial u}{\partial\nu} \in L^2(\partial\Omega)
	\text{ with } \frac{\partial u}{\partial\nu} + \alpha u = 0 \right\},\\
	-\Delta_\Omega^\alpha u &= -\Delta u,
\end{aligned}
\end{displaymath}
where $\Delta u = \sum_{i=1}^d \frac{\partial^2 u}{\partial x_i^2}$ is the distributional Laplacian and $\frac{\partial u}{\partial \nu}$ is the outer normal derivative of $u$, that is, the function $\frac{\partial u}{\partial \nu} =: h$ such that
\begin{equation}
\label{eq:l2-normal-derivative}
	\int_\Omega \nabla u \cdot \overline{\nabla v} + \Delta u \overline{v}\,\dx = \int_{\partial\Omega} h\overline{v}\,\textrm{d}\sigma
\end{equation}
for all $v \in H^1 (\Omega)$. One may show that any $u \in H^1 (\Omega)$ with $\Delta u \in L^2 (\Omega)$ has an outer normal derivative in $L^2 (\partial\Omega)$ in the sense of \eqref{eq:l2-normal-derivative}. For more details on this approach to the Robin Laplacian (for real $\alpha$), we again refer to \cite[Section~2]{ArendtMazzeo} or \cite[Section~1]{RohlederBook}. If $\alpha = 0$, then we write $-\Delta_\Omega^N$ in place of $-\Delta_\Omega^0$ for the operator associated with the form $a_0$, which we call the \emph{Neumann Laplacian}, and if $a_0$ is restricted to $H^1_0 (\Omega) \times H^1_0 (\Omega)$, then we call the associated operator the \emph{Dirichlet Laplacian}, which we denote by $-\Delta_\Omega^D$:
\begin{displaymath}
\begin{aligned}
	D(-\Delta_\Omega^D) &= \left\{ u \in H^1_0 (\Omega): \Delta u \in L^2 (\Omega) \right\},\\
	-\Delta_\Omega^D u &= -\Delta u.
\end{aligned}
\end{displaymath}
The following theorem is well known.

\begin{theorem}
\label{thm:dirichlet-and-neumann}
Let $\Omega \subset \R^d$, $d \geq 1$, be a bounded Lipschitz domain. The operators $-\Delta_\Omega^D$ and $-\Delta_\Omega^N$ are self-adjoint and semi-bounded from below in $L^2(\Omega)$. Their spectra $\sigma(-\Delta_\Omega^D) \subset (0,\infty)$, $\sigma (-\Delta_\Omega^N) \subset [0,\infty)$ are discrete, consisting only of eigenvalues of finite multiplicity, whose algebraic and geometric multiplicities always coincide, and with $+\infty$ as their only point of accumulation.
\end{theorem}

We now turn to the Robin Laplacian. The following lemma is key to establishing its properties.

\begin{lemma}
\label{lem:robin-operator}
Let $\Omega \subset \R^d$, $d \geq 1$, be a bounded Lipschitz domain and $\alpha \in \C$. The Robin form given by \eqref{eq:robin-form} is bounded in $H^1(\Omega)$ and sectorial of semi-angle $\theta$ for any $0 < \theta < \pi/2$.
\end{lemma}

\begin{proof}
By the trace theorem, $a_\alpha$ is well defined, bounded on $H^1 (\Omega) \times H^1 (\Omega)$
and sectorial.
\end{proof}

The following theorem now follows from Kato's first representation theorem \cite[Theorem~VI.2.1]{Kato}, the subsequent corollary \cite[Corollary ~VI.2.3]{Kato} and \cite[Theorem~VII.4.2]{Kato}, plus the fact that the form domain $H^1 (\Omega)$ is densely and compactly embedded in $L^2(\Omega)$ since $\Omega$ is a bounded Lipschitz domain.

\begin{theorem}
\label{thm:robin-operator}
Let $\Omega \subset \R^d$, $d \geq 1$, be a bounded Lipschitz domain and $\alpha \in \C$. The operator $-\Delta_\Omega^\alpha$ is:
\begin{enumerate}
\item semi-bounded from below in $L^2(\Omega)$,
\item locally uniformly (in $\alpha \in \C$) m-sectorial of semi-angle $\theta$ for any $0 < \theta < \pi/2$,
\item densely defined on $L^2(\Omega)$ and $H^1(\Omega)$, and 
\item its spectrum $\sigma(-\Delta_\Omega^\alpha)$ is discrete, consisting of eigenvalues of finite algebraic multiplicity, with their only point of accumulation being $\infty\in\overline{\C}$.
\end{enumerate}
Moreover, $-\Delta_\Omega^\alpha$ is self-adjoint if and only if $\alpha\in\R$. Finally, for any given $\alpha\in\R$, its eigenfunctions may be chosen to form an orthonormal basis of $L^2(\Omega)$.
\end{theorem}

For future reference, we state explicitly the weak form of the eigenvalue equation: $\lambda$ is an eigenvalue of the operator $-\Delta_\Omega^\alpha$, with eigenfunction $\psi$, if and only if
\begin{equation}
\label{eq:weak-eigenvalue-equation}
	a_\alpha [\psi,v] = \int_\Omega \nabla \psi \cdot \overline{\nabla v}\,\dx + \int_{\partial\Omega} \alpha \psi\, \overline{v}\,\dsigma(x)
	= \lambda \int_\Omega \psi\, \overline{v}\,\dx \qquad \text{for all } v \in H^1 (\Omega).
\end{equation}

We also briefly state for the record a result on the well-posedness of the associated parabolic equation. We will not need this here, so we do not go into any details.

\begin{theorem}
\label{thm:semigroup}
Let $\Omega \subset \R^d$, $d \geq 1$, be a bounded Lipschitz domain and $\alpha \in \C$. The operator $\Delta_\Omega^\alpha$ generates a holomorphic $C_0$-semigroup of operators of semi-angle $\theta$, for any $0 < \theta < \pi/2$.
\end{theorem}

For more on holomorphic semigroups, including their definition, see \cite[Chapter~3]{ABHN}.

\begin{proof}
This follows immediately from the resolvent estimate contained in the m-sectoriality assertion of Theorem~\ref{thm:robin-operator}, combined with Proposition~3.7.4 and Theorem~3.7.11 of \cite{ABHN}.
\end{proof}

Finally, we briefly summarise what happens if $\alpha$ is allowed to be a function instead of a constant.

\begin{remark}
\label{rem:operator-variable-alpha}
If $\alpha \in L^\infty (\partial\Omega,\C)$, then the form $a_\alpha$ may be defined in the same way and maintains its properties (in particular Lemma~\ref{lem:robin-operator}) due to the continued validity of the trace theorem and hence the estimate
\begin{equation}
\label{eq:variable-alpha-trace}
	\left| \int_{\partial\Omega} \alpha |u|^2\,\dsigma(x)\right| \leq 
	\|\alpha\|_{L^\infty(\partial\Omega)}\|u\|_{L^2(\partial\Omega)}^2
	\leq \varepsilon \|\nabla u\|_{L^2(\Omega)}^2 + C(\varepsilon,\|\alpha\|_{L^\infty(\partial\Omega)})\|u\|_{L^2(\Omega)}^2.
\end{equation}
It follows that Theorem~\ref{thm:robin-operator} holds with the obvious modifications that $-\Delta_\Omega^\alpha$ is self-adjoint if and only if $\alpha (x) \in \R$ for all $x \in \partial\Omega$, and that the local uniform sectoriality depends only on $\|\alpha\|_{L^\infty(\partial\Omega)}$, since for given semi-angle $\theta$ the vertex in the sectoriality estimate can be chosen in dependence only on the estimate given in \eqref{eq:variable-alpha-trace}. Theorem~\ref{thm:semigroup} then holds verbatim.
\end{remark}

\section{Dependence of the Robin eigenvalues and eigenfunctions on the parameter}
\label{sec:holomorphy}

In this section we wish to study the dependence of the eigenvalues of $-\Delta_\Omega^\alpha$, and the corresponding eigenprojections, on the parameter $\alpha \in \C$, for a fixed domain $\Omega$. We do this in two parts: firstly, we apply Kato's theory of \emph{holomorphic families of operators} to show that there is a family of eigencurves (as functions of $\alpha \in \C$), each of them analytic apart from at potential crossing points, which describe the totality of the spectrum for any fixed $\alpha$, and that the eigenprojections as operators on $L^2(\Omega)$ likewise depend analytically on $\alpha$, except at the crossing points. However, here caution is recommended: the \emph{normalised} eigenfunctions themselves do not change analytically: see Theorem~\ref{thm:normalised-ef-are-not-holomorphic}. Then, in Section~\ref{sec:derivative}, we obtain a formula for the derivative of an eigencurve with respect to $\alpha$, at any point where the corresponding eigenprojection is one-dimensional (i.e., the eigenvalue is simple).

\subsection{A holomorphic family of operators}
\label{sec:family}

As mentioned, we will start by applying Kato's theory, see \cite[Chapter~VII]{Kato}, to study the behaviour of the eigenvalues and eigenprojections of the Robin Laplacians $-\Delta_\Omega^\alpha$ in dependence on the parameter $\alpha \in \C$ (where, as before, $\Omega \subset \R^d$, $d\geq 1$, is a fixed bounded, Lipschitz domain); to emphasise this dependence and for ease of notation, in this section we will write
\begin{equation}
\label{eq:robin-operator-family-notation}
	\mathcal{A} (\alpha) := -\Delta_\Omega^\alpha.
\end{equation}
We first recall some more theory. For an isolated eigenvalue $\lambda$ of a linear operator $\mathcal{A}$ on a Hilbert space $H$, its \emph{eigenprojection} $Q_{\lambda}$ is defined as follows (see \cite[Section~III.6.5]{Kato}). Take a closed curve $\Gamma_{\lambda}\subset\rho(\mathcal{A})$ enclosing $\lambda$ but no other point of $\sigma(\mathcal{A})$ and define
\begin{equation}
\label{eq:eigen}
	Q_{\lambda}=-\frac{1}{2\pi\I}\oint_{\Gamma_{\lambda}}(\mathcal{A}-z I )^{-1}\,\rd z.
\end{equation}
Then $Q_{\lambda}$, which is independent of the choice of $\Gamma_{\lambda}$, is a projection onto the algebraic eigenspace of $\lambda$ in~$H$.

\begin{theorem}
\label{thm:holo}
Let $\Omega \subset \R^d$, $d\geq 1$, be a bounded, Lipschitz domain and let $\mathcal{A} (\alpha)$, $\alpha \in \C$, be given by~\eqref{eq:robin-operator-family-notation}.
\begin{enumerate}
\item The operator family $\mathcal{A}(\alpha)$, $\alpha\in \C$, is holomorphic and even self-adjoint holomorphic, i.e.\ $\mathcal{A}(\alpha)^*=\mathcal{A}(\overline{\alpha})$.
\item Each eigenvalue $\lambda_k (\alpha)$ can be extended to a meromorphic function with at most algebraic singularities at non-real crossing points of eigenvalues, and there are only finitely many eigenvalue curves meeting at locally finitely many crossing points. The same is true of the corresponding eigenprojections $Q_\lambda$ and eigennilpotents $(\mathcal{A}(\alpha)-\lambda(\alpha))Q_{\lambda(\alpha)}$.
\end{enumerate}
\end{theorem}

\begin{remark}
\label{rem:cont}
If two different eigenvalue curves $\lambda_1(\alpha)$ and $\lambda_2(\alpha)$ meet at $\lambda$ for $\alpha=\alpha_0$, i.e.\ $\lambda = \lambda_1(\alpha_0) = \lambda_2(\alpha_0)$, the corresponding separating curves $\Gamma_{\lambda_1(\alpha)}$, $\Gamma_{\lambda_2(\alpha)}$ in \eqref{eq:eigen} do not exist in the limit $\alpha\to\alpha_0$.
However, the holomorphic continuation of the total projection $\widehat{Q}_{\lambda}(\alpha):=Q_{\lambda_1(\alpha)}+Q_{\lambda_2(\alpha)}$ exists in $\alpha_0$ and is equal to the eigenprojection for $\lambda$ of $\mathcal{A}(\alpha_0)$. In addition, by \cite[Sections~VII.4.5,~II.2]{Kato}, the weighted eigenvalue mean $\widehat\lambda(\alpha):=\frac{1}{m}(m_1\lambda_1(\alpha)+m_2\lambda_2(\alpha))$ (with $m_j$ denoting the respective algebraic multiplicities, which are locally constant, and $m=m_1+m_2$ the multiplicity at $\alpha_0$) is holomorphic in $\alpha_0$.
A corresponding statement holds in the case of more than two curves meeting at $\lambda$, but in general the eigennilpotents may be discontinuous in $\alpha_0$.
\end{remark}

\begin{remark}
\label{rem:spectral-exhaustion}
Theorem~\ref{thm:holo} proves in particular parts (1) and (2) of Theorem~\ref{thm:analytic-dependence}. Let us briefly explain in particular how we obtain the fact that the extensions of the eigenvalues $\lambda_k (\alpha_0)$ for given $\alpha_0 \in \R$ exhaust the spectrum for any $\alpha \in \C$. Indeed, if there were some $\alpha \in \C$ and an eigenvalue $\lambda(\alpha)$ which did not lie on any of the eigencurves $\lambda_k (\alpha)$, then $\lambda (\alpha)$ could itself be extended to an analytic eigenvalue curve on $\C$ by Theorem~\ref{thm:holo}(2), and in particular we would have an eigenvalue $\lambda (\alpha_0)$ not included among the the $\lambda_k (\alpha_0)$, a contradiction to the assumption that $(\lambda_k (\alpha_0))_{k \in \N}$ (counting multiplicities) is the totality of the spectrum at $\alpha_0$.
\end{remark}

\begin{proof}[Proof of Theorem~\ref{thm:holo}]
(1) By \cite[Theorem~VII.4.2]{Kato}, $\mathcal{A}(\alpha)$, $\alpha\in\C$, is a holomorphic family of operators, and by \cite[Remark~VII.4.7]{Kato}, it is a self-adjoint holomorphic family. (2) Then it follows from \cite[Theorem~VII.1.8]{Kato} that the eigenvalues and eigenprojections depend (locally) holomorphically on $\alpha$, and hence so do the eigennilpotents.
Since the operator family is self-adjoint holomorphic, there are no singularities at real crossing points of eigenvalues, see \cite[Section~VII.3.1]{Kato}. 
The finiteness of the number of eigenvalue curves meeting at a crossing point, and of the local number of crossing points, follows from $\mathcal A(\alpha)$ having compact resolvent and from the holomorphy of the eigenvalue curves.
More precisely, since the total projection (see Remark~\ref{rem:cont}) is locally holomorphic, \cite[Problem~III.3.21]{Kato} implies that the dimension of its range is locally constant and thus finite. This also implies that if there were infinitely many crossing points in a compact set, then finitely many eigenvalue curves meet at infinitely many points which have an accumulation point; now the identity theorem implies that the eigenvalue curves have to be identical. 

It remains to prove that for any fixed $\alpha_0 \in \C$ each eigenvalue $\lambda_k (\alpha_0)$ can be extended to a function which is holomorphic on $\C$ except at the crossing points. We fix such a $\lambda_k (\alpha_0)$ and take an arbitrary compact subset $K\subset\C$ that is the closure of an open, connected set. It suffices to prove that if $K$ contains $\alpha_0$ in its interior, then there is a bounded holomorphic (except for crossing points) eigenvalue curve $\lambda_k (\alpha)$, $\alpha \in K$, which coincides with $\lambda_k (\alpha_0)$ at $\alpha = \alpha_0$.

To this end we consider the resolvent of $\A(\alpha)$ for $\alpha\in K$. We set $\rho_\Omega^K:=\bigcap_{\alpha\in K} \rho(-\Delta_\Omega^\alpha)$. Note that $\rho_\Omega^K\neq\emptyset$ since the operator family $\A(\alpha)$, $\alpha\in K,$ is uniformly sectorial, see Theorem~\ref{thm:robin-operator}(2). Fix $z\in\rho_\Omega^K$; then the resolvent family $R_z(\alpha)=(\A(\alpha)-zI)^{-1}$, $\alpha\in K$, is not only compact but bounded-holomorphic \cite[Theorem~VII.1.3]{Kato}. Thus, the point spectrum $\sigma_p(R_z(\alpha)) = \sigma (R_z (\alpha)) \setminus \{0\}$ consists of eigenvalues of finite algebraic multiplicity, and with $0$ as their only point of accumulation. Denote the eigenvalues of $R_z (\alpha_0)$ by $\mu_j (\alpha_0)$, where the ordering is chosen in such a way that
\begin{displaymath}
	\lambda_j (\alpha_0) = \frac{1}{\mu_j (\alpha_0)} + z
\end{displaymath}
for all $j$. Now the eigenvalue $\mu_k (\alpha_0)$ may be extended to a holomorphic eigenvalue curve, first to a neighbourhood of $\alpha_0$. 
This curve $\mu_k (\alpha)$ cannot take on the value $0$ for any $\alpha \in K$, since otherwise $R_z (\alpha)$ would not be invertible; hence its modulus has a nonzero minimum on any compact set. Together with the bounded-holomorphy of $\A(\alpha)$, $\alpha\in K,$ we obtain that $\mu_k (\alpha)$ can be extended holomorphically to all of $K$ except at only finitely many crossing points with other eigenvalue curves.
Via the identification $\lambda_k (\alpha) = 1/\mu_k (\alpha) + z$ we obtain that $\lambda_k (\alpha)$ is well defined and holomorphic on all of $K$ except at the crossing points.  Since $\alpha_0 \in \C$ and $k$ were arbitrary, this completes the proof.
\end{proof}

Even though Theorem~\ref{thm:holo} establishes that the eigenprojections can be continued holomorphically (away from possible crossing points), the eigenfunctions lose this property when normalised to have $L^2(\Omega)$-norm one:

\begin{theorem}
\label{thm:normalised-ef-are-not-holomorphic}
Let $H$ be a separable Hilbert space and let $D\subset\C$ be an open, connected set. Let $A(\alpha)$ be an operator family on $H$ such that its eigenfunctions $u(\alpha)$ depend holomorphically on $\alpha\in D$. Then the norm $\|u(\alpha)\|_H$ is non-constant on $D$ or u does not depend on $\alpha\in D$.
\end{theorem}

\begin{proof}
Let $\alpha\in D$, assume the family of normalised eigenfunctions $u(\alpha)$ of $\A(\alpha)$ to be holomorphic and fix an arbitrary $\alpha_0\in D$. Then, the function $f:D\to\C$ defined by $f(\alpha)=(u(\alpha_0),u(\alpha))$ satisfies 
\begin{equation*}
|f(\alpha)|\leq \|u(\alpha_0)\|_H \|u(\alpha)\|_H = 1,
\end{equation*}
that is, $f$ is contractive on $D$. Now, since $f(\alpha_0)=1$, the maximum principle yields that $|f|\equiv 1$ is constant and by $f(\alpha_0)=1$ we conclude $f\equiv 1$. Furthermore, for any $\alpha\in D$ we have
\begin{align*}
\|u(\alpha)-u(\alpha_0)\|_H^2&=(u(\alpha)-u(\alpha_0),u(\alpha)-u(\alpha_0)) \\
	&= \|u(\alpha)\|_H^2 + \|u(\alpha_0)\|_H^2 -2\Re (u(\alpha_0),u(\alpha)) = 0.
\end{align*}
Consequently, $u(\alpha)= u(\alpha_0)$ and the family of eigenfunctions is independent of $\alpha$, a contradiction.
\end{proof}

The question whether the eigenfunctions of $-\Delta_{(-a,a)}^\alpha$ are orthogonal in $L^2((-a,a))$ will be clarified in Section \ref{sec:eigenfunctions}.

\begin{remark}
\label{rem:eigennilpotents}
In the case of the domains where one can describe the eigenvalues explicitly (that is, as solutions of transcendental equations), namely intervals, balls and (hyper-) rectangles, it is possible to show that the eigennilpotents are always zero; see Remark~\ref{rem:hyperrectangle-riesz} for the case of hyperrectangles and 
Remark~\ref{rem:ball-riesz} 
for the case of the ball. It thus seems reasonable to expect that the eigennilpotents are zero on any Lipschitz domain.
\end{remark}

\begin{remark}
\label{rem:eigennilpotents-not-zero}
However, it is easy to see that there can be nontrivial eigennilpotents if $\alpha$ is allowed to be a function on the boundary. Take the simplest possible case of an interval $\Omega=(-a,a)$ and suppose $\alpha:\{-a,a\}\to\C$ is a function. Then for some values of $\alpha$ the eigennilpotents are non-zero: indeed, following \cite[Section 3]{KBZ}, we let $t\in\R$ and consider purely imaginary $\alpha_t(x)$ of the form
\begin{displaymath}
\alpha_t(x)=
\begin{cases}
-\I t &\text{for}\; x=-a,\\
+\I t &\text{for}\; x=+a.
\end{cases}
\end{displaymath}
Then the spectrum of the Robin Laplacian $\A(\alpha_t)=-\Delta_{(-a,a)}^{\alpha_t}$ reads $\sigma\left(\A(\alpha_t)\right)=\left\{t^2\right\}\cup\left\{k_j^2\right\}_{j\in\N}$, where $k_j:=\frac{\pi j}{2a}$. That is, the spectrum consists of the eigenvalues of the Neumann Laplacian independently of $t$, plus the eigenvalue $t^2$. This eigenvalue has eigenfunction $u_0(x)=\e^{-\I t x}$, while the rest of the eigenfunctions for $k_j^2$ read 
\begin{displaymath}
u_j(x)=\cos(k_j x)-\I \left(\frac{\I t}{k_j}\right)^{(-1)^j}\sin(k_j x).
\end{displaymath}
Note that each $u_j$, $j\geq 1$, is, like its eigenvalue, independent of $t$. Fix $j \in \N$. The eigenvalue curves $t^2$ and $k_j^2$ obviously cross at $t=k_j$, meaning that the algebraic multiplicity at this point should be two; however, the eigenfunctions $u_0$ and $u_j$ converge to the same function as $t\to k_j$. It can be checked that for the function
\begin{displaymath}
g(x)= \frac{\I}{2}x\e^{-\I t x} - \frac{\e^{2\I ta}}{4t^2}\e^{\I t x}
\end{displaymath}
when $t=k_j$ the corresponding eigennilpotent satisfies 
\begin{displaymath}
\left(\A(\alpha_t)-t^2\right) g = u_0 \neq 0 \qquad\text{and} \qquad
\left(\A(\alpha_t)-t^2\right)^2 g =0.
\end{displaymath}
Consequently, $g$ is a root vector and the geometric and algebraic eigenspaces do not coincide at $t=k_j$.
\end{remark}

It would take us too far afield to explore the question of the eigennilpotents here and so we leave it as an open problem to investigate them in the case that $\alpha$ is independent of $x \in \partial\Omega$.

\begin{oproblem}
\label{prob:eigennilpotents}
Given any bounded Lipschitz domain $\Omega \subset \R^d$, suppose that $\lambda = \lambda (\alpha)$ is a repeated eigenvalue of $\mathcal{A}(\alpha) = -\Delta_\Omega^\alpha$ for some $\alpha \in \C$. Is the eigennilpotent $(\mathcal{A}(\alpha) - \lambda(\alpha)) Q_{\lambda(\alpha)}$ necessarily equal to zero (where $Q_{\lambda(\alpha)}$ is the eigenprojection, see \eqref{eq:eigen})?
\end{oproblem}

\subsection{The derivative with respect to $\alpha$}
\label{sec:derivative}

We now give a formula for the derivative of a simple eigenvalue $\lambda$ with respect to $\alpha \in \C$ (that is, along its corresponding eigencurve), which by Theorem~\ref{thm:holo} always exists. For $\alpha \in \R$ the corresponding formula is reasonably well known (especially but not only in the special case $\alpha = 0$); see \cite[Section~4.3.2]{BFK} and the references therein.

\begin{theorem}
\label{thm:derivative}
Let $\Omega \subset \R^d$, $d\geq 1$, be a bounded, Lipschitz domain, let $\alpha_0 \in \C$, and let $\lambda=\lambda(\alpha)$ be any meromorphic family of eigenvalues. Suppose that for all $\alpha$ in some neighbourhood $B_\delta (\alpha_0)$ of $\alpha_0$, $\lambda (\alpha)$ is a simple eigenvalue of $-\Delta_\Omega^\alpha$, with eigenfunction $\psi(\alpha)$ which is chosen to be holomorphic in $\alpha$. Then in a neighbourhood of $\alpha$ the function
\begin{equation}
\label{eq:deriv-frac}
	\alpha \mapsto \frac{\int_{\partial\Omega} \psi(\alpha)^2\,\dsigma(x)}{\int_\Omega \psi(\alpha)^2\,\dx}
\end{equation}
is meromorphic with at most removable singularities. Its holomorphic continuation is equal to $\lambda'(\alpha)$ at every point in $B_\delta (\alpha_0)$.
\end{theorem}

This justifies writing simply
\begin{equation}
\label{eq:derivative}
	\lambda'(\alpha) = \frac{\int_{\partial\Omega} \psi(\alpha)^2\,\dsigma(x)}{\int_\Omega \psi(\alpha)^2\,\dx}
	\qquad \text{for all } \alpha \in B_\delta (\alpha_0),
\end{equation}
and in particular Theorem~\ref{thm:derivative} implies Theorem~\ref{thm:analytic-dependence}(3).

We leave it as an open problem to determine whether the mapping \eqref{eq:deriv-frac} can actually have (removable) singularities, or whether the denominator never vanishes.

\begin{oproblem}
Let $\lambda(\alpha)$ be any simple eigenvalue of $-\Delta_\Omega^\alpha$ for some $\Omega \subset \R^d$ bounded and Lipschitz and $\alpha \in \C$, and denote by $\psi(\alpha)$ its eigenfunction, scaled arbitrarily. Does it follow that
\begin{displaymath}
	\int_\Omega \psi(\alpha)^2\,\dx \neq 0 \ ?
\end{displaymath}
\end{oproblem}

We first prove that under the assumptions of the theorem the derivative of the eigenfunction $\psi$ with respect to $\alpha$, which we denote by $\psi'(\alpha)$ (and which exists as an element of $L^2(\Omega)$ by another application of Theorem~\ref{thm:holo}) is actually in $H^1(\Omega)$. Notationally, we will take $z \in \C$ to be small enough that $\alpha + z \in B_\delta (\alpha_0)$, that is, $|\alpha + z - \alpha_0| < \delta$.

\begin{lemma}
\label{lem:ef-h1}
Under the assumptions of Theorem~\ref{thm:derivative}, we have $\psi'(\alpha) \in H^1(\Omega)$ for all $\alpha \in B_\delta (\alpha_0)$.
\end{lemma}

\begin{proof}
We will show that
\begin{equation}
\label{eq:ef-h1-gradient-est}
	\limsup_{z \to 0} \frac{\|\nabla \psi(\alpha+z)-\nabla \psi(\alpha)\|_2^2}{|z|^2} < \infty.
\end{equation}
Since we already know that $\nabla \psi'(\alpha)$ exists in the distributional sense (as $\psi'(\alpha) \in L^2(\Omega)$), it will then follow from \eqref{eq:ef-h1-gradient-est} that actually $\nabla \psi'(\alpha) \in L^2 (\Omega)$.

To prove \eqref{eq:ef-h1-gradient-est}, we fix $z \in \C$ sufficiently small (as explained above) and use the weak form of the equation for both $\lambda(\alpha+z)$ and $\lambda(\alpha)$ to obtain (with $(\cdot,\cdot)$ the inner product on $L^2(\Omega)$)
\begin{displaymath}
\begin{split}
	&\|\nabla(\psi(\alpha+z)-\psi(\alpha))\|_2^2\\
	&\qquad=\Re \int_\Omega (\nabla \psi(\alpha+z)-\nabla \psi(\alpha))\cdot \overline{(\nabla \psi(\alpha+z)-\nabla\psi(\alpha))}\,\dx\\
	&\qquad=\Re \left[\lambda (\alpha+z) \big( \psi(\alpha+z),\psi(\alpha+z)-\psi(\alpha) \big)\right] - \Re \left[(\alpha+z)\int_{\partial\Omega}
		\psi(\alpha+z)\overline{(\psi(\alpha+z)-\psi(\alpha))}\,\dsigma(x)\right]\\
	&\qquad\qquad -\Re \left[\lambda (\alpha) \big( \psi(\alpha),\psi(\alpha+z)-\psi(\alpha) \big)\right] + \Re \left[\alpha\int_{\partial\Omega}
		\psi(\alpha)\overline{(\psi(\alpha+z)-\psi(\alpha))}\,\dsigma(x)\right]\\
	&\qquad=\Re \big( (\lambda(\alpha+z)\psi(\alpha+z)-\lambda(\alpha)\psi(\alpha)),\psi(\alpha+z)-\psi(\alpha) \big)\\
	&\qquad\qquad -\Re\int_{\partial\Omega}((\alpha+z)\psi(\alpha+z)-\alpha\psi(\alpha))\overline{(\psi(\alpha+z)-\psi(\alpha))}\,\dsigma(x).
\end{split}
\end{displaymath}
We next estimate the integrand in the boundary integral as follows:
\begin{displaymath}
\begin{split}
	&-\Re \Big[((\alpha+z)\psi(\alpha+z)-\alpha\psi(\alpha))\overline{(\psi(\alpha+z)-\psi(\alpha))}\Big]\\
	&\qquad = -\Re (\alpha+z)|\psi(\alpha+z)-\psi(\alpha)|^2 + \Re \left[z\psi(\alpha)\overline{(\psi(\alpha+z)-\psi(\alpha))}\right]\\
	&\qquad \leq -\Re (\alpha+z)|\psi(\alpha+z)-\psi(\alpha)|^2 + \frac{1}{2}|\psi(\alpha+z)-\psi(\alpha)|^2 + \frac{|z|^2}{2}|\psi(\alpha)|^2.
\end{split}
\end{displaymath}
Applying the trace inequality in the form
\begin{displaymath}
	\int_{\partial\Omega} |u|^2\,\dsigma(x) \leq \varepsilon \|\nabla u\|^2 + C_\varepsilon \|u\|^2
\end{displaymath}
for all $u \in H^1(\Omega)$, where $C_\varepsilon>0$ depends only on $\varepsilon>0$, to each of the two integrals
\begin{displaymath}
	\left|-\Re(\alpha+z) + \frac{1}{2}\right|\int_{\partial\Omega} |\psi(\alpha+z)-\psi(\alpha)|^2\,\dsigma(x) \quad \text{and} \quad
	\frac{|z|^2}{2}\int_{\partial\Omega} |\psi(\alpha)|^2 \,\dx,
\end{displaymath}
and choosing $\varepsilon > 0$ small enough that $\eta := \varepsilon [-\Re(\alpha+z) + \frac{1}{2}] < 1$ leads us to
\begin{displaymath}
\begin{split}
	&\|\nabla(\psi(\alpha+z)-\psi(\alpha))\|_2^2\\
	&\qquad\leq \Re \big( (\lambda(\alpha+z)\psi(\alpha+z)-\lambda(\alpha)\psi(\alpha)),\psi(\alpha+z)-\psi(\alpha) \big)\\
	&\qquad\quad +\eta \|\nabla(\psi(\alpha+z)-\psi(\alpha))\|_2^2 + C_\varepsilon
		\left(-\Re(\alpha+z) + \frac{1}{2}\right) \|\psi(\alpha+z)-\psi(\alpha)\|_2^2\\
	&\qquad\quad + |z|^2\left(\frac{\varepsilon}{2}\|\nabla\psi(\alpha)\|_2^2 + \frac{C_\varepsilon}{2}\|\psi(\alpha)\|_2^2\right).
\end{split}
\end{displaymath}
Now $\varepsilon$ may be chosen independently of $\alpha \in B_\delta (\alpha_0)$; in particular, with such a choice, the coefficient of $|z|^2$ depends only on $\alpha$, that is, we may write
\begin{displaymath}
	C_\alpha := \frac{\varepsilon}{2}\|\nabla\psi(\alpha)\|_2^2 + \frac{C_\varepsilon}{2}\|\psi(\alpha)\|_2^2
\end{displaymath}
for this coefficient. We now divide by $|z|^2$ and pass to the limit as $z \to 0$ to obtain
\begin{displaymath}
\begin{split}
	\limsup_{z \to 0} \frac{\|\nabla(\psi(\alpha+z)-\psi(\alpha))\|^2}{|z|^2}
	&\leq \frac{1}{1-\eta}\Re ( \lambda'(\alpha)\psi(\alpha)+\lambda(\alpha)\psi'(\alpha),\psi'(\alpha) )\\
	&\qquad+\frac{1}{1-\eta}C_\varepsilon \left(-\Re \alpha + \frac{1}{2}\right)\|\psi'(\alpha)\|^2 + C_\alpha.
\end{split}
\end{displaymath}
Since we already know that $\psi'(\alpha) \in L^2 (\Omega)$, the right-hand side of the above inequality is finite. This establishes \eqref{eq:ef-h1-gradient-est} and hence completes the proof of the lemma.
\end{proof}

\begin{proof}[Proof of Theorem~\ref{thm:derivative} and hence of Theorem~\ref{thm:analytic-dependence}(3)]
We choose $\psi(\alpha) \in H^1 (\Omega)$ as a test function in the weak form of the eigenvalue equation for $\lambda(\alpha)$:
\begin{displaymath}
	\int_{\Omega}(\nabla\psi(\alpha))^2\,\dx+\alpha\int_{\partial\Omega}\psi(\alpha)^2\,\dsigma(x)-\lambda(\alpha)\int_{\Omega}\psi(\alpha)^2\,\dx=0.
\end{displaymath}
The left-hand side clearly depends holomorphically on $\alpha$. Moreover, since $\psi'(\alpha) \in H^1 (\Omega)$ by Lemma~\ref{lem:ef-h1}, we may calculate its derivative as
\begin{displaymath}
\begin{split}
	2\int_{\Omega}\nabla\psi'(\alpha)\cdot\nabla\psi(\alpha)\,\dx+ &\int_{\partial\Omega}\psi(\alpha)^2\,\dsigma(x) + 
	 2\alpha \int_{\partial\Omega}\psi'(\alpha)\psi(\alpha)\,\dsigma(x)\\ &\qquad -\lambda'(\alpha)\int_{\Omega}\psi(\alpha)^2\,\dx
	-2\lambda(\alpha)\int_{\Omega}\psi'(\alpha)\psi(\alpha)\,\dx = 0.
\end{split}
\end{displaymath}
But the weak form of the eigenvalue equation for $\lambda (\alpha)$ also implies that
\begin{displaymath}
	2\int_{\Omega}\nabla\psi'(\alpha)\cdot\nabla\psi(\alpha)\,\dx+2\alpha \int_{\partial\Omega}\psi'(\alpha)\psi(\alpha)\,\dsigma(x)
	= 2\lambda(\alpha)\int_{\Omega}\psi'(\alpha)\psi(\alpha)\,\dx,
\end{displaymath}
whence
\begin{equation}
\label{eq:almost-deriv}
	\lambda'(\alpha)\int_{\Omega}\psi(\alpha)^2\,\dx = \int_{\partial\Omega}\psi(\alpha)^2\,\dsigma(x).
\end{equation}
This yields \eqref{eq:derivative} in the case that $\int_{\Omega}\psi(\alpha)^2\,\dx \neq 0$. But since we know that $\lambda'(\alpha)$ is holomorphic in $B_\delta (\alpha_0)$, as are the mappings
\begin{displaymath}
	\alpha \mapsto \int_\Omega \psi(\alpha)^2\,\dx, \qquad \alpha \mapsto \int_{\partial\Omega}\psi(\alpha)^2\,\dsigma(x),
\end{displaymath}
if the left-hand side of \eqref{eq:almost-deriv} vanishes at some point, then the right-hand side must vanish as well, and to the same order. It follows that any singularities of the mapping \eqref{eq:deriv-frac} in $B_\delta (\alpha_0)$ are removable. This completes the proof of the theorem.
\end{proof}


\section{Basis properties of the eigenfunctions}
\label{sec:eigenfunctions}

Given the analytic dependence of the eigenfunctions $\{e_k(\alpha)\}_{k\geq 1}$ of the Robin Laplacian on $\alpha \in \C$, it is a natural question to ask whether they also still have reasonable basis properties for non-real $\alpha$. In this section we will explore this question and, in particular, prove parts (4) and (5) of Theorem~\ref{thm:analytic-dependence}. 

We start with the negative result (5), that the eigenfunctions do not generally form an orthonormal basis.

\begin{theorem}\label{thm:NoONB}
Let $\Omega\in\R^d$, $d\geq 1$, be a bounded Lipschitz domain and $\alpha\in\C$. Then the eigenfunctions $e_k(\alpha)$, $k\in\N$, of $-\Delta_\Omega^\alpha$ can be chosen to form an orthonormal basis of $L^2(\Omega)$ if and only if $\alpha\in\R$.
\end{theorem}

\begin{proof}
For ease of notation, in this section we will write $\A(\alpha):=-\Delta_\Omega^\alpha$. For $\alpha\in\R$ the claim follows from the selfadjointness of $\A(\alpha)$.

Let $\alpha\in\C\setminus\R$ and assume that the eigenfunctions $\{e_k(\alpha)\}_{k=1}^\infty$ of $\A(\alpha)$ do form an orthonormal basis of $L^2(\Omega)$. To distinguish the notation from the complex conjugation $\overline{z}$ of $z\in\C$ and $M^\ast$ of $M\subset\C$, let $\mathrm{cl}(M)$ be the closure of $M$. Let $u\in D\left(\A(\alpha)^\ast\right)=D(\A(\overline{\alpha}))\subset L^2(\Omega)$ have $L^2(\Omega)$-norm one. Then there is a unique representation of $u$,
\begin{displaymath}
u=\sum_{k=1}^\infty \left(u,e_k(\alpha)\right)e_k(\alpha),
	\qquad
	1=\|u\|_2^2 = \sum_{k=1}^\infty |(u,e_k(\alpha))|^2,
\end{displaymath} 
which we use to calculate
\begin{align}
\overline{\left(u,\A(\alpha)^\ast u\right)} 
	&=\left(\A(\alpha)^\ast u,u\right)  \notag \\
	&= \sum_{k=1}^\infty \overline{\left(u,e_k(\alpha)\right)}\left((-\Delta_\Omega^\alpha)^\ast u,e_k(\alpha)\right)\notag \\
	&= \sum_{k=1}^\infty \left(e_k(\alpha),u\right)\left( u,\lambda_k(\alpha) e_k(\alpha)\right)
	 = \sum_{k=1}^\infty \left|\left(u,e_k(\alpha)\right)\right|^2\overline{\lambda_k(\alpha)}. \label{eq:convex-combination-eigenvalues}
\end{align} 
By the definition of the numerical range \eqref{eq:form-numerical-range} and the identity $\mathrm{cl}(W(a_\alpha))= \mathrm{cl}(W(\A(\alpha)))$ \cite[Corollary~VI.2.3]{Kato}
we obtain
\begin{equation}
\label{eq:numerical-ranges-complex-conjugated}
\mathrm{cl}(W(\A(\alpha)^\ast))
	=\mathrm{cl}(W(a_\alpha^\ast))
		=\mathrm{cl}(W(a_\alpha))^\ast
			=\mathrm{cl}(W(\A(\alpha)))^\ast.
\end{equation}
Note that due to the normalisation of $u$ the right-hand side of \eqref{eq:convex-combination-eigenvalues} can be interpreted as a convex combination of the complex conjugated elements $\lambda_k(\alpha)\in\sigma(\A(\alpha))$; the convex hull of the whole spectrum will be denoted by
\begin{displaymath}
\conv\left(\sigma(\A(\alpha))\right)=\conv \left\{ \lambda_k(\alpha) : k\in\mathbb{N} \right\}.
\end{displaymath}
Due to \eqref{eq:numerical-ranges-complex-conjugated} and \eqref{eq:convex-combination-eigenvalues} we obtain
\begin{displaymath}
\mathrm{cl}(W(\A(\alpha)))^\ast = \mathrm{cl}(W(\A(\alpha)^\ast)) = \mathrm{cl}\left(\conv\left(\sigma(\A(\alpha))\right)^\ast\right)
\end{displaymath}
and by complex conjugation of both sides we arrive at
\begin{equation}\label{eq:numerical-range-equal-convex-hull}
\mathrm{cl}(W(a_\alpha))=\mathrm{cl}(W(\A(\alpha))) = \mathrm{cl}\left(\conv(\sigma(\A(\alpha)))\right).
\end{equation}
This equation leads us to a contradiction as follows. Since $\A(\alpha)$ is sectorial and its resolvent is compact, for any sufficiently large $r>0$ the \textit{truncated} convex hull
\begin{displaymath}
P_r(\alpha):=\conv\left\{\lambda_k(\alpha) : k\in\N,\; |\lambda_k(\alpha)|\leq r  \right\} \subset\C
\end{displaymath}
is a polygon which contains at most finitely many eigenvalues of $\A(\alpha)$. To show that $P_r(\alpha)$ is contained in the upper half-plane it is sufficient to prove $\Im\lambda_k(\alpha)>0$ for all $k\in\N$: due to
\begin{displaymath}
\sigma(A(\alpha))\subset W(a_\alpha)\subset \{z\in\C : \Im z\geq 0\}
\end{displaymath}
it is clear that $\Im\lambda_k(\alpha)\geq 0$ for all $k\in\N$. Now assume that there exists an eigenvalue $\lambda\in P_r(\alpha)\cap\R$, that is, we find a corresponding (normalised) eigenfunction $u\in D(\A(\alpha))$ such that
\begin{displaymath}
\lambda=(\A(\alpha)u,u) = \int_\Omega|\nabla u|^2\dx + \alpha \int_{\partial\Omega} |u|^2 \dsigma(x) \in\R.
\end{displaymath}
This holds if and only if $u|_{\partial\Omega}=0$, that is, $u$ is an eigenfunction of the Dirichlet Laplacian $\A^D:=-\Delta_\Omega^D$. Furthermore, $u\in D(\A(\alpha))$ yields
\begin{displaymath}
0=\partial_\nu u + \alpha u = \partial_\nu u
\end{displaymath}
on $\partial\Omega$ and $u$ is additionally a Neumann eigenfunction, a contradiction. In other words, we have shown that 
\begin{equation}\label{eq:numerical-range-upper-half-plane}
\mathrm{cl}\left(\conv(\sigma(\A(\alpha)))\right)\cap\R=\emptyset.
\end{equation}
Since the principal eigenvalue $\lambda_1=\min\sigma(\A^D)$ can be represented by the variational max-min characterisation, there exists a normalised minimising function (the associated eigenfunction) $u_1\in H_0^1(\Omega)$ such that $\lambda_1 =\int_\Omega |\nabla u_1|^2\dx$. Recall that the domains of the Dirichlet Laplacian and the Robin form satisfy 
\begin{displaymath}
D(\A^D)\subset H_0^1(\Omega)\subset H^1(\Omega)=D(a_\alpha),
\end{displaymath}
see Section \ref{sec:robin-operator}. Consequently, $\lambda_1\in \mathrm{cl}(W(a_\alpha))$, a contradiction to \eqref{eq:numerical-range-equal-convex-hull} and \eqref{eq:numerical-range-upper-half-plane}.
Hence  the eigenfunctions $\{e_k(\alpha)\}_{k=1}^\infty$ of $\A(\alpha)$ do not form an orthonormal basis of $L^2(\Omega)$.
\end{proof}

One may show by explicit calculation that, even on the interval $\Omega = (-a,a)$, consistent with Theorem~\ref{thm:NoONB}, for given $\alpha \in \C \setminus \R$ the eigenfunctions of the Robin Laplacian $-\Delta_\Omega^\alpha$ belonging to different eigenspaces are not in general orthogonal to each other. Hence, for our positive result, we necessarily need to introduce ``weaker'' notions of basis. Here we will consider three: Bari, Riesz and Abel bases. The definitions of the first two of these, namely Definitions \ref{def:riesz-basis} and \ref{def:bari-basis}, are taken from \cite[3.6.16-19]{Istratescu}. For what follows we assume $(H,\|\cdot\|_H)$ to be a separable complex Hilbert space.

\begin{definition}
\label{def:basis}
A set $\mathcal{B}=\{e_k\}_{k=1}^\infty\subset H$ is called a \emph{basis} of $H$ if for each $h\in H$ there exists a unique, convergent series representation $h=\sum_{k=1}^\infty h_k e_k$ with coefficients $h_k=h_k(h)\in\mathbb{C}$.
\end{definition}

\begin{definition}
\label{def:riesz-basis}
Let $\mathcal{B}=\{e_k\}_{k=1}^\infty$ be a basis of $H$. Then $\mathcal{B}$ is called a \emph{Riesz basis} if there are constants $0<m\leq M$ such that 
\begin{displaymath}
m \|(h_k)_k\|_{\ell^2} \leq \|h\|_H \leq M \|(h_k)_k\|_{\ell^2}
\end{displaymath}
holds for any $h=\sum_{k=1}^\infty h_k e_k\in H$.
\end{definition}

\begin{definition}
\label{def:bari-basis}
A set $\mathcal{B}=\{e_k\}_{k=1}^\infty\subset H$ is called a \emph{Bari basis} of $H$ if there exists an orthonormal basis $\mathcal{B}'=\{e_k'\}_{k=1}^\infty$ of $H$ such that $\mathcal{B}$ is quadratically near $\mathcal{B}'$, that is
\begin{displaymath}
\sum_{k=1}^\infty \| e_k-e_k'\|_H^2 <\infty.
\end{displaymath}
\end{definition}


An Abel basis, as first introduced in \cite{Lidskii} and also defined for example in \cite[Section~1.2.13]{YaYa}, is always defined with respect to the eigenvectors and generalised eigenvectors (for short, generalised eigenvectors) of a densely defined sectorial operator $A$. The intuitive idea is that the formal series expansion
\begin{displaymath}
	\sum_{k=1}^\infty h_k e_k
\end{displaymath}
of an element $h\in H$ in the generalised eigenvectors $e_k$ of $A$ may not converge, but if the Fourier coefficients $h_k$ can be multiplied by a weight $\e^{-\lambda_k^\gamma t}$ (where $\lambda_k$ is the eigenvalue corresponding to $e_k$), such that
\begin{displaymath}
	\sum_{k=1}^\infty h_k \e^{-\lambda_k^\gamma t} e_k
\end{displaymath}
converges for each fixed $t>0$, and this series then converges to $h$ as $t \to 0$, then $\{e_k\}_{k=1}^\infty$ is an Abel basis of order $\gamma \geq 0$. (Note that an Abel basis will not generally be a basis in the sense of Definition~\ref{def:basis}, since the unweighted Fourier series expansion is explicitly not required to converge.) We will follow the definition given in \cite{YaYa}.

\begin{definition}
\label{def:abel-basis}
Suppose $A: H \supset D(A) \to H$ is a densely defined operator with purely discrete spectrum, such that all but finitely many of its eigenvalues lie in the sector $T_\theta^+ = \{ z \in \C: |\arg z| < \theta \}$ for some $\theta \in (0,\pi)$. Then we say that the generalised eigenvectors of $A$ form an \emph{Abel basis of $H$ of order $\gamma \geq 0$} if $\gamma \theta < \pi/2$ and there exists an enumeration of the eigenvalues $\{\lambda_k\}_{k=1}^\infty$ (with $\{e_k\}_{k=1}^\infty$ the corresponding enumeration of the generalised eigenvectors) such that for this fixed enumeration, for each $h \in H$, there exists a sequence of coefficients $h_k \in \C$ for which the series
\begin{equation}
\label{eq:ht}
	h(t) := \sum_{k=1}^\infty h_k \e^{-\lambda_k^\gamma t} e_k
\end{equation}
is convergent for all $t>0$, and $h(t) \to h$ in $H$ as $t \to 0^+$.

For the eigenvalues $\lambda$ which do not lie in $T_\theta^+$, the weight $\e^{-\lambda_k^\gamma t}$ in \eqref{eq:ht} is to be interpreted as $1$, while if some $\lambda$ is a repeated eigenvalue and $\{e_j,\ldots,e_{j+\ell}\}$ is a basis of its eigenspace, then the corresponding terms in the series \eqref{eq:ht} are to be interpreted in terms of the eigenprojection, that is, $\sum_{k=j}^{j+\ell} h_k \e^{-\lambda^\gamma t} e_k$ is to be replaced by
\begin{displaymath}
	\frac{1}{2\pi\I} \oint_\Gamma \e^{-\lambda^\gamma t} (A - zI)^{-1} h\,\textrm{d}z,
\end{displaymath}
where $\Gamma$ is any closed path in $\C$ separating $\lambda$ from the rest of the spectrum.
\end{definition}

The definition can be extended to allow $\gamma\theta < \pi$ in place of $\gamma\theta < \pi/2$; we refer, again, to \cite[Section~1.2.13]{YaYa}.

One may derive from the definitions that an orthonormal basis is always a Bari basis, a Bari basis is always a Riesz basis, and a Riesz basis, if it consists of the generalised eigenfunctions of a suitable operator, is always an Abel basis of order zero. The latter, in turn, is an Abel basis of any positive order $\gamma > 0$, provided only that the sectoriality estimate $\gamma\theta < \pi$ still holds.

Our goal is to show that the eigenfunctions of $-\Delta_\Omega^\alpha$ form (at least) an Abel basis of $L^2(\Omega)$, for any $\alpha\in\C$. This is based on a theorem of Agranovich (the main theorem of \cite{Agranovich94}), which we recall here for ease of reference.

\begin{theorem}
\label{thm:agranovich}
Suppose $H$ and $V$ are separable complex Hilbert spaces such that $V \hookrightarrow H$ is compact, and suppose that $a: V \times V \to \C$ is a bounded, coercive sesquilinear form. Denote by $b:= \Re a = (a + \overline{a})/2$ and $c: = \I\Im a = a - b$ the real and imaginary forms, respectively, which add to give $a$. Denote by $A$ and $B$ the operators on $H$ associated with $a$ and $b$, respectively. Suppose that
\begin{enumerate}
\item[(i)] there exist $0 \leq q \leq 1$ and $m>0$ such that
\begin{equation}
\label{eq:agranovich-trace}
	|c [u,u] | \leq m\|B^{1/2}u\|_H^{2q} \|u\|_H^{2-2q}
\end{equation}
for all $u \in V$, and
\item[(ii)] there exists $p>0$ such that the sequence of eigenvalues $\lambda_k (B)$, $k\geq 1$, of $B$ (bounded from below by assumption), repeated according to their multiplicities, has the asymptotic behaviour
\begin{equation}
\label{eq:agranovich-asymptotics}
	\limsup_{k\to\infty} \frac{\lambda_k(B)}{k^p} > 0.
\end{equation}
\end{enumerate}
Then $A$ has discrete spectrum, the invariant subspaces of $A$ are all finite dimensional, and the corresponding eigenfunctions and generalised eigenfunctions of $A$ constitute
\begin{enumerate}
\item a Bari basis of $H$ if $p(1-q)>1$, or
\item a Riesz basis of $H$ if $p(1-q)=1$, or
\item an Abel basis of $H$, of order $1/p + (q-1) + \delta$ for any (sufficiently small) $\delta>0$, if $p(1-q)<1$.
\end{enumerate}
\end{theorem}

Theorem~\ref{thm:agranovich} was already used for a similar purpose in \cite[Section~5]{HKS} to prove a corresponding one-dimensional result; more precisely, for the Laplacian on a compact metric graph, equipped with complex $\delta$ conditions at one or more of the vertices (corresponding to a complex Robin condition), one can apply (2) to obtain a Riesz basis. With this background, we can now state our main positive result, which corresponds to Theorem~\ref{thm:analytic-dependence}(4).

\begin{theorem}
\label{thm:abel-basis}
Let $\Omega \subset \R^d$, $d\geq 1$, be a bounded Lipschitz domain and $\alpha \in \C$.
\begin{enumerate}
\item If $d=1$, then there is a Riesz basis of $L^2(\Omega)$ consisting of the eigenfunctions and generalised eigenfunctions of $-\Delta_\Omega^\alpha$;
\item If $d\geq 2$, then there is an Abel basis of $L^2(\Omega)$ of order $(d-1)/2 + \delta$ for any (sufficiently small) $\delta>0$, consisting of the eigenfunctions and generalised eigenfunctions of $-\Delta_\Omega^\alpha$.
\end{enumerate}
\end{theorem}

We leave it as an open problem to determine whether in fact the eigenfunctions of $-\Delta_\Omega^\alpha$ still form a Riesz basis of $L^2(\Omega)$ if $\Omega \subset \R^d$ is a bounded Lipschitz domain in dimension $d\geq 2$, as they do when $d=1$, and when $\Omega$ is rectangular (see Remark~\ref{rem:hyperrectangle-riesz}); and we recall Open Problem~\ref{prob:eigennilpotents}, to establish that the eigennilpotents are always trivial, that is, that all generalised eigenfunctions are in fact eigenfunctions.

\begin{oproblem}
Let $\Omega \subset \R^d$, $d\geq 2$, be bounded and Lipschitz and let $\alpha \in \C$. Do the eigenfunctions of $-\Delta_\Omega^\alpha$ form a Riesz basis of $L^2(\Omega)$?
\end{oproblem}

\begin{proof}[Proof of Theorem~\ref{thm:abel-basis}]
We only need to apply Theorem~\ref{thm:agranovich}, as was done in \cite[Section~5]{HKS} for $d=1$, and in fact we refer there for the proof in this case.

So suppose that $d\geq 2$. Obviously, we choose $H=L^2(\Omega)$ and $V=H^1(\Omega)$. Given $\alpha \in \C$, which will be fixed throughout, we suppose $\omega \geq 0$ to be such that $a_\alpha [u,u] + \omega (u,u)$ is coercive $H^1 (\Omega)$, which we may always do by the trace inequality, cf.~Lemma~\ref{lem:robin-operator} (here and throughout $(\cdot,\cdot)$ is the inner product on $L^2(\Omega)$). We then choose
\begin{displaymath}
	a[u,v] :=a_\alpha [u,v] + \omega (u,v),
\end{displaymath}
so that $b=a_{\Re \alpha} + \omega (\cdot,\cdot)$, $A = -\Delta_\Omega^\alpha + \omega I$, and $B = -\Delta_\Omega^{\Re \alpha} + \omega I$, and to apply Theorem~\ref{thm:agranovich} we only need to check the conditions (i) and (ii). We will show that (i) holds for $q=1/2$ and (ii) holds for any $p\leq d/2$, leading in particular to the order of the Abel basis claimed in the theorem.

For (i), first note that for any operator $B$ satisfying the assumptions of the theorem, we have that
\begin{displaymath}
	(B^{1/2}u,B^{1/2}u) = (Bu,u) =  b[u,u]
\end{displaymath}
for all $u \in D(B)$, and in particular $\|B^{1/2}u\|_{2} = b[u,u]^{1/2}$ for all $u \in V$. 

Up to a possibly different constant, the form $a_{\Re \alpha} + \omega (\cdot,\cdot)$ defines an equivalent norm on $H^1 (\Omega)$. Thus, in our setting, and with $q=1/2$, \eqref{eq:agranovich-trace} reduces to the question of the existence of a constant $m>0$ such that, for all $u \in H^1(\Omega)$
\begin{displaymath}
	|\Im \alpha| \int_\Omega |u|^2\,\dsigma \leq m\|u\|_{H^1(\Omega)}\|u\|_{2}.
\end{displaymath}
But this, in turn, follows immediately from the trace inequality \eqref{eq:alternative-trace-form} of Remark~\ref{rem:alternative-trace-form}, to be proved below.

For (ii), note that the constant $\omega$ has no effect on the asymptotic behaviour of the eigenvalues; thus we may assume without loss of generality that $\omega=0$. We are thus interested in the smallest $p>0$ such that
\begin{displaymath}
	\limsup_{k\to\infty} \frac{\lambda_k (-\Delta_\Omega^{\Re \alpha})}{k^p} > 0.
\end{displaymath}
But by the Weyl asymptotics for the Robin Laplacian, valid for any $\Re \alpha \in \R$, in fact
\begin{displaymath}
	\lambda_k (-\Delta_\Omega^{\Re\alpha}) = C_d(|\Omega|) k^{d/2} + o (k^{d/2})
\end{displaymath}
as $k\rightarrow\infty$ for a constant $C_d(|\Omega|)>0$ (see, for example, \cite{Ivrii,VaSa}), leading to $p \leq d/2$.
\end{proof}


\section{On the numerical range}
\label{sec:numerical-range}

In this section we will give bounds on the numerical range of the form $a_\alpha$ associated with the operator $-\Delta_\Omega^\alpha$ on a general Lipschitz domain $\Omega \subset \R^d$, which we recall is given by
\begin{displaymath}
	W (a_\alpha) = \left\{ \int_\Omega |\nabla u|^2\,\dx + \int_{\partial\Omega}\alpha |u|^2\,\dsigma :
	u \in H^1(\Omega)\text{ with } \|u\|_2=1 \right\} \subset \C.
\end{displaymath}
Since here $\sigma (-\Delta_\Omega^\alpha) = \sigma_p (-\Delta_\Omega^\alpha) \subset W (-\Delta_\Omega^\alpha) \subset W (a_\alpha)$, in addition to giving an independent proof of the sectoriality of the form and the operator claimed in Section~\ref{sec:robin-operator}, these bounds will more importantly give an estimate on the rate at which any eigenvalues can diverge in the regime $\Re \alpha < 0$, as well as the size of the imaginary part of the eigenvalues: in particular, the following theorem contains Theorem~\ref{thm:general-eigenvalue-estimate}. We obtain somewhat different results, and require a different method of proof, in the case where $\partial\Omega$ is sufficiently smooth ($C^2$) on the one hand, and Lipschitz on the other.

\begin{theorem}
\label{thm:numerical-range}
Suppose $\Omega \subset \R^d$, $d\geq 2$, is a bounded Lipschitz domain. Then there exist constants $C_1\geq 2$ and $C_2>0$ depending only on $\Omega$ such that for $\alpha \in \C$ the set $W (a_\alpha)$ is contained in
\begin{displaymath}
	\Lambda_{\Omega,\alpha} = \left\{ t+\alpha\cdot s \in \C: t\geq 0,\, s\in [0,C_1\sqrt{t}+C_2] \right\}.
\end{displaymath}
In particular, we have the estimate
\begin{equation}
\label{eq:numerical-range-real-part}
	\Re \lambda \geq -\frac{C_1^2}{4}|\Re \alpha|^2 - C_2|\Re \alpha|
\end{equation}
for all $\lambda \in \sigma (-\Delta_\Omega^\alpha)$. If $\Omega$ has $C^2$ boundary, then we may choose $C_1 = 2$.
\end{theorem}

The regions $\Lambda_{\Omega,\alpha}$ for different values of $\alpha$ are depicted in Figures~\ref{fig:SimpleCurve} and~\ref{fig:MultipleCurves}. The constants $C_1, C_2$ depend on the geometry of $\partial\Omega$, and with our method of proof it should be possible to give an estimate on them, at least in principle. See~Remark~\ref{rem:comega} for a discussion of the meaning of $C_2$ in the case of smooth domains, where we obtain an expression for $C_2$ related to the curvature of $\partial\Omega$. It does not seem clear that we should expect $C_1 = 2$ in Theorem~\ref{thm:numerical-range} for domains of class $C^1$, not just $C^2$; cf.~Remark~\ref{rem:neg-alpha-bound}(1).

\begin{figure}[h]
  \includegraphics[scale=0.2]{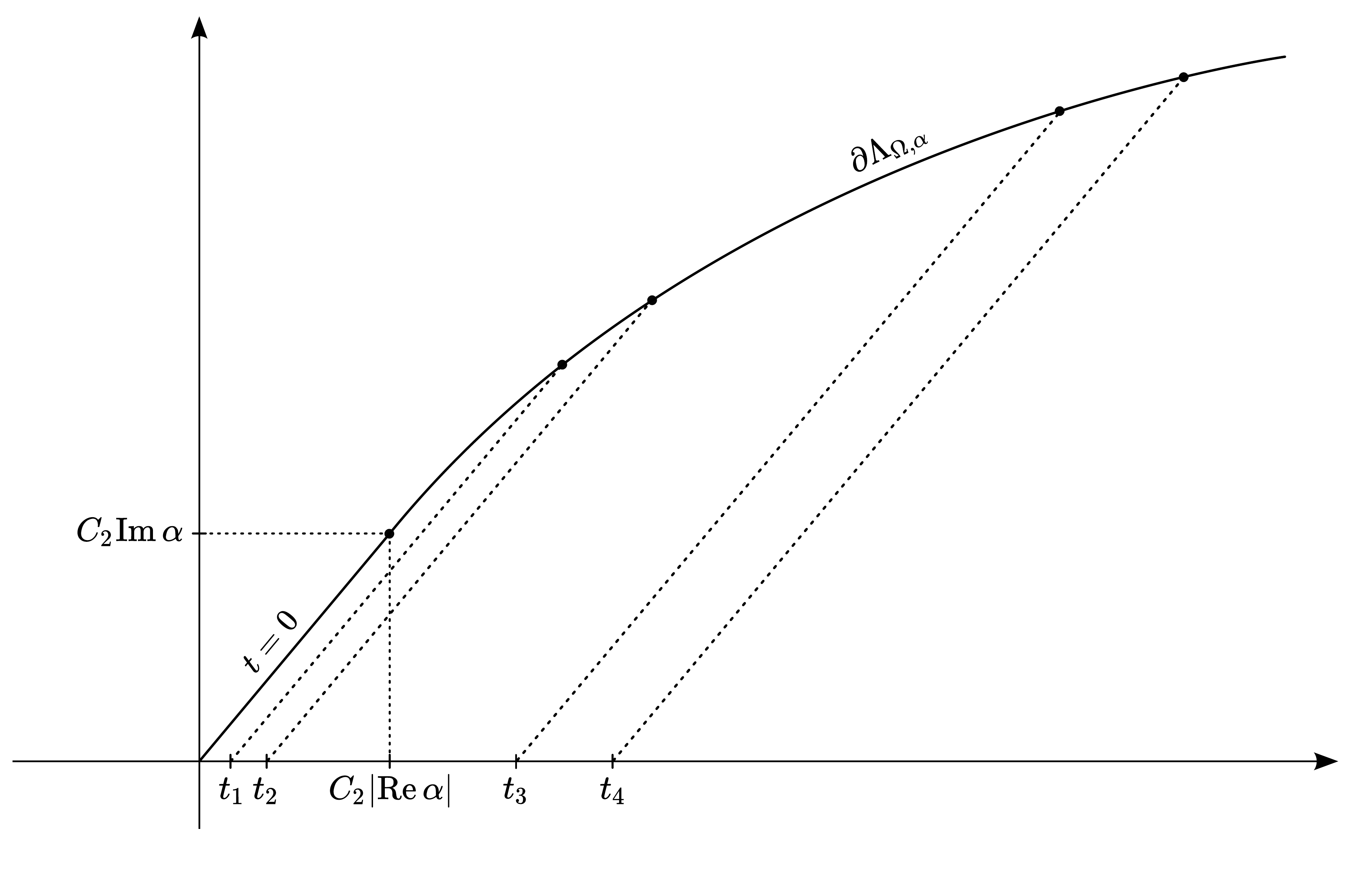}
  \caption{The set $\Lambda_{\Omega,\alpha}$, which contains the numerical range $W(a_\alpha)$, for a representative choice of $\Re \alpha >0$ and $\Im \alpha >0$, corresponding to the region between the curve $\partial\Lambda_{\Omega,\alpha}$ and the real axis. The region is composed of the union of segments of the form $\{ t+\alpha\cdot s \in \C: s\in [0,C_1\sqrt{t}+C_2] \}$, each of slope $\Im\alpha/\Re\alpha$, for different values of $t \geq 0$; the dotted lines show these segments for selected values of $t_1,\ldots, t_4 > 0$. Their endpoints form a parabolic section of $\partial\Lambda_{\Omega,\alpha}$ open to the right.}
  \label{fig:SimpleCurve}
\end{figure}

\begin{figure}[h]
  \includegraphics[scale=0.2]{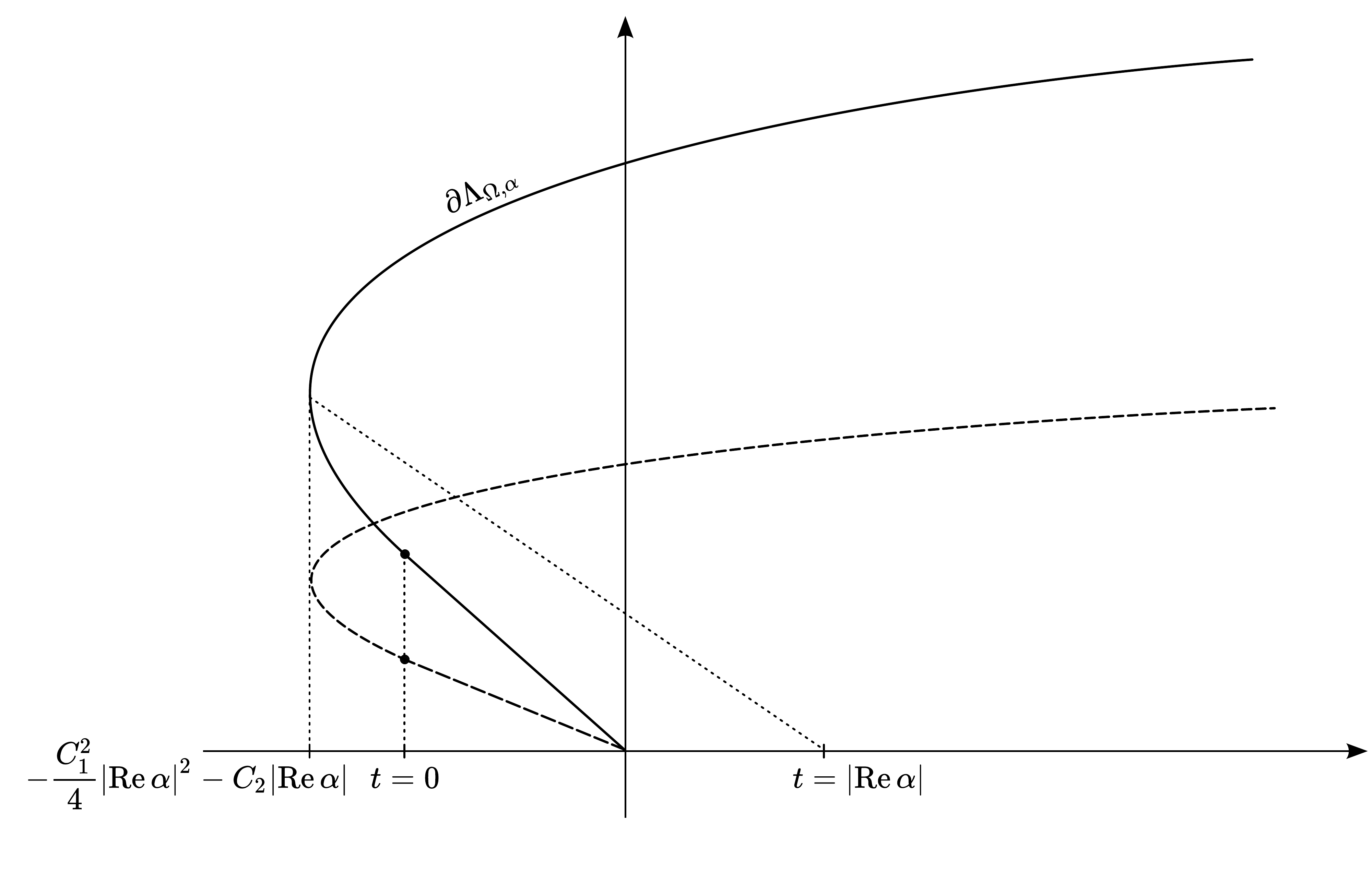}
  \caption{The set $\Lambda_{\Omega,\alpha}$ for $\Re \alpha < 0$ and two different choices of $\Im \alpha > 0$ (whose upper boundaries correspond to the solid and dashed curves, respectively). As $\Im \alpha \to 0$, the region collapses to the part of the real axis from $-\frac{C_1^2}{4}|\Re\alpha|^2 - C_2 |\Re\alpha|$ to $+\infty$.}
  \label{fig:MultipleCurves}
\end{figure}

Since these bounds are new even in the case of real negative $\alpha$ (see also Corollary~\ref{cor:neg-alpha-bound}), we wish to discuss how they fit in with known results before we turn to the proofs.

\begin{remark}
\label{rem:neg-alpha-bound}
(1) We recall the bound
\begin{equation}
\label{eq:neg-alpha-upper-bound}
	\lambda_1 (\alpha) < -|\alpha|^2
\end{equation}
on the principal Robin eigenvalue $\lambda_1 (\alpha) = \lambda_1 (-\Delta_\Omega^\alpha)$ of any bounded Lipschitz domain $\Omega \subset \R^d$ for $\alpha < 0$, which may be obtained by a simple variational argument (see \cite[Theorem~2.3]{GiorgiSmits} or \cite[Proposition~4.12]{BFK}). Together with this bound, Corollary~\ref{cor:neg-alpha-bound} gives a new, simpler, proof of the asymptotic behaviour $\lambda_1 (\alpha) = -|\alpha|^2 + \mathcal{O}(\alpha)$ as $\alpha \to -\infty$, if $\Omega$ is $C^2$. The only other proof that $\lambda_1 (\alpha) = -|\alpha|^2 + o(\alpha^2)$ on $C^2$ -- actually $C^1$ -- domains, which is completely different and involves a blow-up argument, is the principal result of \cite{LouZhu}; all other proofs (which give more terms in the expansion) require more boundary regularity. Indeed, if $\Omega$ is $C^3$, then, as $\alpha \to -\infty$,
\begin{equation}
\label{eq:two-term-asymptotics}
	\lambda_1 (\alpha) = -|\alpha|^2 - (d-1)\bar{\kappa}_{\max}|\alpha| + \mathcal{O} (\alpha^{2/3}),
\end{equation}
where $\bar{\kappa}_{\max}$ denotes the maximal mean curvature of $\partial\Omega$; see \cite[Section~4.4.2.1]{BFK} for a discussion and references. It is interesting to note that in the case of smooth $\Omega$ the constant $C_2$ appearing in Theorem~\ref{thm:numerical-range} is likewise related to the curvature of $\partial\Omega$ (see Remark~\ref{rem:comega} for more details); the presence of the curvature suggests that the ``smooth'' version of Theorem~\ref{thm:numerical-range} (i.e., with $C_1=2$) does not hold under significantly weaker regularity assumptions than $C^2$. We leave it as an open problem to determine whether a better bound than ours is possible for $C^1$ domains (see Open Problem~\ref{prob:C1} below).

(2) We recall that for domains $\Omega$ with piecewise smooth boundary and a finite number of ``model corners'', the asymptotic behaviour of the principal eigenvalue becomes
\begin{equation}
\label{eq:neg-alpha-lipschitz-principal-ev}
	\lambda_1 (\alpha) = -C|\alpha|^2 + o(\alpha^2)
\end{equation}
as $\alpha \to -\infty$, for a constant $C\geq 1$ depending on the opening angle(s) of the ``most acute'' corner(s) of $\Omega$ (we refer to \cite[pp.~94--95]{BFK} for details and references); it is an open problem to show that \eqref{eq:neg-alpha-lipschitz-principal-ev} also holds on general Lipschitz domains \cite[Open~Problem~4.17]{BFK}. The lower bound of Theorem~\ref{thm:numerical-range} in the form of Corollary~\ref{cor:neg-alpha-bound}, together with \eqref{eq:neg-alpha-upper-bound}, at least implies a two-sided asymptotic bound of this form.
\end{remark}

\begin{oproblem}
\label{prob:C1}
Let $\Omega \subset \R^d$ be bounded and of class $C^1$ and suppose that $\alpha \in \R$. Is it true that
\begin{displaymath}
	\lambda_k (\alpha) = -|\alpha|^2 + \mathcal{O} (\alpha)
\end{displaymath}
as $\alpha \to -\infty$, for each $k \in \N$?
\end{oproblem}

\begin{remark}
\label{rem:cosine-function}
We note in passing that the bound of Theorem~\ref{thm:numerical-range} on the numerical range (and the spectrum) of $-\Delta_\Omega^\alpha$ implies that, for any Lipschitz domain $\Omega \subset \R^d$ and any $\alpha \in \C$, the operator $\Delta_\Omega^\alpha$ generates a \emph{cosine function}, that is, the corresponding wave equation is well posed (see \cite[Section~3.14]{ABHN} for more details on cosine functions of operators); in fact, it is known that an operator is the generator of a cosine function if and only if its numerical range and spectrum are contained in a parabolic region such as $-\Lambda_{\Omega,\alpha}$; see \cite[Theorem~3.17.4]{ABHN}. Here, we see that $-\Lambda_{\Omega,\alpha}$ is contained in the parabolic region described in that theorem for sufficiently large $\omega>0$, how large depending on $\alpha$, $C_1$ and $C_2$.
\end{remark}

The proof of Theorem~\ref{thm:numerical-range} is based on the following trace-type inequality.

\begin{lemma}
\label{lem:boundary-control}
Let $\Omega \subset \R^d$, $d\geq 2$, be a bounded Lipschitz domain. Then there exist constants $C_1\geq 2$ and $C_2>0$, both depending only on $\Omega$, such that
\begin{equation}
\label{eq:boundary-control}
	\int_{\partial\Omega} |u|^2\,\dsigma \leq C_1\|\nabla u\|_2 + C_2
\end{equation}
for all $u \in H^1(\Omega)$ with $\|u\|_2=1$. If $\Omega$ has $C^2$ boundary, then we may choose $C_1=2$.
\end{lemma}

The constants $C_1$ and $C_2$ will be the same as the ones appearing in the statements of Theorems~\ref{thm:numerical-range} and~\ref{thm:general-eigenvalue-estimate}. While we doubt that the lemma is new in the case of Lipschitz domains, we are unaware of any reference; moreover, to the best of our knowledge its application in the context of Robin eigenvalue asymptotics is new. But we consider the real novelty, and difficulty, to consist in obtaining the best possible constant $C_1=2$ in the case of $C^2$ domains, which leads to sharp eigenvalue bounds in this case. The proofs for the cases of $C^2$ and Lipschitz boundaries are, correspondingly, completely different. For the smooth case, which we treat first, we first need a technical lemma involving the geometry of $\Omega$ near its boundary, where we will heavily rely on the assumption that $\partial\Omega$ is $C^2$. We first introduce some notation: for a bounded domain $\Omega \subset \R^d$, we set $d_\Omega : \R^d \to \R$,
\begin{equation}
\label{eq:dist-function}
	d_\Omega (x) := \begin{cases} \dist (x,\partial\Omega) = \inf_{z \in \partial\Omega} |x-z| \qquad &\text{if } x \in \overline\Omega\\
	-\dist (x,\partial\Omega) \qquad &\text{if } x \in \R \setminus \overline\Omega\end{cases}
\end{equation}
to be the signed distance function to $\partial\Omega$, $d_\Omega \in C(\R^d)$. Given any $\varepsilon > 0$ and $t \in [0,\varepsilon]$, we also set
\begin{equation}
\label{eq:boundary-strip}
	\Omega_\varepsilon := \{ x \in \R^d : d_\Omega (x) < \varepsilon \}
\end{equation}
to be the (open) ``strip'' around $\partial\Omega$ of width $2\varepsilon$, where we also write
\begin{equation}
\label{eq:boundary-division}
\begin{aligned}
	\Omega_\varepsilon^+ &:= \Omega_\varepsilon \cap \Omega = \{ x \in \Omega : d_\Omega (x) < \varepsilon \}\\
	\Omega_\varepsilon^- &:= \Omega_\varepsilon \cap \R^d \setminus \Omega =  \{ x \in \R^d\setminus \Omega : d_\Omega (x) < \varepsilon \}
\end{aligned}
\end{equation}
and finally
\begin{equation}
\label{eq:distance-levels}
	S_t := \{x \in \Omega : d_\Omega (x) = t \}
\end{equation}
to be the level surfaces of $d_\Omega$ in $\Omega_\varepsilon$,
\begin{displaymath}
	\Omega_\varepsilon = \bigcup_{t\in (-\varepsilon,\varepsilon)} S_t.
\end{displaymath}

\begin{lemma}
\label{lem:boundary-geometry}
Suppose $\Omega \subset \R^d$ is a bounded domain of class $C^2$. Then there exists $\varepsilon > 0$ such that
\begin{enumerate}
\item $d_\Omega |_{\overline{\Omega}_\varepsilon} \in C^2 (\overline{\Omega}_\varepsilon)$;
\item for each $x \in \overline{\Omega}_\varepsilon$ there exists a unique minimiser $z \in \partial\Omega$ such that $d_\Omega (x) = |x-z|$;
\item for each $x \in \overline{\Omega}_\varepsilon \setminus \partial\Omega$,
\begin{displaymath}
	\nabla d_\Omega (x) = \frac{x-z}{|x-z|}
\end{displaymath}
with $z$ as in {\rm(2)}. In particular, $|\nabla d_\Omega (x)| = 1$ for all $x \in \overline{\Omega}_\varepsilon$;
\item for each $t \in [-\varepsilon,\varepsilon]$, $S_t$ is a compact manifold of class $C^1$; and
\item for each $f \in C^1 (\overline{\Omega}_\varepsilon)$ the function
\begin{displaymath}
	t \mapsto \int_{S_t} f \,\dsigma
\end{displaymath}
is differentiable at every $t \in (-\varepsilon,\varepsilon)$, and its derivative, given by
\begin{equation}
\label{eq:distance-level-derivative}
	\int_{S_t} \partial_t f + f\Delta d_\Omega\,\dsigma,
\end{equation}
is in $C ([-\varepsilon,\varepsilon])$. In particular, for any $f \in C^1(\overline{\Omega_\varepsilon^+})$ and any $\varepsilon_1 \in [0,\varepsilon)$,
\begin{equation}
\label{eq:level-set-ftc}
	\int_{S_{\varepsilon_1}} f\,\dsigma - \int_{\partial\Omega} f\,\dsigma 
	= \int_{\Omega_{\varepsilon_1}^+} \partial_t f + f\Delta d_\Omega\,\dx.
\end{equation}
\end{enumerate}
\end{lemma}

\begin{proof}
(1) is contained in \cite[Appendix, Lemma~1]{GilbargTrudinger} (since $\partial\Omega$ is assumed to be $C^2$), see also \cite[Lemma~2.4.2]{BEL}; (2) follows (possibly for a different $\varepsilon$) from \cite[Lemma~4.11]{Federer} (which requires that $\partial\Omega$ be $C^{1,1}$) together with a simple covering argument using the fact that $\partial\Omega$ is compact (in the language of \cite{Federer}, (2) means that $\textrm{reach}(\partial\Omega) > 0$). (3) then follows from \cite[Theorem~4.8]{Federer}, where we note that $\nabla d_\Omega \in C^1(\overline{\Omega}_\varepsilon)$ and $|\nabla d_\Omega|=1$ in $\overline{\Omega}_\varepsilon \setminus \partial\Omega$ implies that $|\nabla d_\Omega|=1$ everywhere in $\overline{\Omega}_\varepsilon$; and (4) follows from (1) using the Implicit Function Theorem and the fact that $\nabla d_\Omega$ never vanishes on $\overline{\Omega}_\varepsilon$ by (3), together with a covering argument since $S_t$ is clearly compact.

For (5), fix $f \in C^1 (\overline{\Omega}_\varepsilon)$ and for brevity write
\begin{displaymath}
	F(t) := \int_{S_t} f\,\dsigma.
\end{displaymath}
We first claim that \eqref{eq:distance-level-derivative} is the distributional derivative of $F$. Indeed, for any test function $\varphi \in C_c^\infty (-\varepsilon,\varepsilon)$, we have
\begin{displaymath}
	\int_{-\varepsilon}^\varepsilon F(t)\varphi(t)\,\text{d}t = \int_{-\varepsilon}^\varepsilon \int_{S_t} f\varphi(t)\,\dsigma\,\text{d}t
	= \int_{\Omega_\varepsilon} f\varphi\circ d_\Omega\,\text{d}x
\end{displaymath}
by the coarea formula in the form of \cite[Section~3.4.3]{EvansGariepy}, using the fact that the $S_t$ are the level surfaces of $d_\Omega$ and $|\nabla d_\Omega| = 1$ everywhere by (3). In particular,
\begin{displaymath}
\begin{aligned}
	\int_{-\varepsilon}^\varepsilon F(t)\varphi'(t)\,\text{d}t 
	&= \int_{\Omega_\varepsilon} f\varphi'\circ d_\Omega\,\dx
	= \int_{\Omega_\varepsilon} f \nabla d_\Omega \cdot \nabla (\varphi \circ d_\Omega)\,\dx\\
	&= -\int_{\Omega_\varepsilon} \varphi \circ d_\Omega \divergence (f\nabla d_\Omega)\,\dx
	= -\int_{-\varepsilon}^\varepsilon \varphi(t) \int_{S_t} \divergence (f\nabla d_\Omega)\,\dx\,\text{d}t,
\end{aligned}
\end{displaymath}
where for the second last equality we have used the divergence theorem (integration by parts) and the compact support of $\varphi$, and the last equality follows from another application of the coarea formula. The claim now follows from the short calculation
\begin{displaymath}
	\divergence (f\nabla d_\Omega) = \nabla f \cdot \nabla d_\Omega + f\Delta d_\Omega = \partial_t f + f\Delta d_\Omega,
\end{displaymath}
valid pointwise in $\overline{\Omega}_\varepsilon$ since $d_\Omega$ is $C^2$ by (1), and using the fact that $\nabla d_\Omega$ points in the direction of $t$ by (3). We next note that the integrand in \eqref{eq:distance-level-derivative} is in $C (\overline{\Omega}_\varepsilon)$ and hence a short argument using the compactness of $S_t$ and the uniform continuity of the integrand shows that the integral in \eqref{eq:distance-level-derivative} is in fact in $C([-\varepsilon,\varepsilon])$; in particular, it is the pointwise derivative of $F$ at every point in $(-\varepsilon,\varepsilon)$.

Finally, for \eqref{eq:level-set-ftc}, by what we have just shown we may apply the Fundamental Theorem of Calculus in the form of \cite[Theorem~7.21]{Rudin} to the function $F$ on the interval $[0,\varepsilon_1]$ (for any $\varepsilon_1<\varepsilon$) to obtain
\begin{displaymath}
	F(\varepsilon_1) - F(0) = \int_0^{\varepsilon_1} \int_{S_t} \partial_t f + f\Delta d_\Omega\,\dsigma.
\end{displaymath}
A final application of the coarea formula to the integral on the right-hand side, together with the definition of $F$, yields \eqref{eq:level-set-ftc}.
\end{proof}

\begin{proof}[Proof of Lemma~\ref{lem:boundary-control}]
\emph{The case of $C^2$ boundary.} We keep the notation from \eqref{eq:dist-function}, \eqref{eq:boundary-strip}, \eqref{eq:boundary-division} and \eqref{eq:distance-levels} and note that it suffices to prove \eqref{eq:boundary-control} for all $u \in C^1 (\overline\Omega)$, by density of the latter set in $H^1 (\Omega)$ for bounded $\Omega$ of class $C^2$ (cf.~\cite[Section~7.6]{GilbargTrudinger}) and the trace theorem. We let $\varepsilon>0$ be as in Lemma~\ref{lem:boundary-geometry} (in particular, by making $\varepsilon$ a little smaller if necessary we assume that \eqref{eq:level-set-ftc} holds with $\varepsilon$ in place of $\varepsilon_1$) and choose a cut-off function $\varphi \in C^1 (\overline{\Omega})$ such that $0 \leq \varphi \leq 1$ in $\overline{\Omega}$, $\varphi = 0$ outside $\Omega_\varepsilon$, $\varphi|_{S_t}$ is constant for all $t \in [0,\varepsilon]$, and $\varphi|_{\partial\Omega} = 1$. (The existence of such a function is guaranteed by the regularity statements in Lemma~\ref{lem:boundary-geometry}: indeed, if we let $\psi \in C^\infty ([0,\infty))$ be any smooth function satisfying $\psi(0)=1$ and $\psi(t)=0$ for all $t\geq \varepsilon$, then we may take $\varphi = \psi \circ d_\Omega$.)

Now fix $u \in C^1 (\overline{\Omega})$ such that $\|u\|_2=1$. Then $f:=|u|^2\varphi \in C^1 (\overline\Omega)$ and we apply the formula \eqref{eq:level-set-ftc} to $f$, and use the fact that $\varphi=1$ on $\partial\Omega$ and $\varphi=0$ on $S_\varepsilon$, to obtain
\begin{displaymath}
\begin{aligned}
	-\int_{\partial\Omega} |u|^2\,\dsigma &= \int_{\Omega_\varepsilon^+} \partial_t (|u|^2\varphi)+|u|^2\varphi \Delta d_\Omega\,\dx\\
	&= \int_{\Omega_\varepsilon} 2\varphi \Re(\overline{u}\partial_t u) + |u|^2\partial_t\varphi + |u|^2\varphi\Delta d_\Omega\,\dx.
\end{aligned}
\end{displaymath}
Using the fact that $\varphi = 0$ on $\Omega \setminus \Omega_\varepsilon^+$, we may therefore estimate
\begin{displaymath}
\begin{aligned}
	\int_{\partial\Omega}|u|^2\,\dsigma 
	&\leq 
2\|\varphi\|_\infty\|u\|_2\|\nabla u\|_2
+\|\nabla\varphi\|_\infty\|u\|_2^2+\max_{x\in\overline{\Omega}_\varepsilon}|\Delta d_\Omega|\|\varphi\|_\infty\|u\|_2^2\\
	&=2\|\nabla u\|_2+\|\nabla\varphi\|_\infty+\max_{x\in\overline{\Omega_\varepsilon^+}}|\Delta d_\Omega(x)|
\end{aligned}
\end{displaymath}
using the normalisation $\|u\|_2=1$ as well as $\|\varphi\|_\infty=1$ (where all norms are over $\Omega$). This proves \eqref{eq:boundary-control} with
\begin{equation}
\label{eq:comega}
	C_2:=\|\nabla\varphi\|_\infty+\max_{x\in\overline{\Omega_\varepsilon^+}}|\Delta d_\Omega(x)|.
\end{equation}

\emph{The case of Lipschitz boundary.} Since in the case of general Lipschitz domains the corresponding parametrisation of $\Omega_\varepsilon$ does not enjoy the same regularity properties, we give a different, local argument. Fix $z \in \partial\Omega$ and a neighbourhood $\mathcal{U}_z$ of $z$ such that within $\mathcal{U}_z$, $\partial\Omega$ is given by the graph of a Lipschitz function $g:\R^{d-1} \to \R$ such that $\Omega \cap \mathcal{U}_z$ lies in the region $\{ (x_1,\ldots,x_d): x_d < g(x_1,\ldots,x_{d-1}) \}$ (where we use the notation $(x_1,\ldots,x_d) \in \R^d \simeq \R^{d-1} \times \R$). Then in this coordinate system, the normal vector to $\partial\Omega$ given by $\nu = (\nu_1,\ldots,\nu_d) : \partial\Omega \to \R^d$, which is an $L^\infty$-function since $\partial\Omega$ is Lipschitz, satisfies
\begin{equation}
\label{eq:lipschitz-normal-condition}
	\essinf \{ \nu_d (y) : y \in \partial\Omega \cap \mathcal{U}_z \} > 0.
\end{equation}
Now fix a test function $\varphi \in C_c^\infty (\R^d)$ such that $0\leq \varphi \leq 1$, $\varphi|_{\partial\Omega \cap \mathcal{U}_z} = 1$ and $\varphi(y) = 0$ for all $y \in \partial\Omega$ with $\nu_d(y) \leq 0$. (By shrinking the neighbourhood $\mathcal{U}_z$ if necessary, we can always guarantee the existence of such a $\varphi$.)

Then for a given function $u \in H^1 (\Omega)$ with $\|u\|_2=1$, we have
\begin{displaymath}
	\essinf_{y \in \partial\Omega \cap \mathcal{U}_z} \nu_d (y) \int_{\partial\Omega \cap \mathcal{U}_z} |u|^2\,\dsigma
	\leq \int_{\partial\Omega} \varphi |u|^2 \nu_d \,\dsigma
	= \int_{\Omega} \frac{\partial}{\partial x_d} (\varphi |u|^2)\,\dx
\end{displaymath}
by the divergence theorem applied to the function $F = (0,\ldots,0, \varphi|u|^2) \in W^{1,1}(\Omega)$ and the Lipschitz domain $\Omega$ (see \cite[Th\'eor\`eme~3.1.1]{Necas}). The latter integral may be estimated by
\begin{equation}
\label{eq:intermediate-lipschitz-trace}
	\int_{\Omega} \frac{\partial}{\partial x_d} (\varphi |u|^2)\,\dx \leq \|\nabla\varphi\|_\infty\|u\|_2^2+2\|\varphi\|_\infty\|\nabla u\|_2\|u\|_2;
\end{equation}
using the normalisations $\|u\|_2=1$, $\|\varphi\|_\infty=1$, this estimate may be expressed as
\begin{displaymath}
	\int_{\partial\Omega \cap \mathcal{U}_z} |u|^2\,\dsigma \leq C_{1,z}\|\nabla u\|_2 + C_{2,z}
\end{displaymath}
for suitable constants $C_{1,z}, C_{2,z}>0$ depending on $z$. Since $\partial\Omega$ is compact, a simple covering argument now yields \eqref{eq:boundary-control}. Note that for every $z \in \partial\Omega$ we have $C_{1,z} = 2/\essinf_{y \in \partial\Omega \cap \mathcal{U}_z} \nu_d (y) \geq 2$ since $|\nu|=1$; hence also $C_1 \geq 2$.
\end{proof}

\begin{remark}
\label{rem:alternative-trace-form}
For Lipschitz $\Omega$, the above proof also yields the slightly different trace inequality
\begin{equation}
\label{eq:alternative-trace-form}
	\int_{\partial \Omega} |u|^2 \,\dsigma \leq C(\Omega) \|u\|_{H^1}\|u\|_{2}
\end{equation}
for all $u \in H^1 (\Omega)$, needed in the proof of Theorem~\ref{thm:abel-basis}. Indeed, by \eqref{eq:intermediate-lipschitz-trace}, we have
\begin{displaymath}
\begin{aligned}
	\essinf_{y \in \partial\Omega \cap \mathcal{U}_z} \nu_d (y) \int_{\partial\Omega \cap \mathcal{U}_z} |u|^2\,\dsigma
	\leq \int_{\Omega} \frac{\partial}{\partial x_d} (\varphi |u|^2)\,\dx
	&\leq \|\nabla\varphi\|_\infty\|u\|_2^2+2\|\varphi\|_\infty\|\nabla u\|_2\|u\|_2\\
	&\leq \left(\|\nabla\varphi\|_\infty + 2\|\varphi\|_\infty\right)\|u\|_{H^1}\|u\|_2,
\end{aligned}
\end{displaymath}
leading to
\begin{displaymath}
	\int_{\partial\Omega \cap \mathcal{U}_z} |u|^2\,\dsigma \leq C_z \|u\|_{H^1}\|u\|_2
\end{displaymath}
for all $u \in H^1(\Omega)$, for a constant $C_z>0$ depending only on $z \in \partial\Omega$. A covering argument as in the above lemma then yields \eqref{eq:alternative-trace-form}.
\end{remark}

\begin{proof}[Proof of Theorem~\ref{thm:numerical-range}]
Let $C_1\geq 2$, $C_2>0$ be the constants from Lemma~\ref{lem:boundary-control} (in particular, we assume $C_1=2$ if $\Omega$ is $C^2$). 
Fix $u \in H^1(\Omega)$ with $\|u\|_2=1$ and set
\begin{displaymath}
	\lambda := \|\nabla u\|_2^2 + \int_{\partial\Omega} \alpha |u|^2\,\dsigma \in W (a_\alpha).
\end{displaymath}
For $t := \|\nabla u\|_2^2 \geq 0$ and $s := \int_{\partial\Omega} |u|^2\,\dsigma \geq 0$, we have
\begin{displaymath}
	\Re \lambda = t + \Re\alpha\cdot s, \qquad \Im \lambda = \Im\alpha \cdot s;
\end{displaymath}
moreover, by Lemma~\ref{lem:boundary-control}, we obtain that $s \leq C_1\sqrt{t} + C_2$; thus $\lambda \in \Lambda_{\Omega,\alpha}$.

To see that every $\lambda \in W (a_\alpha)$, and hence every $\lambda \in \sigma (-\Delta_\Omega^\alpha)$, satisfies the estimate \eqref{eq:numerical-range-real-part}, we first remark that if $\Re \alpha \geq 0$, then clearly $\Lambda_{\Omega,\alpha} \subset \{ z \in \C: \Re z \geq 0 \}$. Hence we may assume without loss of generality that $\Re \alpha < 0$. Then by Lemma~\ref{lem:boundary-control} and the inequality
\begin{displaymath}
	2\|\nabla u\|_2 \leq \frac{C_1}{2}|\Re \alpha| + \frac{2}{C_1|\Re \alpha|}\|\nabla u\|_2^2,
\end{displaymath}
we have
\begin{displaymath}
\begin{aligned}
	\Re \lambda = \|\nabla u\|_2^2 + \Re \alpha \int_{\partial\Omega}|u|^2\,\dsigma
	&\geq \|\nabla u\|_2^2-|\Re \alpha|\left[\frac{C_1}{2}\left(\frac{C_1}{2}|\Re\alpha|+\frac{2}{C_1|\Re \alpha|}\|\nabla u\|_2^2\right)+C_2\right]\\
	&= -\left(\frac{C_1}{2}\right)^2|\Re \alpha|^2 - C_2 |\Re \alpha|.
\end{aligned}
\end{displaymath}
\end{proof}

\begin{remark}
\label{rem:comega}
Suppose that $\partial\Omega$ is $C^2$. We recall that the constant $C_2=C_2(\Omega)$ appearing in Theorem~\ref{thm:numerical-range} and Lemma~\ref{lem:boundary-control}, as noted in \eqref{eq:comega}, may in this case be taken as
\begin{displaymath}
	C_2 = \|\nabla \varphi\|_\infty + \max_{x \in \overline{\Omega_\varepsilon^+}} |\Delta d_\Omega (x)|,
\end{displaymath}
where $\varepsilon>0$ is as in Lemma~\ref{lem:boundary-geometry} and $\varphi$ is chosen to have support in $\overline{\Omega_\varepsilon^+}$. Let us be a bit more specific. We may take $\|\nabla \varphi\|_\infty$ to be $1/\varepsilon$, corresponding to a linear function of $t \in [0,\varepsilon]$ extended by $0$ at $t=\varepsilon$ (which can be approximated arbitrarily well in the $\infty$-norm by $C^1$ functions), while for $x \in \overline{\Omega_\varepsilon^+}$, it is known that the Hessian of the signed distance function is equal to the Weingarten map of the (unique) surface $S_t$ passing through $x$, at $x$. In particular,
\begin{displaymath}
	|\Delta d_\Omega (x)| = \left|\sum_{j=1}^{d-1} \kappa_j^{S_t}(x)\right| = (d-1)\left|\bar{\kappa}^{S_t} (x)\right|
\end{displaymath}
where $\kappa_1^{S_t}(\,\cdot\,),\ldots, \kappa_d^{S_t}(\,\cdot\,)$ are the principal curvatures at a given point of $S_t$ and $\bar{\kappa}^{S_t}$ is its mean curvature \cite[Lemma 2.4.2 and Remark 2.4.4]{BEL}. This means that the essentially optimal form of the constant $C_2$ coming from our proof -- to be compared with the coefficient of $\alpha$ in \eqref{eq:two-term-asymptotics} -- is
\begin{equation}
\label{eq:optimal-comega}
	C_2 = \varepsilon^{-1} + (d-1)\max_{t \in [0,\varepsilon]} \max_{x \in S_t} |\bar{\kappa}^{S_t} (x)|,
\end{equation}
where $\varepsilon > 0$ is any constant for which Lemma~\ref{lem:boundary-geometry}(2) holds; in the language of \cite{Federer}, we may take any $\varepsilon \in (0, \textrm{reach}(\partial\Omega)]$. As a simple example, in the case of a ball $B$ of radius $R>0$, since $\kappa_j^{S_t} \equiv 1/(R-t)$ for all $j$ and we may take any $\varepsilon < R$, we thus end up with
\begin{equation}
\label{eq:comega-ball}
	\Re \lambda \geq -|\Re \alpha|^2 - \min_{r \in (0,R)}\left[\frac{1}{R-r} + \frac{d-1}{r}\right]|\Re \alpha| 
 =-|\Re \alpha|^2 - \frac{d+2\sqrt{d-1}}{R} |\Re \alpha|, 
\end{equation}
which may be compared with the known bound and asymptotics for real negative $\alpha$
\begin{displaymath}
	-|\alpha|^2 - \frac{d-1}{R}|\alpha| > \lambda_1 (-\Delta_B^\alpha) = -|\alpha|^2 - \frac{d-1}{R}|\alpha| + o (\alpha)
\end{displaymath}
where the inequality is valid for all $\alpha < 0$ and the asymptotic expansion is for $\alpha \to -\infty$, see \cite[Theorem~2 and~eq.~(1.2)]{AFK}.
\end{remark}

\begin{remark}
\label{rem:range-variable-alpha}
If we allow variable $\alpha \in L^\infty(\partial\Omega,\C)$, then it is clear that similar results hold since the key trace estimate, Lemma~\ref{lem:boundary-control}, does not depend on $\alpha$, although the region $\Lambda_{\Omega,\alpha}$ can no longer be described explicitly in general. However, \eqref{eq:numerical-range-real-part} has a direct equivalent: if we set
\begin{displaymath}
	\|\Re \alpha\|_\infty := \esssup_{x \in \partial\Omega} |\Re \alpha (x)|, \qquad
	\|\Im \alpha\|_\infty:=  \esssup_{x \in \partial\Omega} |\Im \alpha (x)|,
\end{displaymath}
then, mimicking the arguments of the proof of Theorem~\ref{thm:numerical-range} we obtain the estimate
\begin{equation}
	\Re \lambda \geq -\frac{C_1^2}{4}\|\Re \alpha\|_\infty^2 - C_2 \|\Re \alpha\|_\infty
\end{equation}
for all $\lambda \in \sigma (-\Delta_\Omega^\alpha)$, or more generally all $\lambda \in W(a_\alpha)$, where $C_1\geq 2$, $C_2>0$, and $C_1=2$ if $\partial\Omega$ is $C^2$; even in the case of real-valued $\alpha$, this may be viewed as a partial generalisation of \cite[Remark~1.1]{LouZhu}, which establishes the asymptotics for real-valued variable $\alpha$ of the form $\alpha = tb(x)$, $t \to -\infty$, for a fixed function $b \in C(\partial\Omega)$. Moreover, we can still obtain parabolic estimates on the numerical range of the type necessary to ensure that $\Delta_\Omega^\alpha$ generates a cosine function (cf.~Remark~\ref{rem:cosine-function}). For simplicity assume that $\Re \alpha (x) \geq 0$ almost everywhere (whence also $\Re \lambda \geq 0$ for any $\lambda \in W (a_\alpha)$); then, with $C_1,C_2$ as above,
\begin{equation}
\label{eq:eig-im-est}
	|\Im \lambda| = \left|\int_{\partial\Omega}\Im \alpha \, |u|^2\,\dsigma(x)\right| \leq \|\Im \alpha\|_\infty (C_1\|\nabla u\|_2 + C_2)
	\leq \|\Im \alpha\|_\infty (C_1\sqrt{\Re \lambda} + C_2),
\end{equation}
independently of $\Re \alpha \geq 0$.
\end{remark}
\section{The Dirichlet-to-Neumann operator}
\label{sec:dno}

From now on, we will be interested in the asymptotic behaviour of the eigenvalues of $-\Delta_\Omega^\alpha$ as $\alpha \to \infty$ in $\C$. To this end, we will exploit the duality between the Robin eigenvalue problem \eqref{eq:robin-laplacian} and the eigenvalue problem
\begin{equation}
\label{DN_EigenvalueProblem}
	M(\lambda)g=\alpha g
\end{equation}
of the \emph{Dirichlet-to-Neumann operator} $M(\lambda)$ acting on $\partial\Omega$, defined for $\lambda$ in the resolvent set of the Dirichlet Laplacian. For more information on this operator, we refer to, e.g., \cite{ArendtMazzeo,AtE11,BehrndtTerElst,BehrndtRohleder,Daners,GesztesyMitrea,Marletta}; the relationship between this operator and the Robin Laplacian (at least for real $\alpha$) is explored in \cite[Section~2]{ArendtMazzeo} and \cite[Section~8]{AtE11}, for example, and for complex $\alpha$ see for example \cite[Section~3]{GesztesyMitrea}.

In order to define $M(\lambda)$, we first need to recall a solubility result for the inhomogeneous Dirichlet problem. Here and in what follows we fix a bounded Lipschitz domain $\Omega \subset \R^d$, $d\geq 2$, write $\tr u = u|_{\partial\Omega}$ for the trace of a function $u \in H^1 (\Omega)$ (though, as we have done previously, if there is no ambiguity we will tend to omit the ``$\tr$'' notation), and recall that every $g \in H^{1/2}(\partial\Omega)$ is the trace of a function $u \in H^1 (\Omega)$.

\begin{lemma}
\label{DNop:UniqueSolutionOfDirichletProblem}
Let $\Omega \subset \R^d$, $d\geq 2$, be a bounded Lipschitz domain and let $\lambda\in\rho(-\Delta_\Omega^D) \subset \C$. For each $g\in H^{1/2}(\partial\Omega)$, the Dirichlet boundary value problem
\begin{equation}
\label{eq:inhomogeneous-dirichlet}
\begin{aligned}
	-\Delta u &=\lambda u \qquad &&\text{in } \Omega,\\
	u &=g \qquad &&\text{on } \partial\Omega,
\end{aligned}
\end{equation}
interpreted in the usual weak sense, has a unique solution $u_\lambda\in H^1(\Omega)$, that is, $u_\lambda$ solves
\begin{equation}
\label{eq:weak-dirichlet}
	\int_\Omega \nabla u \cdot \overline{\nabla v} \,\dx = \lambda \int_\Omega u\overline{v}\,\dx
\end{equation}
for all $v \in H^1_0 (\Omega)$, and $\tr u = g$. Moreover, for such $\lambda$, if
\begin{equation}
\label{eq:h1lambda}
	H^1 (\lambda) := \{ u \in H^1 (\Omega): -\Delta u = \lambda u \text{ in the sense of } \eqref{eq:weak-dirichlet} \},
\end{equation}
then we have the direct sum decomposition $H^1 (\Omega) = H^1_0 (\Omega) \oplus H^1 (\lambda)$.
\end{lemma}

\begin{proof}
For $\lambda \in \R \cap \rho (-\Delta_\Omega^D)$, this follows immediately from \cite[Lemma~2.2]{ArendtMazzeo}, together with the fact that $H^{1/2} (\partial\Omega) = \tr H^1 (\Omega)$; for general $\lambda \in \rho (-\Delta_\Omega^D)$ the same proof works verbatim.
\end{proof}

We denote by $P(\lambda): H^{1/2}(\partial\Omega) \to H^1(\Omega)$ the Poisson operator given by
\begin{equation}
\label{eq:poisson}
	g \mapsto u_\lambda,
\end{equation}
where $u_\lambda$ solves \eqref{eq:inhomogeneous-dirichlet}, which is well defined for any $\lambda \in \rho (-\Delta_\Omega^D)$; indeed, one may show that $P(\lambda)$ is a bijection from $H^{1/2}(\partial\Omega)$ onto $H^1(\lambda)$ as defined in \eqref{eq:h1lambda} and in fact a right inverse of the trace operator. We can now define the Dirichlet-to-Neumann operator. For $\lambda \in \rho (-\Delta_\Omega^D)$, we first define a sesquilinear form $q_\lambda: H^{1/2} (\partial \Omega) \times H^{1/2} (\partial \Omega) \to \C$ by
\begin{equation}
\label{eq:dno-form}
	q_\lambda [g,h] = \int_\Omega \nabla P(\lambda)g \cdot \overline{\nabla P(\lambda)h} - \lambda P(\lambda)g\,\overline{P(\lambda)h}\,\dx.
\end{equation}
The (negative) \emph{Dirichlet-to-Neumann operator} $M(\lambda): D(M(\lambda)) \subset L^2(\partial\Omega) \to L^2(\partial\Omega)$ is then the operator in $L^2(\partial\Omega)$ associated with $-q_\lambda$, which a short calculation shows to be given by
\begin{equation}
\label{eq:dno-definition}
\begin{aligned}
	D(M(\lambda)) &= \left\{ g \in H^{1/2} (\partial\Omega): \frac{\partial}{\partial\nu} P(\lambda)g \in L^2(\partial\Omega) \right\},\\
	M(\lambda) &= -\frac{\partial}{\partial\nu} P(\lambda),
\end{aligned}
\end{equation}
where the normal derivative $\frac{\partial}{\partial\nu} P(\lambda)g$ was defined in \eqref{eq:l2-normal-derivative}. In words, the Dirichlet-to-Neumann operator maps given \emph{Dirichlet data} $g=\tr u$ to the \emph{Neumann data} $-\frac{\partial u}{\partial\nu}$ of the same solution $u=P(\lambda)g$ of $-\Delta u = \lambda u$.

\begin{lemma}
\label{lem:dno-bounded-h1-l2}
Let $\lambda \in \rho(-\Delta_\Omega^D)$. The operator $-M(\lambda)$ is closed, densely defined and m-sectorial, and has compact resolvent in $L^2(\partial\Omega)$. In particular, its spectrum consists of eigenvalues of finite algebraic multiplicity.
\end{lemma}

\begin{proof}
Everything except the sectoriality follows immediately since $H^{1/2}(\partial\Omega)$ is densely and compactly embedded in $L^2 (\partial\Omega)$, and $q_\lambda$ is closed on $H^{1/2}(\partial\Omega)$. For the sectoriality of the operator, it suffices to show that $q_\lambda$ is sectorial, that is, that there exist constants $\omega, \mu \in \R$ such that
\begin{equation}
\label{eq:dn-sectorial}
	\Re q_\lambda [g,g] + \omega \|g\|_{L^2(\partial\Omega)}^2 \geq \mu \|g\|_{H^{1/2}(\partial\Omega)}^2
\end{equation}
for all $g \in H^{1/2} (\partial\Omega)$; to prove \eqref{eq:dn-sectorial}, by the fact that the trace map is bounded from $H^1(\Omega)$ to $H^{1/2}(\partial\Omega)$ it certainly suffices to show that for any $\lambda \in \C$ there exists $\omega \geq 0$ such that
\begin{displaymath}
	\int_\Omega |\nabla u|^2 - \Re\lambda |u|^2\,\dx + \omega \int_{\partial\Omega} |u|^2\,\dsigma,
\end{displaymath}
$u \in H^1(\Omega)$, defines an equivalent norm on $H^1 (\Omega)$. But this, in turn, follows immediately from Maz'ya's inequality in the form of \cite[eq.~(4)]{AtE11}. We conclude that $q_\lambda$ and $M(\lambda)$ are sectorial. (See also \cite[Corollary~2.2 and Section~4.4]{AtE12} for a slightly different but equivalent approach in the case of real $\lambda$, which can however be carried over verbatim to complex $\lambda$.)
\end{proof}

\begin{remark}
\label{rem:dno-various}
(1) It may be shown that $D(M(\lambda)) = H^1 (\partial \Omega)$ for any $\lambda \in \rho (-\Delta_\Omega^D)$, use \cite[Theorem 4.25]{McLean} with $s=1/2$; however, we will not need this.

(2) It is also possible to define the Dirichlet-to-Neumann operator for $\lambda \in \sigma (-\Delta_\Omega^D)$, either as a multi-valued operator, or by factoring out the eigenfunctions of the eigenvalue $\lambda$ of the Dirichlet Laplacian from $H^1(\Omega)$. The conclusion of Lemma~\ref{lem:dno-bounded-h1-l2} continues to hold with appropriate modifications; all the details may be found in \cite{ArendtMazzeo}.

(3) In dimension $d=1$, i.e. for a bounded, non-degenerate interval, the Dirichlet-to-Neumann operator can be represented by the $2 \times 2$-matrix given by \eqref{Interval:DNOperatorMatrix}. Obviously, Lemma~\ref{lem:dno-bounded-h1-l2} continues to hold in this case.
\end{remark}


\begin{lemma}
\label{MeroDN}
The Dirichlet-to-Neumann operator $M(\lambda)$ is meromorphic with respect to the spectral parameter $\lambda\in\C$. Its singularities are poles of finite order and coincide with the eigenvalues of the corresponding Dirichlet Laplacian, i.e., the set of singularities of $\lambda\mapsto M(\lambda)$ is $\sigma (-\Delta_\Omega^D)$.
For $\lambda \in\rho(-\Delta_\Omega^D)$, $M(\lambda)$, is a self-adjoint holomorphic operator family and the corresponding quadratic forms are holomorphic of type (a). 
\end{lemma}

\begin{proof}
Suppose $\lambda\in\rho(-\Delta_\Omega^D)$. Then the Poisson operator $P(\lambda)$ given by \eqref{eq:poisson} and its adjoint $P^\ast(\lambda): H_0^{-1}(\Omega)\rightarrow H^{-1/2}(\partial\Omega)$ are well defined. We now invoke a perturbation formula for $M(\lambda)$ in terms of the fixed operator $M(0)$ (see \cite[Lemma~2.4 for $\mu=0$]{BehrndtRohleder}):
\begin{equation}\label{eq:pertformula}
	M(\lambda)
	=M(0)+\lambda P(0)^\ast\left(I+\lambda(-\Delta_\Omega^D - \lambda I)^{-1}\right)P(0).
\end{equation}
This expression depends polynomially on $\lambda$ and on the resolvent of the Dirichlet Laplacian which is known to be a meromorphic function with poles of finite order, as follows from Theorem~\ref{thm:dirichlet-and-neumann}. 
This proves that $M(\lambda)$, $\lambda \in\rho (-\Delta_\Omega^D)$, is a holomorphic operator family.
 It is self-adjoint holomorphic, i.e.\ $M(\overline\lambda)=(M(\lambda))^*$, by \eqref{eq:pertformula} and using that $M(0)$ is self-adjoint and that $\rho (-\Delta_\Omega^D)$ is symmetric about the real axis.
The corresponding quadratic forms $q_{\lambda}$  are holomorphic of type (a)  (see \cite[Section~VII.4.2]{Kato}), 
where the sectoriality was proved in Lemma~\ref{lem:dno-bounded-h1-l2}.
\end{proof}

\begin{remark}
One can show by exactly the same argument as in the proof of Theorem~\ref{thm:holo} that the corresponding eigenprojections can be chosen to depend holomorphically on $\lambda \in \rho (-\Delta_\Omega^D)$.
\end{remark}

We can now state the following duality result linking the eigenvalues of the operators $M(\lambda)$ and $-\Delta_\Omega^\alpha$. While this is standard, and in the case of real $\alpha$ and $\lambda$ well known (see \cite[Theorem~3.1]{ArendtMazzeo}), for completeness' sake we give a proof.
In fact, the connection between elliptic differential operators and operators of Dirichlet-to-Neumann type, or so-called Titchmarsh--Weyl $M$-functions, is also known in the non-selfadjoint case (see~\cite[Theorem~4.10]{BMNW}) but here we give a direct proof including the eigenfunctions.

\begin{theorem}
\label{thm:robin-dn-duality}
Let $\Omega \subset \R^d$, $d\geq 2$, be a bounded Lipschitz domain. For any $\alpha \in \C$ and any $\lambda \in \rho (-\Delta_\Omega^D)$, we have that $\lambda \in \sigma (-\Delta_\Omega^\alpha)$ if and only if $\alpha \in \sigma (M(\lambda))$. Moreover, $u$ is an eigenfunction of $-\Delta_\Omega^\alpha$ corresponding to $\lambda$ if and only if $\tr u$ is an eigenfunction of $M(\lambda)$ corresponding to $\alpha$.
\end{theorem}

\begin{proof}
Note first that the spectra of $M(\lambda)$ and $-\Delta_\Omega^\alpha$ consist only of eigenvalues of finite multiplicity. Now $\lambda \in \sigma (-\Delta_\Omega^\alpha)$ for given $\alpha \in \C$ with eigenfunction $u \in H^1(\Omega)$ if and only if
\begin{equation}
\label{eq:robin-ev-eqn-weak-form}
	a_\alpha [u,v] = \int_\Omega \nabla u \cdot \overline{\nabla v}\,\dx+\int_{\partial\Omega} \alpha u\overline{v}\,\dsigma
	= \lambda \int_\Omega u\overline{v}\,\dx
\end{equation}
for all $v \in H^1(\Omega)$, while $\alpha \in \sigma (M(\lambda))$ for given $\lambda \in \rho(-\Delta_\Omega^D)$ with eigenfunction $g \in H^{1/2}(\partial\Omega)$ if and only if
\begin{equation*}
	q_\lambda [g,h] = \int_\Omega \nabla P(\lambda)g \cdot \overline{\nabla P(\lambda)h} - \lambda P(\lambda)g\,\overline{P(\lambda)h}\,\dx
	= -\alpha \int_{\partial\Omega} g\overline{h}\,\dsigma
\end{equation*}
for all $h \in H^{1/2}(\partial\Omega)$. Using the fact that $P(\lambda)g$ satisfies \eqref{eq:weak-dirichlet} together with the direct sum decomposition of Lemma~\ref{DNop:UniqueSolutionOfDirichletProblem}, it follows that the eigenfunction $g$ satisfies
\begin{equation}
\label{eq:dn-ev-weak-form}
	\int_\Omega \nabla P(\lambda)g \cdot \overline{\nabla v} - \lambda P(\lambda)g\,\overline{v}\,\dx
	= -\alpha \int_{\partial\Omega} g\,\overline{\tr v}\,\dsigma
\end{equation}
for all $v \in H^1(\Omega)$. Comparing \eqref{eq:robin-ev-eqn-weak-form} and \eqref{eq:dn-ev-weak-form} leads immediately to the statement $\lambda \in \sigma (-\Delta_\Omega^\alpha)$ if and only if $\alpha \in \sigma (M(\lambda))$ (as long as $\lambda \in \rho (-\Delta_\Omega^D)$), with $g = \tr u$, or, equivalently, $u = P(\lambda)g$.

\end{proof}

\begin{remark}
\label{rem:jordan-chains}
A corresponding statement holds for any generalised eigenfunctions, as shown very recently in \cite{BehrndTerElst-Jordan}. Indeed, suppose $\{u_0,u_1,\dots,u_m\}$ and $\{\phi_0,\phi_1,\dots,\phi_n\}$ are Jordan chains of the operators $-\Delta_\Omega^\alpha$ and $M(\lambda)$, respectively, that is, the function $u_0$ is an eigenfunction of $-\Delta_\Omega^\alpha$ corresponding to an eigenvalue $\lambda$ and $u_1,\dots,u_m$ are generalised eigenfunctions for $\lambda$, and that the same holds for $\phi_0,\phi_1,\dots,\phi_n$ and $M(\lambda)$ with respect to the eigenvalue $\alpha$. Then we have $n=m$ and the Jordan chains are characterised by $\tr u_k=\phi_k$ for $k=0,\dots,n$ \cite[Theorem 4.1]{BehrndTerElst-Jordan}. (Note that this theorem is proved for more general Schr\"odinger operators with complex Robin boundary conditions.)
\end{remark}

\begin{proof}[Proof of Theorem~\ref{thm:Eigencurve_EitherConvDir_Div}]
By Lemma~\ref{MeroDN} the Dirichlet-to-Neumann operator $M(\lambda)$ is a meromorphic operator family whose set of singularities consists of poles of finite order and coincides with the spectrum $\sigma(-\Delta_\Omega^D)$ of the corresponding Dirichlet Laplacian. Now let $(\alpha_k)_k$ be any complex sequence with $\alpha_k\to\infty$ as $k\to\infty$. Assume that the eigenvalues $\lambda_k:=\lambda(\alpha_k)$ on a common analytic branch remain bounded as $k\to\infty$; without loss of generality we may suppose that $\lambda_k \to \lambda_0 \in \C$ as $k \to \infty$. Then by Theorem~\ref{thm:robin-dn-duality}, for each $k$ we may write $\alpha_k = \alpha(\lambda_k)$ for the Dirichlet-to-Neumann eigenvalues, which likewise belong to a common analytic branch. For this branch we have $\alpha_k \to \infty$ as $\lambda_k \to \lambda_0$. By definition, this means that $\lambda_0$ must be a singularity of the operator family $M(\lambda)$. The only possibility is that $\lambda_0 \in \sigma(-\Delta_\Omega^D)$.
\end{proof}

\section{The points of accumulation of the Robin eigenvalues}
\label{sec:limit-points}

In this section we study the question of which values $\lambda \in \C$ can be reached as points of accumulation of the eigenvalues of $-\Delta_\Omega^\alpha$ as $\alpha \to \infty$, also in dependence on how $\alpha \to \infty$ in $\C$; our principal aim is to prove Theorem~\ref{thm:EitherConvDir_Div}. As mentioned in the introduction, the Dirichlet-to-Neumann operator will be used in the proof, more precisely of part (2). For (1), we will draw on some ideas similar to the ones of \cite{CCH} for the case of real negative $\alpha \to -\infty$; in particular, the following lemma, which we will use repeatedly, recalls \cite[Lemma~2.1]{CCH}. Throughout this section we suppose $\Omega \subset \R^d$ to be a fixed bounded Lipschitz domain; and for $A\subseteq\C$ the set of points of accumulation of $A$ is denoted by $\myacc (A)$.

\begin{lemma}
\label{lem:convergence-h1-criterion}
Let $(\alpha_k)_{k\in\N} \subset \C$ be any divergent sequence in $\C$ and for each $k\in \N$ select a Robin eigenvalue $\lambda_k := \lambda (\alpha_k) \in \sigma(-\Delta_\Omega^{\alpha_k})$ (we do \emph{not} require the $\lambda_k$ to belong to the same analytic eigenvalue curve). Suppose that:
\begin{enumerate}
\item[(i)] the sequence $(\lambda_k)_{k \in \N}$ is bounded, and
\item[(ii)] for each $k \in \N$ there exists an associated eigenfunction $\psi_k$ with $L^2(\Omega)$-norm $1$, such that the sequence $\{\|\psi_k\|_{H^1(\Omega)}\}_{k\in\N}$ of $H^1(\Omega)$-norms is bounded.
\end{enumerate}
Then
\begin{displaymath}
	\myacc \{ \lambda_k : k \in \N \} \subseteq \sigma (-\Delta_\Omega^D).
\end{displaymath}
Moreover, if up to a subsequence $\lambda_k \to \lambda \in \sigma (-\Delta_\Omega^D)$, then up to a further subsequence the $\psi_k$ converge weakly in $H^1(\Omega)$ to a Dirichlet eigenfunction associated with $\lambda$.
\end{lemma}

\begin{proof}
Let $\lambda$ be any point of accumulation; without loss of generality we suppose that $\lim\limits_{k\to\infty} \lambda_k = \lambda$. We first claim that
\begin{displaymath}
	\int_{\partial\Omega} |\psi_k|^2\,\dsigma \to 0;
\end{displaymath}
in fact, this follows since
\begin{displaymath}
	\int_{\partial\Omega} |\psi_k|^2\,\dsigma = \frac{1}{\alpha_k}\left[\lambda_k - \int_\Omega |\nabla \psi_k|^2\,\dx\right]
\end{displaymath}
for $\alpha_k\neq 0$, and by assumption the $\lambda_k$ and $\int_\Omega |\nabla \psi_k|^2\,\dx$ are bounded. Next, since the $\psi_k$ are bounded in $H^1(\Omega)$, up to a subsequence there exists $\psi \in H^1(\Omega)$ such that $\psi_k \rightharpoonup \psi$ weakly in $H^1(\Omega)$. The convergence is strong in $L^2 (\Omega)$ by compactness of the embedding $H^1(\Omega) \hookrightarrow L^2(\Omega)$, while the traces converge in $L^2(\partial\Omega)$ by compactness of the trace mapping. In particular, $\psi$ has zero trace, so $\psi \in H^1_0 (\Omega)$. Finally, using the eigenvalue equation for $\lambda_k$ and the weak $H^1$-convergence of the $\psi_k$, for all $\varphi \in H^1_0 (\Omega)$ (whose traces vanish) we have
\begin{displaymath}
	\int_\Omega \nabla \psi \cdot \overline{\nabla \varphi}\,\dx 
=\lim_{k\to\infty} \int_\Omega \nabla \psi_k \cdot \overline{\nabla \varphi}\,\dx
	= \lim_{k\to\infty}\lambda_k \int_\Omega \psi_k \overline{\varphi}\,\dx 
=\lambda \int_\Omega \psi \overline{\varphi}\,\dx.
\end{displaymath}
Since $\psi \in H^1_0 (\Omega)$, this says exactly that $\lambda$ is an eigenvalue, and $\psi$ a corresponding eigenfunction, of the Dirichlet Laplacian.
\end{proof}

We can now give the proof of Theorem~\ref{thm:EitherConvDir_Div}(1); we will in fact prove the following slightly more precise version, which also allows us to conclude convergence of the eigenfunctions. As before, we do not require our eigenvalues to belong to the same analytic curve. Finally, we recall that the sector $T_\theta^-$ was introduced in Definition~\ref{def:sectors}.

\begin{theorem}
\label{thm:robin-accum-sector}
Let $(\alpha_k)_{k \in \N}$ be any divergent sequence in the sector $\C \setminus T_\theta^-$ for some $\theta>0$, and for each $k \in \N$ let $\lambda_k := \lambda (\alpha_k) \in \sigma(-\Delta_\Omega^{\alpha_k})$ be any corresponding eigenvalue. Then
\begin{displaymath}
	\myacc \{ \lambda_k : k \in \N \} \subseteq \sigma (-\Delta_\Omega^D).
\end{displaymath}
Moreover, if up to a subsequence $\lambda_k \to \lambda \in \sigma (-\Delta_\Omega^D)$, then there exist eigenfunctions for $\lambda_k$ which, possibly up to a further subsequence, converge weakly to a Dirichlet eigenfunction for $\lambda$.
\end{theorem}

\begin{proof}[Proof of Theorem~\ref{thm:robin-accum-sector} and hence of Theorem~\ref{thm:EitherConvDir_Div}(1)]
It suffices to show that under the stated assumptions if the $\lambda_k$ (or any subsequence thereof) are bounded, then they always admit eigenfunctions $\psi_k$ which are bounded in $H^1(\Omega)$ under the normalisation $\|\psi_k\|_{2}=1$, since we may then directly apply Lemma~\ref{lem:convergence-h1-criterion} in order to obtain the conclusion of the theorem. To this end, first assume without loss of generality that the sequence $\lambda_k$ actually converges to some $\lambda \in \C$; we distinguish between two possibilities, which together completely cover the sector $\C \setminus T_\theta^-$: (i) $\Re \alpha_k \geq 0$ for all $k \in \N$; (ii) $\left| \frac{\Re \alpha_k}{\Im \alpha_k} \right|$ remains bounded.

In case (i), we suppose that, for each $\lambda_k$, $\psi_k$ is any associated eigenfunction such that $\|\psi_k\|_{2} = 1$ and simply observe that
\begin{displaymath}
 \int_\Omega |\nabla \psi_k|^2\,\dx + \Re \alpha_k \int_{\partial\Omega}|\psi_k|^2\,\dsigma =	\Re \lambda_k \to \Re \lambda.
\end{displaymath}
Since $\Re \alpha_k \geq 0$, this is only possible if the sequence $(\|\nabla \psi_k\|_{2)}^2)_{k\in\N}$ remains bounded, which in turn means that the $\psi_k$ form a bounded sequence in $H^1(\Omega)$.

For (ii), we let the $\psi_k$ be as before but now consider the imaginary parts: we have that
\begin{displaymath}
	 \Im \alpha_k \int_{\partial\Omega} |\psi_k|^2\,\dsigma =\Im \lambda_k \to \Im \lambda;
\end{displaymath}
by assumption on the $\alpha_k$, this means that
\begin{displaymath}
	\Re \alpha_k \int_{\partial\Omega} |\psi_k|^2\,\dsigma \quad\text{ and hence also }\quad \alpha_k \int_{\partial\Omega} |\psi_k|^2\,\dsigma
\end{displaymath}
must in particular remain bounded as $k \to \infty$. Since $\lambda_k$ was also assumed bounded, we conclude that
\begin{displaymath}
	\int_\Omega |\nabla \psi_k|^2\,\dx = \lambda_k - \alpha_k \int_{\partial\Omega} |\psi_k|^2\,\dsigma
\end{displaymath}
likewise remains bounded (recall that $\|\psi_k\|_{2} = 1$), meaning that the $\psi_k$ are bounded in $H^1(\Omega)$.
\end{proof}

We next turn to the proof of Theorem~\ref{thm:EitherConvDir_Div}(2). This is in fact an immediate application of the fact that the Dirichlet-to-Neumann operator is unbounded.

\begin{proof}[Proof of Theorem~\ref{thm:EitherConvDir_Div}(2)]
First suppose that $\lambda \not\in \sigma (-\Delta_\Omega^D)$ and let $M(\lambda)$ be the Dirichlet-to-Neumann operator introduced in Section~\ref{sec:dno} (see \eqref{eq:dno-definition}). Then by Lemma~\ref{lem:dno-bounded-h1-l2}, $M(\lambda)$ admits a sequence of eigenvalues $\alpha_k \in \C$ such that $|\alpha_k| \to \infty$. By Theorem~\ref{thm:robin-dn-duality}, for each such $\alpha_k \in \C$, we have that $\lambda \in \sigma (-\Delta_\Omega^{\alpha_k})$.

For $\lambda \in \sigma (-\Delta_\Omega^D)$ the argument is the same except that $M(\lambda)$ becomes a multivalued operator; see Remark~\ref{rem:dno-various}(2).
\end{proof}

\begin{remark}
\label{rem:limit-points-various}
We draw explicit attention to the marked contrast between parts (1) and (2) of Theorem~\ref{thm:EitherConvDir_Div}: on the one hand, for $\alpha$ diverging away from the negative real semiaxis (more precisely outside the sector $T_\theta^-$ for arbitrarily small $\theta>0$), all eigenvalues either diverge absolutely or converge to points in the Dirichlet spectrum. Moreover, this is not just true of the individual analytic curves of eigenvalues but for any arbitrary sequence of eigenvalues in this region.

On the other hand, for \emph{any} $\lambda \in \C$ we can find an infinite sequence of parameters $\alpha_k$, which must end up ``close'' to the negative real semi-axis, for which $\lambda$ is a Robin eigenvalue (this is where the sufficiently large eigenvalues of the Dirichlet-to-Neumann operator $M(\lambda)$ are to be found, for any $\lambda$). Thus the whole of $\C$ can be obtained as points of accumulation if we place no restriction on $\alpha$. The reason why this is not inconsistent with Theorem~\ref{thm:Eigencurve_EitherConvDir_Div} is that there we are interested in the behaviour of the \emph{analytic curves of eigenvalues} (rather than sequences of $\alpha_k$ which may be drawn from different analytic curves).
\end{remark}

\section{Higher-dimensional examples: the hyperrectangle and the ball}
\label{sec:examples}

\subsection{The interval revisited}
\label{subsec:ExampleInterval}

We first return to the one-dimensional case sketched in Section~\ref{sec:interval}. We start by giving the remaining details, in particular the proofs of Theorems~\ref{Interval:thm:ConvergenceDirichletSpectrum} and~\ref{thm:DivergingEigenvaluesInterval}. We will later use results to discuss the eigenvalue asymptotics for higher dimensional rectangles, in Section~\ref{subsec:hyperrectangles}.


\subsubsection{On the Dirichlet-to-Neumann matrix}
	\label{subsubsec:dn-matrix}

The general solution $u$ of the Dirichlet problem on the open interval $\Omega=(-a,a)$ for $a>0$ \eqref{Interval:DirichletProblem} with Dirichlet data $g=(g_1,g_2)^T\in\mathbb{C}^2$ is given by
\begin{equation}\label{eq:1-dim-EF-cos-sin}
	u(x)=C_+\cos(\sqrt{\lambda} x)+ C_-\sin(\sqrt{\lambda} x).
\end{equation}
The coefficients $C_+$ and $C_-$ are given by
\begin{equation}
\label{Interval:GeneralSolution_coefficients}
	C_+:=\frac{ g_2 + g_1}{2\cos(\sqrt{\lambda} a)}\qquad\text{and}\qquad C_-:=\frac{g_2- g_1}{2\sin(\sqrt{\lambda} a)},
\end{equation}
where $C_+=0$ if $u$ is odd and $C_-=0$ if $u$ is even. The normal derivatives of $u$ read
\begin{displaymath}
	-u'(-a) =\sqrt{\lambda}\left(-\frac{ g_2+ g_1}{2}\tan \sqrt{\lambda} a -\frac{ g_2- g_1}{2}\cot \sqrt{\lambda} a \right),
\end{displaymath}
and similarly
\begin{displaymath}
	u'(+a) =\sqrt{\lambda}\left(-\frac{ g_2+ g_1}{2}\tan \sqrt{\lambda} a +\frac{ g_2- g_1}{2}\cot \sqrt{\lambda} a \right),
\end{displaymath}
meaning that the $\mathbb{C}^{2\times 2}$-valued Dirichlet-to-Neumann operator is given by
\begin{equation}
\label{Interval:DNOperator}
	M(\lambda)=\frac{\sqrt{\lambda}}{2}	
	\begin{pmatrix}
	\tan \sqrt{\lambda} a -\cot \sqrt{\lambda} a & \tan \sqrt{\lambda} a+\cot \sqrt{\lambda} a \\
	\tan \sqrt{\lambda} a  +\cot \sqrt{\lambda} a & \tan \sqrt{\lambda} a-\cot \sqrt{\lambda} a
	\end{pmatrix}.
\end{equation}
The expression \eqref{Interval:DNOperatorMatrix} for $M(\lambda)$ used in Section~\ref{sec:interval} then follows from the identities
\begin{align*}
\tan z-\cot z = -2\cot 2z\qquad\text{and}\qquad\tan z+\cot z = 2\csc 2z
\end{align*}
for $z\neq 0$.

\subsubsection{The convergent eigenvalues}
	\label{subsub:convergentEigenvalues}

We begin with the easy case of the eigenvalues which remain bounded as $\alpha \to \infty$.

\begin{proof}[Proof of Theorem \ref{Interval:thm:ConvergenceDirichletSpectrum}]
The poles of $\cot$ and $\csc$ are of order one, and thus so are the poles of the meromorphic Dirichlet-to-Neumann operator $M(\lambda)$ given by \eqref{Interval:DNOperatorMatrix}. 
If $\lambda(\alpha)$ remains bounded as $\alpha\rightarrow\infty$ in $\C$, the only possibility for this behaviour is that $\sqrt{\lambda}$ approaches one of said poles, namely $\sqrt{\lambda}\rightarrow \pi j/(2a)$ for any $j\in\mathbb{Z}$, that is $\lambda\rightarrow \pi^2 j^2/(4a^2)$ as $\alpha\rightarrow\infty$.
\end{proof}

\begin{remark}
In principle, one could derive additional terms in the asymptotic expansion of $\lambda(\alpha)$ as $\alpha \to \infty$, in powers of $\alpha^{-1}$; let us sketch briefly how one might get further information. The poles being the eigenvalues of the Dirichlet Laplacian allows us to obtain a partial fraction decomposition $M(\lambda)=A_j/(\sqrt{\lambda}-\sqrt{\lambda_j})+G_j(\sqrt{\lambda})$ for a matrix-valued function $G_j$ which is holomorphic (thus bounded) in a neighbourhood of $\sqrt{\lambda_j}$, and matrices $A_j$. Calculating the residues $\pm\pi j/(2a^2)$ of the on- and off-diagonal components of $M(\lambda)$, we can write down $A_j$ explicitly, which, together with the bounded $G_j$ terms, may yield a more detailed statement.
\end{remark}

\subsubsection{Divergence away from the positive real axis}

We next consider the case of the eigenvalues $\lambda$ diverging to $\infty$ away from the positive real axis, that is, for which $\Im \sqrt{\lambda} \to \pm\infty$. We recall from Section~\ref{sec:interval} that then the matrix $M(\lambda)$ has two divergent eigenvalues $\alpha$ whose squares both behave like
\begin{equation}
\label{Interval:Asymptotics_ExampleChapter}
	\alpha^2 = -\lambda + \mathcal{O}_\pm\left(\lambda\e^{\mp 2a\Im\sqrt{\lambda}}\right).
\end{equation}
This equation implies the anticipated asymptotic behaviour $\lambda\sim -\alpha^2$, however, we are also interested in the asymptotic remainder term. Consequently, our next goal is to invert this equation from $\alpha(\lambda)$ to obtain the asymptotic equation for $\lambda(\alpha)$ and thus the proof of Theorem~\ref{thm:DivergingEigenvaluesInterval}. 

We will restrict ourselves to the case $\Im \sqrt{\lambda} \to +\infty$. We first sketch the idea behind our inversion, namely an application of Rouch\'e's theorem, because we will also use this again in Section~\ref{subsec:balls} when considering the ball. 
Let $\tau\geq 0$ and let $h:\C\to\C$ be a continuous function such that $h(z)\to 0$ as $\Im z\rightarrow+\infty$. 
Suppose that $\alpha=\alpha(\lambda)$, as a holomorphic function of $\lambda$, satisfies the asymptotics 
\begin{equation}\label{eq:asy-with-remainder-before-2.4-proof}
\alpha(\lambda)=\I \sqrt{\lambda} +\tau +g(\sqrt{\lambda})
\end{equation}
as $\Im \sqrt{\lambda}\rightarrow +\infty$ for a certain error term $g(\sqrt{\lambda})$ which is $\mathcal{O}(h(\sqrt{\lambda}))$; for the choice of $h$ see \eqref{eq:IntervalDNAsymptotics} for the interval and \eqref{eq:asympt-alpha-of-lambda} for the ball. For given $\lambda$ and hence $\alpha = \alpha (\lambda)$, we define a new holomorphic function $f_{\alpha}$ by $f_{\alpha}(z):=\I z+\tau-\alpha$, whose only zero is given by $z_{\alpha}:=\I(\tau-\alpha)$. Then \eqref{eq:asy-with-remainder-before-2.4-proof} becomes
\begin{equation*}
f_\alpha(z) +g(z) = 0
\end{equation*}
if and only if $z=\sqrt{\lambda(\alpha)}$. Let $B^\alpha:=B_{r_\alpha}(z_{\alpha})$ be a ball with centre $z_{\alpha}$ and some given radius $r_\alpha>0$. Then, by Rouch\'{e}'s theorem, if $\alpha$ is sufficiently large and $|g|<|f_{\alpha}|$ on $\partial B^\alpha$, both $f_{\alpha}$ and $f_{\alpha}+g$ have exactly one zero in $B^\alpha$.

This technique proves not only the existence of an eigenvalue of the Robin Laplacian that satisfies said asymptotics, but gives an error term in the asymptotic expansion of $\lambda(\alpha)$ as follows.
By construction, for each $z\in\partial B^\alpha$ we have $ |f_{\alpha}(z)|=r_{\alpha}$. Moreover, $g=\mathcal{O}(h)$ as $\Im z\to +\infty$ implies the existence of some constant $\delta>0$, such that
\begin{equation}\label{ineq:RoucheInequality2}
|g(z)|\leq\delta |h(z)|
\end{equation}
on $\partial B^{\alpha}$ for all sufficiently large $\alpha$. For all such $\alpha$ we want $r_\alpha$ to satisfy 
\begin{equation}
\label{ineq:RoucheInequality} \delta\left|h(z)\right| <  r_\alpha, \quad z\in\partial B^{\alpha}.
\end{equation}
To ensure this inequality, the decay of $h$ is crucial: if it is too slow, then the method fails. 
This will be clarified in the following proof.

\begin{proof}[Proof of Theorem~\ref{thm:DivergingEigenvaluesInterval}]
The eigenvalues $\alpha_\pm$ of the Dirichlet-to-Neumann matrix \eqref{Interval:DNOperatorMatrix} read
\begin{equation}
\alpha_\pm = \sqrt{\lambda}\left(\pm\csc(2a\sqrt{\lambda})-\cot(2a\sqrt{\lambda})\right).
\end{equation}
As $\alpha\rightarrow\infty$ in $\C$ we have either $|\sqrt{\lambda}|\rightarrow\infty$ or $\sqrt{\lambda}$ is forced to approach a zero of $\sin(2a\,\cdot\,)$, which corresponds to a Dirichlet eigenvalue. The second case, in particular, requires $\lambda$ to remain bounded and thus is covered by Theorem~\ref{Interval:thm:ConvergenceDirichletSpectrum} (alternatively, one could adapt the proof of Theorem~\ref{thm:EitherConvDir_Div} to dimension $d=1$). We divide the proof into four steps:

\textit{Step 1:} We assume that for some given $\theta \in (0,\pi)$ some Robin eigenvalue $\lambda$ diverges to $\infty$ away from the real axis, inside the sector $S_\theta^+$ which, in particular, yields $\Im\sqrt{\lambda}\rightarrow +\infty$. In Section \ref{subsec:interval_divergent} we saw that for this behaviour of $\sqrt{\lambda}$ we obtain 
\begin{displaymath}\label{IntervalRev:AlphaAsympt}
	\alpha = +\I\sqrt{\lambda} + \mathcal{O} \left(\sqrt{\lambda}\e^{- 2a\Im\sqrt{\lambda}}\right)
\end{displaymath}
as $\Im\sqrt{\lambda}\rightarrow \infty$; for more details see \eqref{cotAsy}-\eqref{Interval:Asymptotic_alpha(lambda)}. (The other case $\Im\sqrt{\lambda}\rightarrow -\infty$ will be discussed in Step 3.)

\textit{Step 2:} It remains to invert this asymptotical behaviour by means of the Rouch\'e inversion technique sketched above and to show, based on the assumed asymptotic behaviour of $\alpha$, the existence of exactly two (see {Step 3 and 4}) divergent eigenvalues $\lambda$ which obey \eqref{eq:DivergingEigenvaluesInterval} away from the real axis. Here we deal with the inversion; as mentioned before, we will only consider the case $\Im\sqrt{\lambda}\to +\infty$ in detail. So let $\tau=0$, that is $f_\alpha(z_{\alpha})=0$ for $z_{\alpha}=-\I\alpha$. 
Here we take $h(z):=z \e^{-2a\Im z}$, which satisfies $h(z)\to 0$ as $\Im z\to +\infty$. By construction, every point $z\in\partial B^\alpha$ is represented by
\begin{equation}
\label{eq:sqrtlambda-polar-rep}
	z = z_{\alpha} +r_\alpha \e^{\I\phi} = -\I\alpha+r_\alpha \e^{\I\phi} 
\end{equation}
for some $\phi\in[0,2\pi)$. Our goal is to estimate $h$ as in \eqref{ineq:RoucheInequality}: a short calculation using \eqref{eq:sqrtlambda-polar-rep} gives
\begin{displaymath}
\begin{split}
	|h(z)| 
	=\left| z \e^{-2 a\Im z}  \right| 
	&=\left| \left(-\I\alpha+r_\alpha \e^{\I\phi} \right) \exp\left[-2 a\Im\left(-\I\alpha+r_\alpha \e^{\I\phi} \right)\right]  \right| \\
	&\leq  \left(|\alpha|+r_\alpha \right) 
		\exp\left(2 a r_\alpha \right)\exp\left(2 a\Re(\alpha) \right).
\end{split}
\end{displaymath}
We now choose $r_\alpha$ to ensure \eqref{ineq:RoucheInequality} on $\partial B^\alpha$. To this end, it suffices to find $r_{\alpha}>0$ such that
\begin{equation}\label{ineq:r-alpha}
 \delta \left(|\alpha|+r_\alpha \right)\e^{2 a\Re\alpha} < r_{\alpha} \e^{-2 a r_\alpha }
\end{equation}
for sufficiently large $\alpha$. To show this, we make the \emph{Ansatz} 
\begin{equation}\label{eq:r-alpha-ansatz}
r_{\alpha}= C |\alpha|\e^{2 a\Re\alpha}
\end{equation}
for a suitable constant $C>0$ (in fact we may take any $C>\delta$). Then, for such an $r_\alpha$, \eqref{ineq:r-alpha} is equivalent to
\begin{displaymath}
	\delta |\alpha| \e^{2a\Re\alpha}\left(1+C\e^{2a\Re\alpha}\right)\e^{2Ca|\alpha|\e^{2a\Re\alpha}} < C|\alpha|\e^{2a\Re\alpha},
\end{displaymath}
that is,
\begin{equation}\label{ineq:r-alpha-ansatz-equivalent-form}
	\delta \left(1 + C\e^{2a\Re\alpha}\right)\e^{2Ca|\alpha|\e^{2a\Re\alpha}} < C.
\end{equation}
Since $C\e^{2a\Re\alpha} \to 0$ and $\e^{2Ca|\alpha|\e^{2a\Re\alpha}} \to 1$ as $\Re\alpha\to -\infty$, the left-hand side of \eqref{ineq:r-alpha-ansatz-equivalent-form} converges to $\delta $ and hence \eqref{ineq:r-alpha} is satisfied whenever $\Re \alpha$ is sufficiently large negative, how large depending only on $a$, $\delta$ and $C$. In particular, for the \emph{Ansatz} \eqref{eq:r-alpha-ansatz}, the inequality \eqref{ineq:r-alpha} is then valid.


We arrive at
\begin{displaymath}
	\sqrt{\lambda} (\alpha) = -\I\alpha +\mathcal{O}\left( \alpha \e^{2 a\Re\alpha}\right) 
\end{displaymath}
as $\Re\alpha\rightarrow -\infty$, and thus
\begin{displaymath}
\lambda(\alpha) = -\alpha^2 +\mathcal{O}\left( \alpha^2 \e^{2 a\Re\alpha}\right).
\end{displaymath}

\textit{Step 3:} We remark briefly on the adaptation of the above proof to the assumption $\Im \sqrt{\lambda}\rightarrow -\infty$: one now chooses $f_\alpha(z)=-\I z-\alpha$ which vanishes only for $z_\alpha=+\I \alpha$. Similar calculations as above lead to
\begin{equation}
\label{Interval:Asymptotic_lambda(alpha)_negative}
	\sqrt{\lambda(\alpha)} = +\I\alpha +\mathcal{O}\left(\alpha\e^{2a\Re\alpha}\right)
\end{equation}
as $\Re\alpha\rightarrow -\infty$. 

\textit{Step 4 -- Conclusion:} We obtain that in both cases $\Im \sqrt{\lambda}\rightarrow \pm\infty$ the real part $\Re \alpha$ is always negative and divergent, that is, each divergent $\sqrt{\lambda}$ within a sector of the form $S_\theta^+$ or $S_\theta^-$ (note that $\sqrt{\lambda}\in S_\theta^+$ iff $-\sqrt{\lambda}\in S_\theta^-$) corresponds to $\sqrt{\lambda}\sim\I\alpha$ and $\sqrt{\lambda}\sim -\I\alpha$, respectively. We conclude that, under the assumption that $\alpha$ diverges in a sector $T_\varphi^-$ with $\varphi \in (0,\pi/2)$, then there are exactly two divergent eigenvalues $\lambda\in \C\setminus T_{2\theta}^+$, and these both satisfy $\lambda\sim -\alpha^2$ as $\Re\alpha\rightarrow -\infty$. Moreover, this implies that if $\Re\alpha$ remains bounded from below, then there are no divergent eigenvalues $\lambda\rightarrow\infty$ in $\C\setminus T_{\theta}^+$ for any $0<\theta<\pi/2$.
\end{proof}


\subsubsection{Divergence near the positive real axis}

Finally, we return to the case where $\lambda$ diverges near the positive real axis, say, in such a way that $\Im\lambda\neq 0$ (and hence $\Im \sqrt{\lambda}$) remains bounded. In this case the asymptotics of the Dirichlet-to-Neumann matrix is less obvious: each of its entries is meromorphic with poles on the real axis and those on the off-diagonal do not vanish asymptotically if $\Im\sqrt{\lambda}$ remains bounded. Since poles and zeros of the Dirichlet-to-Neumann operator $M(\lambda)$ correspond to Dirichlet and Neumann eigenvalues, respectively, which are discrete points on the real axis, the chosen path will pass arbitrarily closely to every single one of them. The question arises which associated paths of $\alpha$ in the complex plane correspond to such $\lambda$ paths.

\begin{figure}[h]
\centering
  \includegraphics[scale=0.33]{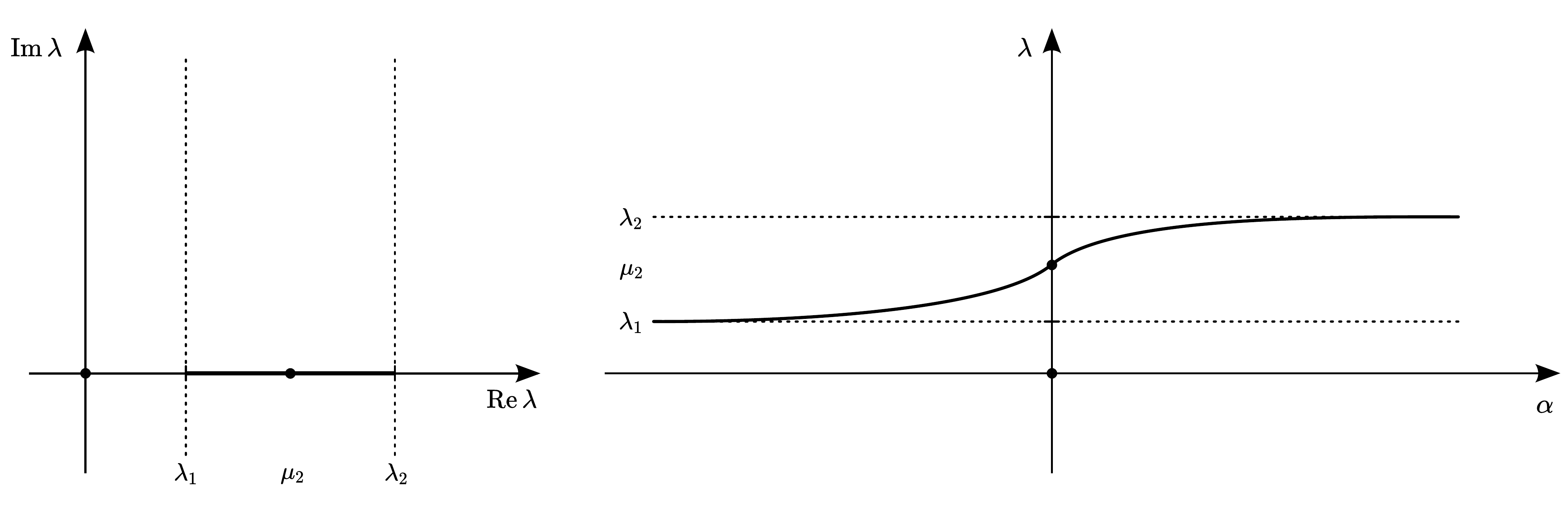}
  \caption{On the left a path in the $\lambda$-plane from one pole $\lambda_1$ of the Dirichlet-to-Neumann operator to the next one $\lambda_2$ while passing a zero (the second Neumann eigenvalue $\mu_2$). On the right-hand side the real curve $\lambda(\alpha)\in\mathbb{R}$ increasing from $\lambda_1$ to $\lambda_2$ as $\alpha\in\mathbb{R}$ tends from $-\infty$ to $+\infty$ \cite[Section 4.3]{BFK}.}
  \label{fig:ProblemCurve}
\end{figure}

It would appear that any such $\lambda$ path requires $\Re\alpha$ to be unbounded from below, cf.~Figure~\ref{fig:ProblemCurve}. The explicit form of $M(\lambda)$ \eqref{Interval:DNOperatorMatrix} allows us to calculate its two eigenvalues $\alpha_\pm$ explicitly, that is, for $z=2\sqrt{\lambda}a$,
\begin{equation}
	\label{eq:interval-DN-EV}
\alpha_\pm 
	=\sqrt{\lambda}\left(-\cot z\pm\csc z\right)\\
	=-\I\sqrt{\lambda}\left(\frac{\e^{2\I z}+1}{\e^{2\I z}-1}\mp \frac{2\e^{\I z}}{\e^{2\I z}-1}\right)
	 =\pm\I\sqrt{\lambda}\;\frac{\e^{\I z}\mp 1}{\e^{\I z}\pm 1}.
\end{equation}
From this equation it follows that $\alpha\sim\pm\I\sqrt{\lambda}$ as $\Im\sqrt{\lambda}\to\infty$ in $\C$ which is the same result as calculated in \eqref{eq:IntervalDNAsymptotics}. However, consider $\sqrt{\lambda}=\sqrt{\lambda}(\tau)$ following some (continuous) path described by $\sqrt{\lambda}=x(\tau)+\I y(\tau)$ for an unbounded function $x:[0,1)\to\R$ which diverges as $\tau\to 1$ and some bounded $y:[0,1)\to\R$. Firstly note that, since the imaginary part $y$ of $\sqrt{\lambda}$ remains bounded, for each such path and each sector $T_\theta^\pm$ there exists some $\tau_\theta\in[0,1)$ such that $\sqrt{\lambda}\in T_\theta^\pm$ for all $\tau_\theta < \tau < 1$; up to a possibly different $\theta>0$ so too does $\lambda$. Secondly, the boundedness of $\Im\lambda$ implies $y=\mathcal{O}(1/x)$. Without loss of generality, we consider $\alpha_+$; from
\begin{displaymath}
\alpha_+
	= (\I x-y)\frac{\e^{2\I ax}-\e^{2ay}}{\e^{2\I ax}+\e^{2ay}}
\end{displaymath}
by a somewhat tideous calculation we arrive at
\begin{align*}
\Re\alpha_+ 
	&=\frac{1}{2}\,\frac{y\left(1-\e^{4ay}\right)}{\cosh(2ay)-\cos(2ax)}-\frac{x\sin(2ax)}{\cosh(2ay)-\cos(2ax)}.
\end{align*}
Both denominators are bounded (and they can only vanish if $y=0$) and so is the numerator of the first quotient. The second numerator, however, diverges (indefinitely) as $x=\Re\sqrt\lambda\to\infty$. 

This proves Proposition~\ref{prop:strip}.

\subsection{Hyperrectangles}
\label{subsec:hyperrectangles}

Based on our understanding of the interval we can easily obtain results for $d$-dimensional \emph{hyperrectangles}, sometimes called cuboids or rectangular parallelepipeds. Fix the dimension $d\geq 2$, choose intervals $(-a_j,a_j)$ for $a_1,\dots,a_d>0$, and set
\begin{equation*}
	Q:=(-a_1,a_1)\times \dots\times (-a_d,a_d).
\end{equation*}
Denote the one-dimensional Robin Laplacian on the edge $e_j:=(-a_j,a_j)$ for $j=1,\dots,d$ by $\mathcal{A}_j$; then there exists a sequence of eigenvalues $\{\lambda_k(\alpha)\}_{k\geq 1}\subseteq \sigma(-\Delta_Q^\alpha) \subset \C$ of the Robin Laplacian $-\Delta_Q^\alpha$ on $Q$ such that each eigenvalue is given by a sum of eigenvalues of the constituent operators $\mathcal{A}_j$, that is
\begin{equation}
\label{eq:hyperrectangle-eigenvalue-representation}
\lambda_k(\alpha)=\sum_{j=1}^d \lambda^{(j)}(\alpha),
\end{equation}
where $\lambda^{(j)}\in\sigma(\mathcal{A}_j)$; indeed, each sum $\lambda^{(1)}+\dots +\lambda^{(d)}$ equals an eigenvalue of $-\Delta_Q^\alpha$.

\begin{remark}
\label{rem:hyperrectangle-riesz}
By Theorem \ref{thm:abel-basis} (2) we know that in $d=1$ dimension and for each $\alpha\in\C$ there exists a Riesz basis of $L^2((-a,a))$ consisting of the eigenfunctions of $-\Delta_{(-a,a)}^\alpha$. Now let $d\geq 2$; for the $d$-dimensional hyperrectangle $Q\subset\R^d$ it follows by separation of variables and using orthogonality of the coordinate axes that there exists a Riesz basis of $L^2(Q)$ consisting of products of the one-dimensional eigenfunctions of the operators $-\Delta_{(-a_j,a_j)}^\alpha$ for $j=1,\dots,d$, which correspond to the eigenvalues of the form \eqref{eq:hyperrectangle-eigenvalue-representation}. Now the basis property (in particular the linear independence) implies that all eigenvalues of $-\Delta_Q^\alpha$ are of the form in~\eqref{eq:hyperrectangle-eigenvalue-representation}.
\end{remark}

Theorems~\ref{Interval:thm:ConvergenceDirichletSpectrum} and~\ref{thm:DivergingEigenvaluesInterval} state that, if $\alpha\rightarrow\infty$ in a sector of the form $T_\phi^-$ for some $0<\phi<\pi/2$, on each $e_j$ there are two eigenvalues of $\mathcal{A}_j$, call them $\lambda_1^{(j)},\lambda_2^{(j)}$, which diverge like $-\alpha^2$.

Consequently, since a single diverging eigenvalue of $\mathcal{A}_j$ can be added to $d-1$ non-divergent eigenvalues on the remaining edges, and for each divergent one we have infinitely many choices, there are infinitely many divergent eigenvalues of $-\Delta_Q^\alpha$ which behave like $-2\alpha^2$. In the next step we choose two divergent eigenvalues $\lambda^{(1)}$ of $\mathcal{A}_1$ and $\lambda^{(2)}$ of $\mathcal{A}_2$ and $d-2$ non-divergent eigenvalues of $\mathcal{A}_3,\dots,\mathcal{A}_d$. Adding everything up we obtain infinitely many eigenvalues of $-\Delta_Q^\alpha$ behaving like $-2\alpha^2$. We proceed successively up to step $(d-1)$ to obtain infinitely many eigenvalues that behave like $-(d-1)\alpha^2$. However, the last step works differently: since there are two divergent eigenvalues for each $\mathcal{A}_j$, we obtain not infinitely many but $2^d$ possibilites for an eigenvalue of $-\Delta_Q^\alpha$ to satisfy the asymptotics $-d\alpha^2$. This results in the following theorem.

\begin{theorem}
\label{hyperrectangles:spectrum}
Let $Q\subset \mathbb{R}^d$, $d\geq 2$, be a hyperrectangle and suppose that $\alpha \to \infty$ in a sector of the form $T_\varphi^-$ for some $\varphi \in (0,\pi/2)$ (see Definition~\ref{def:sectors}). Then for each $j=1,\dots,d-1$ there are infinitely many divergent eigenvalues $\lambda(\alpha)$ of $-\Delta_Q^\alpha$ such that the leading term asymptotics reads $\lambda(\alpha)\sim -j\alpha^2$ and 
$2^d$ eigenvalues which behave like $\lambda(\alpha)\sim -d\alpha^2$.
\end{theorem}

\begin{remark}
As already mentioned in Section \ref{sec:introduction} there are several results on the eigenvalue asymptotics for domains with less regularity and real parameter $\alpha$. However, there are no results for general Lipschitz domains but only for those having a finite number of ``model corners''. Just like in the case of real $\alpha$ \cite{BFK} we expect that the asymptotics is mainly driven by the ``most acute'' corner(s) of the domain -- the sharper the corner(s), the larger the (negative) leading coefficient of the asymptotics.
\end{remark}

\subsection{$d$-dimensional balls}
\label{subsec:balls}

We next consider the model case of higher dimensional balls
\begin{equation*}
	\Omega=B=B_1(0)\subset\mathbb{R}^d
\end{equation*}
in dimension $d \geq 2$. We will use the notation $\partial B = \mathbb{S}^{d-1}$ interchangeably. 

\subsubsection{On spherical harmonics}

We briefly recall a few properties of the eigenvalues and eigenfunctions of the Laplace-Beltrami operator on $\mathbb{S}^{d-1}$, which will be useful in the sequel. Details can be found in \cite[Section 2.2]{DunklXu}.

\begin{definition}
For $l\in\mathbb{N}_0$ let $\mathcal{P}_l^\C(d)$ denote the space of all homogeneous polynomials of degree $l$ in $d$ variables with complex coefficients and define
\begin{enumerate}
	\item $\mathcal{P}_l^\C(\mathbb{S}^{d-1}):=\left\{P|_{\mathbb S^{d-1}}:P\in\mathcal P_l^\C(d)\right\},$
	\item $\mathcal{H}_l^\C(d):=\left\{ P\in\mathcal{P}_l^\C(d): \Delta P=0  \right\},$
	\item $\mathcal{H}_l^\C(\S^{d-1}):=\left\{P|_{\S^{d-1}}:P\in\mathcal \mathcal{H}_l^\C(d)\right\}$.
\end{enumerate}
\end{definition}

\begin{theorem}
Let $E_l$ be the complex eigenspaces of the Laplace-Beltrami operator $\Delta_\omega$ on $\S^{d-1}$. Then we have $E_l=\mathcal{H}_l^\C(\S^{d-1})$,
\begin{displaymath}
	M_l^d:=\dim E_l=\dim \mathcal{H}_l^\C(\S^{d-1}) 
	=	\begin{pmatrix}
			d+l-1\\ l-1
		\end{pmatrix}
	-
		\begin{pmatrix}
			d+l-3\\l-1
		\end{pmatrix},
\end{displaymath}
and
\begin{displaymath}
	L^2_\C(\S^{d-1})=\bigoplus_{l=0}^\infty E_l.
\end{displaymath}
\end{theorem}
The $E_l$ are eigenspaces of the Laplace-Beltrami operator and one can calculate that the corresponding eigenvalues $-\mu_l^2$ in $-\Delta_\omega Y_{l,j} = \mu_k^2 Y_{l,j}$ are given by $-\mu_l^2=-l(d+l-2)$.

\subsubsection{The Dirichlet Laplacian and Dirichlet-to-Neumann operator on balls}
\label{subsec:DirichletLapOnBalls}

Let $u$ solve the Dirichlet eigenvalue problem
\begin{displaymath}
	\begin{aligned}
	-\Delta u &= \lambda u\qquad &&\text{on } B, \\
	u &= g\qquad &&\text{on }\partial B,
	\end{aligned}
\end{displaymath}
for given $g \in L^2 (\partial B)$. Then we can write $u \in L^2 (B)$ and $g \in L^2 (\partial B)$ in their unique series representations
\begin{equation}
\label{Ball_DirichletProblem}
	u(r,\omega)=\sum_{l=0}^\infty \sum_{j=0}^{M_l^d} u_{l,j}(r)Y_{l,j}(\omega),\qquad\quad 
		g(\omega) =\sum_{l=0}^\infty \sum_{j=0}^{M_l^d} g_{l,j}Y_{l,j}(\omega)
\end{equation}
depending on the radius $0\leq r\leq 1$, the angles $\omega\in\mathbb{S}^{d-1}$, and the spherical harmonics $Y_{l,j}\in\mathcal{H}_l^\C(\S^{d-1})$. 
Let $\Delta_\omega$ be the Laplace-Beltrami operator on $\S^{d-1}\subset\mathbb{R}^d$. If we rewrite the Laplace operator in polar coordinates and take the boundary conditions $u_{l,j}(1)=g_{l,j}$ and $u_{l,j}(0)=\delta_{0l}u_{0,j}$ for $j=0,\dots,M_l^d$ and some bounded sequence $(u_{0,j})_{j\geq 0}$ into account, the corresponding Bessel differential equation
is solved by
\begin{equation}\label{ball:bessel-ansatz}
	u_{l,j}(r)=\frac{g_{l,j}}{J_{\frac{d}{2}+l-1}(\lambda )} \;r^{1-\frac{l}{2}} J_{\frac{d}{2}+l-1}(\lambda r).
\end{equation}
Using the identity 
\begin{equation}
\label{RecurrenceRelationDerivative}
	J_m'(\sqrt{\lambda})=\frac{m}{\lambda}J_m(\sqrt{\lambda})-J_{m+1}(\sqrt{\lambda})
\end{equation}
for all $m\in\mathbb{C}$, see \cite[Chap. XVII, 17.21 (B)]{WhittakerWatson}, by taking the normal derivative $\partial_r$ of \eqref{ball:bessel-ansatz} we obtain that for any given $l \in \N_0$ 
the Dirichlet-to-Neumann operator maps the Dirichlet data $g_{l,j}$ onto 
\begin{equation}\label{DirichletToNeumannforBalls}
 \sqrt{\lambda}\frac{J_{\frac{d}{2}+l}(\sqrt{\lambda})}{J_{\frac{d}{2}+l-1}(\sqrt{\lambda})}-l
\end{equation}
and we define $M^{(l)}(\lambda)$ by \eqref{DirichletToNeumannforBalls}. This part $M^{(l)}(\lambda)$ of the Dirichlet-to-Neumann operator on $\partial B$ in the subspace $\mathcal{H}_l^{\C}$, identified in the canonical way with $\C^{M_l^d}$ via the eigenfunctions of $\Delta_\omega|_{\mathcal{H}_l^{\C}}$, is representable by a $M_l^d \times M_l^d$ diagonal matrix each of whose diagonal entries is equal to \eqref{DirichletToNeumannforBalls}.
The Dirichlet-to-Neumann operator $M(\lambda)$ is then obtained by summing over all subspaces $\mathcal{H}_l^{\C}$, that is, it may be represented as a diagonal matrix. In particular, for each $l$, there  are exactly $M_l^d$ eigenvalues $\alpha$ of the Dirichlet-to-Neumann operator $M(\lambda)$ equal to $M^{(l)}(\lambda)$, and for our purposes it suffices to consider the $M^{(l)} (\lambda)$ individually. 
\begin{remark}\label{rem:ball-riesz}
Since each $M^{(l)}(\lambda)$ is diagonal and the Jordan chains of the Robin Laplacian and the corresponding Dirichlet-to-Neumann operator are of the same length, see Remark \ref{rem:jordan-chains}, root vectors and eigenfunctions coincide and the eigennilpotents are always zero.
\end{remark}

Observe that the zeros of the denominator in \eqref{DirichletToNeumannforBalls} are simple and so are the poles of the whole operator. The numerator does not cancel any of the poles, which follows from \eqref{RecurrenceRelationDerivative}.

If we assume that $J_m(\sqrt{\lambda_0})=0=J_{m+1}(\sqrt{\lambda_0})$ for some $m\in\mathbb{C}$ and some $\lambda_0\in\mathbb{C}\backslash\{0\}$ (we only need $m\in\frac{1}{2}\mathbb{N}_0$), then this implies $J_m'(\sqrt{\lambda_0})=0$ and $J_m$ has a zero of order 2, a contradiction.

It follows that $M^{(l)}$ is a meromorphic function having only simple, real poles and an essential singularity in $\pm\infty$. We have that 
\begin{equation*}
	M^{(l)}(\lambda) g=\alpha g\qquad\Rightarrow\qquad |\alpha| = \left| \sqrt{\lambda}\frac{J_{\frac{d}{2}+l}(\sqrt{\lambda})}{J_{\frac{d}{2}+l-1}(\sqrt{\lambda})}-l \right| <\infty
\end{equation*}
implies that for $|\alpha|\rightarrow\infty$ the right-hand side is forced to diverge as well. Using \eqref{DirichletToNeumannforBalls}, we are led via an explicit formula to the same dichotomy we saw in the general case in Theorem~\ref{thm:EitherConvDir_Div}. Namely, there are two possibilities: either $\lambda$ converges to the Dirichlet spectrum or diverges absolutely. In the latter case, as with the interval, we may further distinguish between eigenvalues $\lambda$ diverging away from the positive real axis or in the vicinity of it. This leads to the following three cases.
\begin{enumerate}
\item $\sqrt{\lambda}$ approaches a pole of $M^{(l)}(\lambda)$, i.e. a zero of $J_{\frac{d}{2}+l-1}$, meaning the eigenvalue $\lambda$ converges to some element of the Dirichlet spectrum;
\item $\lambda \to \infty$ in a sector of the form $\C \setminus T_{2\theta}^+$ for some small $\theta > 0$ (see Definition~\ref{def:sectors}, that is, Assumption~\ref{assumptionSector} holds. In this case, $\sqrt{\lambda}$ remains in $S_{\theta}^\pm$ and the quotient of the Bessel functions in the expression for $M^{(l)}$ remains bounded, see \eqref{eq:BesselExpansionOrderZero});
\item the more complicated case of divergence, where $\lambda \to \infty$ in a sector $T_{2\theta}^+$.
\end{enumerate}

We analyse the three cases separately.

\subsubsection{The convergent eigenvalues}
We start with the convergent eigenvalues; we are interested in establishing the rate of convergence. As we intimated for the interval, we may consider the residues of the Dirichlet-to-Neumann operator, which also in the case of balls can be reduced to a scalar problem. For $m\in\mathbb{R}$ and $p\in\mathbb{N}_0$, we denote the $p$th zero of the Bessel function $J_m$ of order $m$ by $j_{m,p}\in\mathbb{R}$.

\begin{theorem}
\label{lem:Ball:DirichletConvergenceRate}
Fix $l,p\in\mathbb{N}_0$. The eigenvalues $\lambda=\lambda(\alpha)$ converging to the Dirichlet spectrum satisfy
\begin{equation}
	\lambda(\alpha) = j_{\frac{d}{2}-l+1,p}^2 -\frac{2j_{\frac{d}{2}-l+1,p}^2}{\alpha} +\mathcal{O}\left(\frac{1}{\alpha^2}\right)
\end{equation}
as $|\alpha|\rightarrow\infty$.
\end{theorem}

\begin{proof}
The statement is proved by a simple calculation. Indeed, setting $m_l:=\frac{d}{2}-l+1$ and using \eqref{RecurrenceRelationDerivative}, one may show that
\begin{align*}
\Res_{j_{m_l,p}}\left( M^{(l)}\right)
	&=\lim_{\sqrt{\lambda}\rightarrow j_{m_l,p}} (\sqrt{\lambda}-j_{m_l,p})\left( \sqrt{\lambda}\frac{J_{m_l+1}(\sqrt{\lambda})}{J_{m_l}(\sqrt{\lambda})}-l\right) \\
	&=j_{m_l,p}J_{m_l+1}(j_{m_l,p})\; \lim_{\sqrt{\lambda}\rightarrow j_{m_l,p}} \left(\frac{J_{m_l}(\sqrt{\lambda})-J_{m_l}(j_{m_l,p})}{\sqrt{\lambda} -j_{m_l,p}}\right)^{-1} \\
	&=j_{m_l,p}J_{m_l+1}(j_{m_l,p})\; \left( J_{m_l}'(j_{m_l,p}) \right)^{-1} \\
	&=j_{m_l,p}J_{m_l+1}(j_{m_l,p})\; \left( \frac{m_l}{j_{m_l,p}} J_{m_l}(j_{m_l,p})-J_{m_l+1}(j_{m_l,p}) \right)^{-1}  \\
	&= -j_{m_l,p}
\end{align*}
for the $p$th pole. From
\begin{align*}
(\sqrt{\lambda}-j_{m_l,p})\alpha 
	 = (\sqrt{\lambda}-j_{m_l,p}) M^{(l)}(\lambda)
	&=\Res_{j_{m_l,p}}\left( M^{(l)}\right) +\mathcal{O}(\sqrt{\lambda}-j_{m_l,p})\\
	&=-j_{m_l,p} +\mathcal{O}(\sqrt{\lambda}-j_{m_l,p})
\end{align*} 
as $\sqrt{\lambda}\rightarrow j_{m_l,p}$, it follows that
\begin{align*}
\sqrt{\lambda(\alpha)} = j_{m_l,p}-\frac{j_{m_l,p}}{\alpha}+\mathcal{O}\left(\frac{1}{\alpha^2}\right)
\end{align*}
and hence for the eigenvalue $\lambda$
\begin{align*}
\lambda(\alpha) = j_{m_l,p}^2 -2\frac{j_{m_l,p}^2}{\alpha} +\mathcal{O}\left(\frac{1}{\alpha^2}\right).
\end{align*}
\end{proof}

\subsubsection{Divergence away from the positive real axis}
\label{subsub:existenceDiv}

We next study those divergent eigenvalues $\lambda$ which remain away from the positive real axis, that is, we now take Assumption~\ref{assumptionSector}. (Here, unlike for the interval, the assumption of divergence in a sector, that is, that $\Im\sqrt{\lambda}$ grows sufficiently rapidly compared with $\Re\sqrt{\lambda}$, will turn out to be important.) We first need an asymptotic expansion of $M^{(l)}$ for large $\lambda$, which in turn requires knowledge of the asymptotics of the Bessel functions appearing in \eqref{DirichletToNeumannforBalls}. To this end, let $H_m^{(1)}, H_m^{(2)}$ be the \textit{Hankel functions} of the first and second kind, that is,
\begin{equation*}
2 J_m(\sqrt{\lambda})=H_m^{(1)}(\sqrt{\lambda})+H_m^{(2)}(\sqrt{\lambda}),
\end{equation*}
and set
\begin{align*}
P(m,\sqrt{\lambda})&= \sum_{l=0}^\infty (-1)^l \frac{\Gamma(m+2l+1/2)}{(2l)! \Gamma(m-2l+1/2)(2\sqrt{\lambda})^{2l}}\\
	&=1-\frac{(4m^2-1)(4m^2-9)}{2!(8\sqrt{\lambda})^2} \\
	&\qquad\qquad\qquad+\frac{(4m^2-1)(4m^2-9)(4m^2-25)(4m^2-49)}{4!(8\sqrt{\lambda})^4}-\dots, \\
Q(m,\sqrt{\lambda})&=\sum_{l=0}^\infty (-1)^l \frac{\Gamma(m+(2l+1)+1/2)}{(2l+1)! \Gamma(m-(2l+1)+1/2)(2\sqrt{\lambda})^{2l+1}} \\
	&=\frac{4m^2-1}{8\sqrt{\lambda}}-\frac{(4m^2-1)(4m^2-9)(4m^2-25)}{3!(8\sqrt{\lambda})^3}+\dots.
\end{align*}
It is known that (see \cite[9.2, p. 364]{Stegun})
\begin{align}
\begin{split}\label{eq:HankelAsymptotics}
H_m^{(1)}(\sqrt{\lambda})&=\sqrt{\frac{2}{\pi \sqrt{\lambda}}}\left( P(m,\sqrt{\lambda})+\I Q(m,\sqrt{\lambda})\right) \e^{\I \sqrt{\lambda}-\frac{\I\pi}{4}(2m+1)}\qquad (-\pi<\arg \sqrt{\lambda} <2\pi), \\
H_m^{(2)}(\sqrt{\lambda})&=\sqrt{\frac{2}{\pi \sqrt{\lambda}}}\left( P(m,\sqrt{\lambda})-\I Q(m,\sqrt{\lambda})\right) \e^{-\I \sqrt{\lambda}+\frac{\I\pi}{4}(2m+1)}\qquad (-2\pi<\arg \sqrt{\lambda} <\pi). 
\end{split}
\end{align}
Let $Q_l(m,\sqrt{\lambda}),P_l(m,\sqrt{\lambda})$ be the sums up to the $l$th summand. To obtain the order $1/\sqrt{\lambda}$, we only have to consider the first terms in the expansions of $P$ and $Q$ to obtain
\begin{align}
J_m(\sqrt{\lambda}) 
	&= \frac{1}{2}\left(H_m^{(1)}(\sqrt{\lambda})+H_m^{(2)}(\sqrt{\lambda})\right) \label{eq:BesselAsymptotics} \\
	&\sim \frac{1}{2}\sqrt{\frac{2}{\pi \sqrt{\lambda}}}
		\Big[
			\left( P_0(m,\sqrt{\lambda})+\I Q_0(m,\sqrt{\lambda})\right) \e^{\I \sqrt{\lambda}-\frac{\I\pi}{4}(2m+1)}  \nonumber  \\
		&\qquad\qquad\qquad\qquad\qquad\qquad+	\left(P_0(m,\sqrt{\lambda})-\I Q_0(m,\sqrt{\lambda})\right) \e^{-\I \sqrt{\lambda}+\frac{\I\pi}{4}(2m+1)}
		\Big] \nonumber  \\
		\begin{split}\label{eq:BesselExpansionOrderZero}
	&=
	 \sqrt{\frac{1}{2\pi \sqrt{\lambda}}}
		\Bigg[
			\left( 1+\I \frac{4m^2-1}{8\sqrt{\lambda}}\right) \e^{\I \sqrt{\lambda}-\frac{\I\pi}{4}(2m+1)}   \\
			&\qquad\qquad\qquad\qquad\qquad\qquad+	\left( 1-\I \frac{4m^2-1}{8\sqrt{\lambda}}\right) \e^{-\I \sqrt{\lambda}+\frac{\I\pi}{4}(2m+1)}
		\Bigg] 	.
		\end{split}
\end{align}
The non-polynomial terms $\e^{\I \sqrt{\lambda}}$ and $\e^{-\I \sqrt{\lambda}}$ of $H_m^{(1)}(\sqrt{\lambda})$ and $H_m^{(2)}(\sqrt{\lambda})$ yield exponential decrease and increase in $S_\theta^+$ and $S_\theta^-$, respectively. In a neighbourhood of the real axis, however, the remainder terms of the increasing expansion dominate the leading terms of the decreasing expansion on the other side of the real axis. This is why we want to add up both terms to obtain \cite[9.2.1, p. 364]{Stegun}, i.e.
\begin{equation}\label{eq:BesselExpansionOnReAxis}
J_m(\sqrt{\lambda}) =  \frac{1}{\sqrt{2\pi \sqrt{\lambda}}} \left(2\cos\left(\sqrt{\lambda}-\frac{\I\pi}{4}[2m+1]\right)+\e^{|\Im\sqrt{\lambda}|}\mathcal{O}\left(|\sqrt{\lambda}|^{-1}\right)\right)
\end{equation}
as $\sqrt{\lambda}\rightarrow\infty$ outside $T_\theta^-$ (in particular in $T_\theta^+$). This expansion outside $T_\theta^+$ (in particular in $T_\theta^-$) is obtained by point reflection of $J_m(\sqrt{\lambda})$ in zero. 

Considering the two cases where $\sqrt{\lambda} \to \infty$ in $S_\theta^-$ and $S_\theta^+$, separately, we arrive at
\begin{equation}
\label{asymptBehaviourBessel}
J_m(\sqrt{\lambda})=
\begin{cases}
 \frac{1}{\sqrt{2\pi \sqrt{\lambda}}}\left(1-\I \frac{4m^2-1}{8\sqrt{\lambda}}\right) \e^{-\I\sqrt{\lambda}} \e^{\frac{\I\pi}{4}(1+2m)}+\mathcal{O}\left(\frac{1}{\lambda}\right)  \\
 \frac{1}{\sqrt{2\pi \sqrt{\lambda}}}\left(1+\I \frac{4m^2-1}{8\sqrt{\lambda}}\right) \e^{+\I\sqrt{\lambda}} \e^{- \frac{\I\pi}{4}(1+2m)}+\mathcal{O}\left(\frac{1}{\lambda}\right) 
\end{cases}
\end{equation}
as $\sqrt{\lambda}\rightarrow\infty$ in $S_\theta^-$ and $S_\theta^+$, respectively. Using \eqref{asymptBehaviourBessel} and recalling the relation $m=\frac{d}{2}+l-1$, we may obtain after an elementary calculation that
\begin{equation}\label{eq:BesselQuotAsymptotics}
	\frac{J_{m+1}(\sqrt{\lambda})}{J_{m}(\sqrt{\lambda})}
	= \pm\I+\frac{d-1}{2\sqrt{\lambda}}+\frac{l}{\sqrt{\lambda}} +\mathcal{O}\left(\frac{1}{\lambda}\right)
\end{equation}
in $S_\theta^\pm$. 
Recalling \eqref{DirichletToNeumannforBalls}, this means that for each $l \in \N_0$ we obtain an $\alpha = \alpha (\lambda)$ with the behaviour
\begin{equation}
\label{eq:asympt-alpha-of-lambda}
	\alpha= M^{(l)}(\sqrt{\lambda}) = \mp\I\sqrt{\lambda}+\frac{d-1}{2}+\mathcal{O}\left(\frac{1}{\sqrt{\lambda}}\right)
\end{equation}
as $\sqrt{\lambda}\rightarrow\infty$ in $S_\theta^\pm$, respectively. This leads to the following theorem.

\begin{theorem}
\label{thm:ExistenceDiv}
Let $\Omega=B_1(0)\subseteq\mathbb{R}^d$, $d\geq 2$, and let $\alpha\rightarrow\infty$ in a sector of the form $T_\phi^-$ for any $0<\phi<\pi/2$. Then there are infinitely many Robin eigenvalues $\lambda(\alpha)$ such that 
\begin{equation}
\label{Balls:Asymptotics}
	\lambda(\alpha)= -\alpha^2 +\alpha(d-1)+\mathcal{O}(1)
\end{equation}
as $\alpha\rightarrow\infty$ in $T_\phi^-$.
\end{theorem}

\begin{proof}
\label{proof:existenceDiv}
We invert \eqref{eq:asympt-alpha-of-lambda} to obtain \eqref{Balls:Asymptotics} using Rouch\'e's theorem, as explained in Section~\ref{subsec:ExampleInterval}. First, let $\tau=(d-1)/2$ and let $\alpha\in\mathbb{C}$ be large in some non-trivial sector $T_\phi^-$ for any large $0<\phi<\pi/2$. Following the approach of the aforementioned section \ref{subsec:ExampleInterval} and restricting ourselved to the \textit{positive} case of $\Im\sqrt{\lambda}\to +\infty$, we have
\begin{equation}
\label{f(lambda)=0}
	f_\alpha(z):= \frac{d-1}{2}+\I z-\alpha \qquad\text{and}\qquad g(\sqrt{\lambda})= \mathcal{O}\left(h(\sqrt{\lambda})\right)
\end{equation} 
for the error term $h(\sqrt{\lambda})=1/\sqrt{\lambda}\to 0$ as $\Im\sqrt{\lambda}\to +\infty$. Then the unique zero $z_\alpha$ of $f_\alpha$ reads $z_\alpha=\I\left(\frac{d-1}{2}-\alpha\right)$. Let $r_\alpha>0$ be the radius of the Ball $B^\alpha=B_{r_\alpha}(z_\alpha)$. Then for $z\in B^\alpha$  and sufficiently large $\alpha$ we can estimate
\begin{align*}
|h(z)|
	\leq \frac{c}{\left| z_\alpha+\e^{\I\phi}\frac{p}{|\alpha|}\right|}
		=\frac{c}{\left|\I\frac{d-1}{2}-\I\alpha+\e^{\I\phi}\frac{p}{|\alpha|}  \right|}
		<\frac{c'}{|\alpha|}
\end{align*}
for some constants $c,c'>0$. We make the Ansatz $|f_\alpha(z)|=r_\alpha =C/|\alpha|$ for a suitable constant $C>0$. Consequently, since \eqref{ineq:RoucheInequality2} holds for some $\delta>0$ for sufficiently large $\alpha$, a similar calculation as in the proof of Theorem \ref{thm:DivergingEigenvaluesInterval} yields $|g(z)|< |f_\alpha(z)|$ on $\partial B_{\frac{C}{|\alpha|}}(z_\alpha)$. We end up with the existence of exactly one eigenvalue $\lambda$ which behaves like $\sqrt{\lambda(\alpha)} = -\I \alpha +\frac{d-1}{2}+ \mathcal{O}\left(1/|\alpha|\right)$, that is, together with the case $\Im\sqrt{\lambda}\to -\infty$, there are exactly two eigenvalues which behave like
\begin{equation*}
\lambda(\alpha) = -\alpha^2+ (d-1)\alpha +\mathcal{O}(1)
\end{equation*}
as $\Re\alpha\rightarrow -\infty$.
\end{proof} 

\begin{remark}
For any divergent eigenvalue curve $\lambda (\alpha)$ on the ball, there is a complete asymptotic expansion in powers of $\alpha$, and the above method might be used to obtain arbitrarily many terms of it. Indeed, the asymptotics of the Bessel functions (or more precisely the Hankel functions) provides us with everything needed to determine the asymptotics of the Dirichlet-to-Neumann operators $M^{(l)}(\lambda)$, and thus of $\alpha$ as $\Im\sqrt{\lambda}\to\pm \infty$: taking more terms in \eqref{eq:BesselAsymptotics} results in more terms in \eqref{asymptBehaviourBessel} and so too in \eqref{eq:asympt-alpha-of-lambda}, which can then again be inverted.
\end{remark}

\subsubsection{Divergence near the positive real axis}

We reconsider the asymptotics in \eqref{eq:BesselExpansionOnReAxis} and observe the oscillating nature of the cosine part as $\Re\sqrt{\lambda}$ increases -- the summand $\frac{\I\pi}{4}[2m+1]$ appearing in the argument is simply a phase shift. Suppose that $\Im\sqrt{\lambda}$ remains bounded as $\sqrt{\lambda}\rightarrow\infty$ in $T_\theta^+$, i.e. we explicitly do not apply Assumption~\ref{assumptionSector}. Then it appears that $J_m(\sqrt{\lambda})$ is dominated by $\lambda^{1/4}\cos(\sqrt{\lambda})$. However, the cosine has zeros on the real axis and thus, in a neighbourhood of said zeros, we need to consider further terms in the asymptotic expansion~\eqref{eq:BesselQuotAsymptotics}. 
This then leads to more complicated behaviour which might be adressed in a later work.

\newcommand{\cprime}{$'$}


\bibliography{bkl-biblio-new}
\bibliographystyle{acm}

\end{document}